\documentclass{birkjour}

\usepackage{amsmath}
\usepackage{mathtools}
\usepackage{amsfonts}
\usepackage[mathscr]{eucal}
\usepackage{amsthm}
\usepackage{accents}
\usepackage{amssymb}
\usepackage{enumerate}
\usepackage{subfigure}
\usepackage{cite}
 \usepackage{color}
\usepackage[show]{ed}

\DeclareMathOperator{\supp}{supp}

\newcommand{\meas}{m}
\usepackage{pgf,tikz}\usetikzlibrary{matrix, calc, arrows}
\usepackage[bookmarks=true, bookmarksopen=true, bookmarksopenlevel=4]{hyperref}
\definecolor{myblue}{rgb}{0,0,0.6}     %
\hypersetup{pdftitle={Sobolev spaces on subsets of Rn},
            pdfauthor={Chandler-Wilde, Hewett, Moiola},
     colorlinks=true, linkcolor=myblue,  citecolor=myblue, filecolor=myblue,   urlcolor=myblue,  }
\allowdisplaybreaks[4]
\graphicspath{{Figs/}}
\newcommand{\hookdoubleheadrightarrow}{%
  \hookrightarrow\mathrel{\mspace{-15mu}}\rightarrow
}
\begin{document}
\newcommand{\rf}[1]{(\ref{#1})}
\newcommand{\mmbox}[1]{\fbox{\ensuremath{\displaystyle{ #1 }}}}	
\newcommand{\hs}[1]{\hspace{#1mm}}
\newcommand{\vs}[1]{\vspace{#1mm}}
\newcommand{\ri}{{\mathrm{i}}}
\newcommand{\re}{{\mathrm{e}}}
\newcommand{\rd}{\mathrm{d}}
\newcommand{\R}{\mathbb{R}}
\newcommand{\Q}{\mathbb{Q}}
\newcommand{\N}{\mathbb{N}}
\newcommand{\Z}{\mathbb{Z}}
\newcommand{\C}{\mathbb{C}}
\newcommand{\K}{{\mathbb{K}}}
\newcommand{\cA}{\mathcal{A}}
\newcommand{\cB}{\mathcal{B}}
\newcommand{\cC}{\mathcal{C}}
\newcommand{\cS}{\mathcal{S}}
\newcommand{\cD}{\mathcal{D}}
\newcommand{\cH}{\mathcal{H}}
\newcommand{\cI}{\mathcal{I}}
\newcommand{\cItilde}{\tilde{\mathcal{I}}}
\newcommand{\cIhat}{\hat{\mathcal{I}}}
\newcommand{\cIcheck}{\check{\mathcal{I}}}
\newcommand{\cIstar}{{\mathcal{I}^*}}
\newcommand{\cJ}{\mathcal{J}}
\newcommand{\cM}{\mathcal{M}}
\newcommand{\cP}{\mathcal{P}}
\newcommand{\cV}{{\mathcal V}}
\newcommand{\cW}{{\mathcal W}}
\newcommand{\scrD}{\mathscr{D}}
\newcommand{\scrS}{\mathscr{S}}
\newcommand{\scrJ}{\mathscr{J}}
\newcommand{\sD}{\mathsf{D}}
\newcommand{\sN}{\mathsf{N}}
\newcommand{\sS}{\mathsf{S}}
\newcommand{\bs}[1]{\mathbf{#1}}
\newcommand{\bb}{\mathbf{b}}
\newcommand{\bd}{\mathbf{d}}
\newcommand{\bn}{\mathbf{n}}
\newcommand{\bp}{\mathbf{p}}
\newcommand{\bP}{\mathbf{P}}
\newcommand{\bv}{\mathbf{v}}
\newcommand{\bx}{\mathbf{x}}
\newcommand{\by}{\mathbf{y}}
\newcommand{\bz}{{\mathbf{z}}}
\newcommand{\bxi}{\boldsymbol{\xi}}
\newcommand{\boldeta}{\boldsymbol{\eta}}	
\newcommand{\ts}{\tilde{s}}
\newcommand{\tGamma}{{\tilde{\Gamma}}}
\newcommand{\done}[2]{\dfrac{d {#1}}{d {#2}}}
\newcommand{\donet}[2]{\frac{d {#1}}{d {#2}}}
\newcommand{\pdone}[2]{\dfrac{\partial {#1}}{\partial {#2}}}
\newcommand{\pdonet}[2]{\frac{\partial {#1}}{\partial {#2}}}
\newcommand{\pdonetext}[2]{\partial {#1}/\partial {#2}}
\newcommand{\pdtwo}[2]{\dfrac{\partial^2 {#1}}{\partial {#2}^2}}
\newcommand{\pdtwot}[2]{\frac{\partial^2 {#1}}{\partial {#2}^2}}
\newcommand{\pdtwomix}[3]{\dfrac{\partial^2 {#1}}{\partial {#2}\partial {#3}}}
\newcommand{\pdtwomixt}[3]{\frac{\partial^2 {#1}}{\partial {#2}\partial {#3}}}
\newcommand{\bnabla}{\boldsymbol{\nabla}}
\newcommand{\dive}{\boldsymbol{\nabla}\cdot}
\newcommand{\curl}{\boldsymbol{\nabla}\times}
\newcommand{\Phixy}{\Phi(\bx,\by)}
\newcommand{\PhiOxy}{\Phi_0(\bx,\by)}
\newcommand{\dxPhixy}{\pdone{\Phi}{n(\bx)}(\bx,\by)}
\newcommand{\dyPhixy}{\pdone{\Phi}{n(\by)}(\bx,\by)}
\newcommand{\dxPhiOxy}{\pdone{\Phi_0}{n(\bx)}(\bx,\by)}
\newcommand{\dyPhiOxy}{\pdone{\Phi_0}{n(\by)}(\bx,\by)}
\newcommand{\eps}{\varepsilon}
\newcommand{\real}[1]{{\rm Re}\left[#1\right]} 
\newcommand{\im}[1]{{\rm Im}\left[#1\right]}
\newcommand{\ol}[1]{\overline{#1}}
\newcommand{\ord}[1]{\mathcal{O}\left(#1\right)}
\newcommand{\oord}[1]{o\left(#1\right)}
\newcommand{\Ord}[1]{\Theta\left(#1\right)}
\newcommand{\hsnorm}[1]{||#1||_{H^{s}(\bs{R})}}
\newcommand{\hnorm}[1]{||#1||_{\tilde{H}^{-1/2}((0,1))}}
\newcommand{\norm}[2]{\left\|#1\right\|_{#2}}
\newcommand{\normt}[2]{\|#1\|_{#2}}
\newcommand{\on}[1]{\Vert{#1} \Vert_{1}}
\newcommand{\tn}[1]{\Vert{#1} \Vert_{2}}
\newcommand{\xt}{\mathbf{x},t}
\newcommand{\PhiF}{\Phi_{\rm freq}}
\newcommand{\cone}{{c_{j}^\pm}}
\newcommand{\ctwo}{{c_{2,j}^\pm}}
\newcommand{\cthree}{{c_{3,j}^\pm}}
\newtheorem{thm}{Theorem}[section]
\newtheorem{lem}[thm]{Lemma}
\newtheorem{defn}[thm]{Definition}
\newtheorem{prop}[thm]{Proposition}
\newtheorem{cor}[thm]{Corollary}
\newtheorem{rem}[thm]{Remark}
\newtheorem{conj}[thm]{Conjecture}
\newtheorem{ass}[thm]{Assumption}
\newtheorem{example}[thm]{Example} 
\newcommand{\tH}{\widetilde{H}}
\newcommand{\Hze}{H_{\rm ze}} 	
\newcommand{\uze}{u_{\rm ze}}		
\newcommand{\dimH}{{\rm dim_H}}
\newcommand{\dimB}{{\rm dim_B}}
\newcommand{\IntClosOm}{\mathrm{int}(\overline{\Omega})}
\newcommand{\IntClosOmOne}{\mathrm{int}(\overline{\Omega_1})}
\newcommand{\IntClosOmTwo}{\mathrm{int}(\overline{\Omega_2})}
\newcommand{\Ccomp}{C^{\rm comp}}
\newcommand{\tCcomp}{\tilde{C}^{\rm comp}}
\newcommand{\uC}{\underline{C}}
\newcommand{\utC}{\underline{\tilde{C}}}
\newcommand{\oC}{\overline{C}}
\newcommand{\otC}{\overline{\tilde{C}}}
\newcommand{\capcomp}{{\rm cap}^{\rm comp}}
\newcommand{\Capcomp}{{\rm Cap}^{\rm comp}}
\newcommand{\tcapcomp}{\widetilde{{\rm cap}}^{\rm comp}}
\newcommand{\tCapcomp}{\widetilde{{\rm Cap}}^{\rm comp}}
\newcommand{\hcapcomp}{\widehat{{\rm cap}}^{\rm comp}}
\newcommand{\hCapcomp}{\widehat{{\rm Cap}}^{\rm comp}}
\newcommand{\tcap}{\widetilde{{\rm cap}}}
\newcommand{\tCap}{\widetilde{{\rm Cap}}}
\newcommand{\ccap}{{\rm cap}}
\newcommand{\ucap}{\underline{\rm cap}}
\newcommand{\uCap}{\underline{\rm Cap}}
\newcommand{\cCap}{{\rm Cap}}
\newcommand{\ocap}{\overline{\rm cap}}
\newcommand{\oCap}{\overline{\rm Cap}}
\DeclareRobustCommand
{\mathringbig}[1]{\accentset{\smash{\raisebox{-0.1ex}{$\scriptstyle\circ$}}}{#1}\rule{0pt}{2.3ex}}
\newcommand{\cirH}{\mathringbig{H}}
\newcommand{\cirHs}{\mathringbig{H}{}^s}
\newcommand{\cirHt}{\mathringbig{H}{}^t}
\newcommand{\cirHm}{\mathringbig{H}{}^m}
\newcommand{\cirHzero}{\mathringbig{H}{}^0}
\newcommand{\deO}{{\partial\Omega}}
\newcommand{\OO}{{(\Omega)}}
\newcommand{\Rn}{{(\R^n)}}
\newcommand{\Id}{{\mathrm{Id}}}
\newcommand{\gap}{\mathrm{Gap}}
\newcommand{\ggap}{\mathrm{gap}}
\newcommand{\isom}{{\xrightarrow{\sim}}}
\newcommand{\half}{{1/2}}
\newcommand{\mhalf}{{-1/2}}
\newcommand{\baro}{{\overline{\Omega}}} 
\newcommand{\inter}{{\mathrm{int}}}

\newcommand{\Hsp}{H^{s,p}}
\newcommand{\Htq}{H^{t,q}}
\newcommand{\tHsp}{{{\widetilde H}^{s,p}}}
\newcommand{\SP}{\ensuremath{(s,p)}}
\newcommand{\Xsp}{X^{s,p}}

\newcommand{\dd}{{d}}\newcommand{\pp}{{p_*}}

\newcommand{\Rnn}{\R^{n_1+n_2}}
\newcommand{\Tr}{{\mathrm{Tr}}}
\title[Sobolev spaces on non-Lipschitz sets]{Sobolev spaces on non-Lipschitz subsets of $\R^n$ with application to boundary integral \\ equations on fractal screens}

\author{S.\ N.\ Chandler-Wilde}
\address{Department of Mathematics and Statistics\\ University of Reading \\
Whiteknights PO Box 220 \\
Reading RG6 6AX\\
United Kingdom}
\email{s.n.chandler-wilde@reading.ac.uk}
\author{D.\ P.\ Hewett}
\address{Department of Mathematics\\ University College London \\
Gower Street \\
London WC1E 6BT\\
United Kingdom}
\email{d.hewett@ucl.ac.uk}
\author{A.\ Moiola}
\address{Department of Mathematics and Statistics\\ University of Reading \\
Whiteknights PO Box 220 \\
Reading RG6 6AX\\
United Kingdom}
\email{a.moiola@reading.ac.uk}

\begin{abstract} We study properties of the classical
fractional
Sobolev spaces
on non-Lipschitz subsets of $\R^n$.
We investigate the extent to which the properties of these spaces, and the relations between them, that hold in the well-studied case of a Lipschitz open set, generalise to non-Lipschitz cases.
Our motivation is to develop the functional analytic framework in which to formulate and analyse integral equations on non-Lipschitz sets. In particular we consider an application to boundary integral equations for wave scattering by planar screens that are non-Lipschitz, including cases where the screen is fractal or has fractal boundary.
\end{abstract}

\maketitle
\section{Introduction}\label{sec:intro}
In this paper we present a self-contained study of %
Hilbert--Sobolev spaces defined on arbitrary open and closed sets of $\R^n$, aimed at applied and numerical analysts interested in linear elliptic problems on rough domains, in particular in boundary integral equation (BIE) reformulations.
Our focus is on the Sobolev spaces $H^s(\Omega)$, $H^s_0(\Omega)$, $\tH^s(\Omega)$, $\cirHs(\Omega)$, and $H^s_F$, all described below, where $\Omega$ (respectively $F$) is an arbitrary open (respectively closed) subset of $\R^n$. Our goal is to investigate %
properties of these spaces (in particular, to provide natural unitary realisations for their dual spaces), and to clarify the nature of the relationships between them.

Our motivation for writing this paper is recent and current work by two of the authors \cite{CoercScreen,CoercScreen2,Ch:13,ScreenPaper} on problems of acoustic scattering by planar screens with rough (e.g.\ fractal) boundaries.
The practical importance of such scattering problems has been highlighted by the recent emergence of ``fractal antennas'' in electrical engineering applications, which have attracted attention due to their miniaturisation and multi-band properties; see the reviews \cite{GiRS:02,WeGa:03} and \cite[\S18.4]{Fal}. The acoustic case considered in \cite{CoercScreen,CoercScreen2,Ch:13,ScreenPaper} and the results of the current paper may be viewed as first steps towards developing a mathematical analysis of problems for such structures.

In the course of our investigations of BIEs on more general sets it appeared to us that the literature on the relevant classical Sobolev spaces, while undeniably vast, is not as complete %
or as clear as desirable in the case when the domain of the functions is an arbitrary open or closed subset of Euclidean space, as opposed to the very well-studied case of a Lipschitz open set.
By ``classical Sobolev spaces" we mean the simplest of Sobolev spaces, Hilbert spaces based on the $L^2$ norm, which are sufficient for a very large part of the study of linear elliptic BVPs and BIEs, and are for this reason the focus of attention for example in the classic monographs \cite{LiMaI} and \cite{ChPi} and in the more recent book by McLean \cite{McLean} that has become the standard reference for the theory of BIE formulations of BVPs for strongly elliptic systems.
However, even in this restricted setting there are many different ways to define Sobolev spaces on subsets of $\R^n$ (via e.g.\ weak derivatives, Fourier transforms and Bessel potentials, completions of spaces of smooth functions, duality, interpolation, traces, quotients, restriction of functions defined on a larger subset,~\ldots).
On Lipschitz open sets (defined e.g.\ as in %
\cite[1.2.1.1]{Gri}), many of these different definitions lead to the same Sobolev spaces and to equivalent norms. But, as we shall see, the situation is more complicated for spaces defined on more general subsets of $\R^n$.

Of course there already exists a substantial literature relating to function spaces on rough subsets of $\R^n$ (see e.g.~\cite{JoWa84,Triebel97FracSpec,Triebel83ThFS,Maz'ya,AdHe,MaPo97,Ca:00,St:03}).
However, many of the results presented here, despite being relatively %
elementary, do appear to be new and of interest and relevance for applications.
That we are able to achieve some novelty may be  due in part to the fact that we restrict our attention to the Hilbert--Sobolev framework, which means that many of the results we are interested in can be proved using Hilbert space techniques and geometrical properties of the domains, without the need for more general and intricate theories such as those of Besov and Triebel--Lizorkin spaces and atomic decompositions \cite{Triebel83ThFS,Maz'ya,AdHe} which are usually employed to describe function spaces on rough sets.
This paper is by no means an exhaustive study, but we hope that the results we provide, along with the open questions that we pose, will stimulate further research in this area.

Many of our results involve the question of whether or not a given subset of Euclidean space can support a Sobolev distribution of a given regularity (the question of ``$s$-nullity'', see \S\ref{subsec:Polarity} below). A number of results pertaining to this question have been derived recently in \cite{HewMoi:15} using standard results from potential theory in \cite{AdHe,Maz'ya}, and those we shall make use of are summarised in \S\ref{subsec:Polarity}. We will also make reference to a number of the concrete examples and counterexamples provided in \cite{HewMoi:15}, in order to demonstrate the sharpness (or otherwise) of our theoretical results.
Since our motivation for this work relates to the question of determining the correct function space setting in which to analyse integral equations posed on rough domains, we include towards the end of the paper an application to BIEs on fractal screens; further applications in this direction can be found in \cite{CoercScreen,Ch:13,ScreenPaper}.

We point out that one standard way of defining Sobolev spaces not considered in detail in this paper is interpolation (e.g.\ defining spaces of fractional order by interpolation between spaces of integer order, as for the famous Lions--Magenes space $H^{1/2}_{00}(\Omega)$). In our separate paper \cite{InterpolationCWHM} we prove that while the spaces $H^s(\Omega)$ and $\tH^s(\Omega)$ form interpolation scales for Lipschitz $\Omega$, if this regularity assumption is dropped the interpolation property does not hold in general (this finding contradicts an incorrect claim to the contrary in \cite{McLean}).
This makes interpolation a somewhat unstable operation on non-Lipschitz open sets, and for this reason we do not pursue interpolation in the current paper as a means of defining Sobolev spaces on such sets.
However, for completeness we collect in Remark~\ref{rem:LionsMagenes}
some basic facts concerning the space $H^{s}_{00}(\Omega)$ on Lipschitz open sets, derived from the results presented in the current paper and in \cite{InterpolationCWHM}.

\subsection{Notation and basic definitions}
In light of the considerable variation in notation within the Sobolev space literature, we begin by clarifying the notation and the basic definitions we use.
For any subset $E\subset\R^n$ we denote the complement of $E$ by $E^c:=\R^n\setminus E$, the closure of $E$ by $\overline{E}$, and the interior of $E$ by ${\rm int}(E)$. We denote by $\dimH(E)$ the Hausdorff dimension of $E$ (cf.\ e.g.\ \cite[\S5.1]{AdHe}), and by $m(E)$ the $n$-dimensional Lebesgue measure of $E$ (for measurable $E$). For $\bx \in \R^n$ and $r>0$ we write $B_r(\bx) := \{\by\in \R^n: |\bx-\by|< r\}$ and $B_r := \{\bx\in \R^n: |\bx|<r\}$.

Throughout the paper, $\Omega$ will denote a non-empty open subset of $\R^n$, and $F$ a non-empty closed subset of $\R^n$.
We say that $\Omega$ is $C^0$ (respectively $C^{0,\alpha}$, $0<\alpha<1$, respectively Lipschitz) if its boundary $\partial\Omega$ can be locally represented as the graph (suitably rotated) of a $C^0$ (respectively $C^{0,\alpha}$, respectively Lipschitz) function from $\R^{n-1}$ to $\R$, with $\Omega$ lying only on one side of $\partial\Omega$.
For a more detailed definition see, e.g., \cite[Definition 1.2.1.1]{Gri}.
We note that for $n=1$ there is no distinction between these definitions: we interpret them all to mean that $\Omega$ is a countable union of open intervals whose closures are disjoint.

Note that in the literature several alternative definitions of Lipschitz open sets can be found (see e.g.\ the discussion in \cite{Fr:79}).
The following definitions are stronger than that given above:
Stein's ``minimally smooth domains'' in \cite[{\S}VI.3.3]{Stein}, which require all the local parametrisations of the boundary to have the same Lipschitz constant and satisfy a certain finite overlap condition;
Adams' ``strong local Lipschitz property'' in \cite[4.5]{Adams};
Ne\v{c}as' Lipschitz boundaries \cite[\S1.1.3]{NEC67};
and Definition~3.28 in \cite{McLean}, which is the most restrictive of this list as it considers only sets with bounded boundaries for which sets it is equivalent to the ``uniform cone condition'' \cite[Theorem~1.2.2.2]{Gri}.
On the other hand, Definition~1.2.1.2 in \cite{Gri} (``Lipschitz manifold with boundary'') is weaker than ours; see \cite[Theorem~1.2.1.5]{Gri}.
In this paper we study function spaces defined on \emph{arbitrary} open sets. Since some readers may be unfamiliar with open sets that fail to be $C^0$, we give a flavour of the possibilities we have in mind. We first point the reader to the examples illustrated in Figure \ref{fig:TSExamples} below (unions of polygons meeting at vertices, double bricks, curved cusps, spirals, and ``rooms and passages'' domains), all of which fail to be $C^0$ at one or more points on their boundaries. But these examples are still rather tame. A more exotic example is the Koch snowflake \cite[Figure~0.2]{Fal}, which fails to be $C^0$ at any point on its (fractal) boundary. Another class of examples we will use to illustrate many of our results (e.g.\ in \S\ref{subsec:3spaces}) is found by taking $\Omega=\Omega_0 \setminus F$, where $\Omega_0$ is a regular ($C^0$, or even Lipschitz) open set (e.g.\ a ball or a cube) and $F$ is an arbitrary non-empty closed subset of $\Omega_0$. The set $F$ may have empty interior, in which case 
$\Omega\neq {\rm int}(\overline\Omega)$. Of particular interest to us will be the case where $F$ is a fractal set. A concrete example (used in the proof of Theorem \ref{thm:notequalbig} and cf.\ Remark \ref{rem:BIE} below) is where $\Omega_0$ is a ball and $F$ is a Cantor set (an uncountable closed set with zero Lebesgue measure---see Figure \ref{fig:CantorDust} for an illustration). 
As we will see, a key role in determining properties of Sobolev spaces defined on the open set $\Omega=\Omega_0\setminus F$ is played by the maximal Sobolev regularity of distributions that are supported inside $F$, which itself is closely related to the Hausdorff dimension of $F$. 

\subsubsection{Slobodeckij--Gagliardo vs Bessel--Fourier}
For $s\in\R$, the fundamental Hilbert--Sobolev spaces on an open set $\Omega\subset \R^n$ are usually defined either
\begin{enumerate}
\item[\emph{(i)}] intrinsically, using volume integrals over $\Omega$ of squared weak (distributional) derivatives for $s\in\N_0$, Slobodeckij--Gagliardo integral norms for $0<s\notin\N$, and by duality for $s<0$ (cf.\ \cite[pp.~73--75]{McLean}); or
\item[\emph{(ii)}]
extrinsically, as the set of restrictions to $\Omega$ (in the sense of distributions) of elements of the global space $H^s(\R^n)$, which is defined for all $s\in\R$ using the Fourier transform and Bessel potentials (cf.\ \cite[pp.~75--77]{McLean}).
\end{enumerate}
Following McLean \cite{McLean}, we denote by $W^s_2(\Omega)$ the former class of spaces and by $H^s(\Omega)$ the latter.
Clearly $H^s(\Omega)\subset W^s_2(\Omega)$ for $s\geq 0$; in fact the two classes of spaces coincide and their norms are equivalent whenever there exists a continuous extension operator $W^s_2\OO\to H^s\Rn$ \cite[Theorem~3.18]{McLean};
this exists (at least for $s\geq 0$) for Lipschitz $\Omega$ with bounded boundary \cite[Theorem~A.4]{McLean}, and more generally for ``minimally smooth domains'' \cite[{\S}VI, Theorem~5]{Stein} and ``$(\varepsilon,\delta)$ locally uniform domains'' %
\cite[Definition~5 and Theorem~8]{Rogers}.
But it is easy to find examples where the two spaces are different: if $\Omega$ is Lipschitz and bounded, and
$\Omega':=\Omega\setminus\Pi$, where $\Pi$ is a hyperplane that divides $\Omega$ into two components, then $H^s(\Omega')=H^s\OO$ for $n/2<s\in\N$ as their elements require a continuous extension to $\R^n$, while the elements of $W^s_2(\Omega')$ can jump across $\Pi$, so $H^s(\Omega')\subsetneqq W^s_2(\Omega')$.

In the present paper we will only investigate the spaces $H^s(\Omega)$ and certain closed subspaces of $H^s(\R^n)$ related to $\Omega$, i.e.\ we choose option \emph{(ii)} above.
We cite two main reasons motivating this choice (see also \cite[\S3.1]{Triebel83ThFS}).

Firstly, while the intrinsic spaces $W^s_2(\Omega)$ described in option \emph{(i)} are the standard setting for BVPs posed in an open set $\Omega$ and their finite element-type discretisations, the extrinsic spaces $H^s(\Omega)$ and certain closed subspaces of $H^s(\R^n)$ arise naturally in BIE formulations.
An example (for details see
\S\ref{sec:BIE} and \cite{CoercScreen,ScreenPaper}) is the scattering of an acoustic wave propagating in $\R^{n+1}$ ($n=1$ or $2$) by a thin screen, assumed to occupy a bounded relatively open subset of the hyperplane $\{\bx\in\R^{n+1},\, x_{n+1}=0\}$. Identifying this hyperplane with $\R^n$ and the screen with an open subset $\Gamma\subset\R^n$ in the obvious way, one can impose either Dirichlet or Neumann boundary conditions on the screen by first taking a (trivial) Dirichlet or Neumann trace onto the hyperplane $\R^n$, then prescribing the value of the restriction of this trace to $\Gamma$, as an element of $H^{1/2}(\Gamma)$ or $H^{-1/2}(\Gamma)$ respectively. The solution to the associated BIE is respectively either the jump in the normal derivative of the acoustic field or the jump in the field itself across the hyperplane, these jumps naturally lying in the closed subspaces $H^{-1/2}_{\overline{\Gamma}}\subset H^{-1/2}(\R^n)$ and $H^{1/2}_{\overline{\Gamma}}\subset H^{1/2}(\R^n)$ respectively (see below
for definitions).

Secondly, on non-Lipschitz open sets $\Omega$ the intrinsic spaces $W^s_2(\Omega)$ have a number of undesirable properties.
For example, for $0<s<1$ the embedding $W^1_2(\Omega)\subset W^s_2(\Omega)$ may fail and the embedding
$W^s_2(\Omega)\subset W^0_2(\Omega)=L^2(\Omega)$ may be non-compact (see \cite[\S~9]{DiPaVa:12}).
Other pathological behaviours are described in \S1.1.4 of \cite{Maz'ya}: for $2\le \ell\in\N$, the three spaces defined by the (squared) norms $\|u\|_{L^\ell_2(\Omega)}^2:=\int_\Omega\sum_{\boldsymbol{\alpha}\in\N^n, |\boldsymbol\alpha|=\ell}|D^{\boldsymbol \alpha} u|^2\rd\bx$,
$\|u\|_{L^0_2(\Omega)}^2+\|u\|_{L^\ell_2(\Omega)}^2$ and
$\sum_{j=0}^\ell\|u\|_{L^j_2(\Omega)}^2$ may be all different from each other.

\subsubsection{``Zero trace'' spaces}
\label{sec:ZeroTrace}
In PDE applications, one often wants to work with Sobolev spaces on an open set $\Omega$ which have ``zero trace'' on the boundary of $\Omega$.
There are many different ways to define such spaces; in this paper we consider the following definitions, which are equivalent only under certain conditions on $\Omega$ and $s$ (as will be discussed in \S\ref{subsec:3spaces}):
\begin{itemize}
\item $H^s_0(\Omega)$, the closure in $H^s(\Omega)$ of the space of smooth, compactly supported functions on $\Omega$.
\item $\tH^s(\Omega)$, the closure in $H^s(\R^n)$ of the space of smooth, compactly supported functions on $\Omega$.
\item $H^s_{\overline\Omega}$, the set of those distributions in $H^s(\R^n)$ whose support lies in the closure $\overline\Omega$.
\item $\cirHs(\Omega)$, defined for $s\ge0$ as the set of those distributions in $H^s(\R^n)$ that are equal to zero almost everywhere in the complement of $\Omega$.
\end{itemize}
$H^s_0(\Omega)$, being a closed subspace of $H^s(\Omega)$, is a space of distributions on $\Omega$, while $\tH^s(\Omega)$, $H^s_{\overline\Omega}$ and $\cirHs(\Omega)$, all being closed subspaces of $H^s\Rn$, are spaces of distributions on $\R^n$ (which can sometimes be embedded in $H^s(\Omega)$ or $H^s_0(\Omega)$, as we will see). All the notation above is borrowed from \cite{McLean} (see also \cite{HsWe08,Steinbach,ChPi}), except the notation $\cirHs(\Omega)$ which we introduce here (essentially the same space is denoted $\tilde W^s_2(\Omega)$ in \cite{Gri}).

We remark that for Lipschitz or smoother open sets $\Omega$, the above spaces are classically characterised as kernels of suitable trace operators (e.g.\ \cite[Theorem~3.40]{McLean}, \cite[Theorem~1.5.1.5]{Gri}, \cite[Chapter 1, Theorem~11.5]{LiMaI}). %
Trace spaces on closed sets $F\subset\R^n$ with empty interior (e.g.\ finite unions of submanifolds of $\R^n$, or fractals such as Cantor sets) are sometimes defined as quotient spaces, e.g.\ \cite[Definition~6.1]{ClHi:13} considers
the space $H^{1/2}([F])$, defined as $H^{1/2}([F]):=
W^1_2(\R^n)/\overline{\scrD(\R^n\setminus F)}^{W^1_2(\R^n\setminus F)}$; other similar trace spaces are $H^s\Rn/\tH^s(\R^n\setminus F)$
and $H^s(\R^n\setminus F)/H^s_0(\R^n\setminus F)$.
While we do not discuss such trace operators or trace spaces in this paper, we point out that our results in \S\ref{subsec:DiffDoms} and %
\S\ref{subsec:Hs0vsHs},
respectively, describe precisely when the latter two trace spaces are or are not trivial.

\subsection{Overview of main results}
We now outline the structure of the paper and summarise our main results.

\paragraph{Preliminary Hilbert space results.}
In \S\ref{sec:hs} we recall some basic facts regarding (complex) Hilbert spaces that we use later to construct unitary isomorphisms between Sobolev spaces and their duals.
The key result in \S\ref{sec:DualSpaceRealisations} (stated as Lemma~\ref{lem:hs_orth}) is that given a unitary realisation $\cH$ of the dual of a Hilbert space $H$ and a closed subspace $V\subset H$, the dual of $V$ can be realised unitarily in a natural way as the orthogonal complement of the annihilator of $V$ in $\cH$.
In \S\ref{subsec:ApproxVar} we consider sequences of continuous and coercive variational equations posed in nested (either increasing or decreasing) Hilbert spaces, and prove the convergence of their solutions under suitable assumptions, using arguments based on C\'ea's lemma.
These results are used in \S\ref{sec:BIE} to study the limiting behaviour of solutions of BIEs on sequences of Lipschitz open sets $\Gamma_j$, including cases where $\Gamma_j$ converges as $j\to \infty$ to a closed fractal set, or to an open set with a fractal boundary.

\paragraph{Sobolev space definitions.}
In \S\ref{subsec:SobolevDef} we recall the precise definitions and basic properties of the function spaces $H^s(\R^n)$, $H^s(\Omega)$, $H^s_0(\Omega)$, $\tH^s(\Omega)$, $\cirHs(\Omega)$, and $H^s_F\subset H^s(\R^n)$ introduced above. Our presentation closely follows that of \cite[Chapter~3]{McLean}.

\paragraph{Duality.}
In \S\ref{subsec:DualAnnih} we describe natural unitary realisations of the duals of the Sobolev spaces introduced in \S\ref{subsec:SobolevDef}. By ``natural'' we mean that the duality pairing extends the $L^2$ inner product, and/or the action of a distribution on a test function.
For example, the dual space of $H^{s}(\Omega)$ can be naturally and unitarily identified with the space $\tH^{-s}(\Omega)$, and vice versa. This is very well known for $\Omega$ sufficiently regular
(e.g.\ Lipschitz with bounded boundary, e.g., \cite[Theorem 3.30]{McLean})
but our proof based on the abstract Hilbert space results in \S\ref{sec:hs} makes clear that the geometry of $\Omega$ is quite irrelevant; the result holds for any $\Omega$ (see Theorem \ref{thm:DualityTheorem}). We also provide what appear to be new realisations of the dual spaces of $H^s_F$ and $H^s_0(\Omega)$.

\paragraph{$s$-nullity.}
In \S\ref{subsec:Polarity} we introduce the concept of $s$-nullity, a measure of the negligibility of a set in terms of Sobolev regularity. This concept will play a prominent role throughout the paper, and many of our key results relating different Sobolev spaces will be stated in terms of the $s$-nullity (or otherwise) of the set on which a Sobolev space is defined, of its boundary, or of the symmetric difference between two sets. For $s\in\R$ we say a set $E\subset\R^n$ is $s$-null if there are no non-zero elements of $H^s(\R^n)$ supported in $E$. (Some other authors \cite{HoLi:56,Li:67a,Li:67b,Maz'ya} refer to such sets as ``$(-s,2)$-polar sets'', or \cite{AdHe,Maz'ya} as sets of uniqueness for $H^s(\R^n)$; for a more detailed discussion of terminology see Remark \ref{rem:polarity}.)
In Lemma \ref{lem:polarity} we collect a number of results concerning $s$-nullity and its relationship to analytical and geometrical properties of sets (for example Hausdorff dimension) that have recently been derived in \cite{HewMoi:15} using potential theoretic results on set capacities taken from \cite{Maz'ya,AdHe}.

\paragraph{Spaces defined on different subsets of $\R^n$.}
Given two different Lipschitz open sets $\Omega_1,\Omega_2\subset\R^n$, the symmetric difference $(\Omega_1\cup\Omega_2)\setminus(\Omega_1\cap\Omega_2)$ has non-empty interior, and hence the Sobolev spaces related to $\Omega_1$ and $\Omega_2$ are different, in particular
$\tH^s(\Omega_1)\neq\tH^s(\Omega_2)$.
If the Lipschitz assumption is lifted the situation is different: for example, from a Lipschitz open set $\Omega$ one can subtract any closed set with empty interior (e.g.\ a point, a convergent sequence of points together with its limit, a closed line segment, curve or other higher dimensional manifold, or a more exotic fractal set) and what is left will be again an open set $\Omega'$.
In which cases is $\tH^s\OO=\tH^s(\Omega')$?
When is $H^s_{\Omega^c}=H^s_{{\Omega^{'}}^c}$? And how is $H^s\OO$ related to $H^s(\Omega')$?
In \S\ref{subsec:DiffDoms} we answer these questions precisely in terms of $s$-nullity.

\paragraph{Comparison between the ``zero-trace'' subspaces of $H^s(\R^n)$.}
The three spaces $\tH^s(\Omega)$, $H^s_{\overline\Omega}$ and $\cirHs(\Omega)$ are all closed subspaces of $H^s\Rn$. For arbitrary $\Omega$ they satisfy the inclusions
$$
\tH^s\OO\subset\cirHs\OO\subset H^s_{\overline\Omega}
$$
(with $\cirHs\OO$ present only for $s\ge0$).
In \S\ref{subsec:3spaces} we describe conditions under which the above inclusions are or are not equalities. For example, it is well known (e.g.\ \cite[Theorem 3.29]{McLean}) that when $\Omega$ is $C^0$ the three spaces coincide. A main novelty in this section is the construction of explicit counterexamples which demonstrate that this is not the case for general $\Omega$. A second is the proof, relevant to the diversity of configurations illustrated in Figure \ref{fig:TSExamples}, that $\tH^s(\Omega) = H^s_{\overline{\Omega}}$ for %
$|s| \leq 1/2$ ($|s|\leq 1$ for $n\geq 2$)
for a class of open sets whose boundaries, roughly speaking, fail to be $C^0$ at a countable number of points.

 \paragraph{When is $H^s_0(\Omega)=H^s(\Omega)$?}
 In \S\ref{subsec:Hs0vsHs} we investigate the question of when $H^s_0(\Omega)$ is or is not equal to $H^s(\Omega)$. One classical result (see \cite[Theorem 1.4.2.4]{Gri} or \cite[Theorem 3.40]{McLean}) is that if $\Omega$ is Lipschitz and bounded then $H^s_0(\Omega)=H^s(\Omega)$ for $0\leq s\leq 1/2$. Using the dual space realisations derived in \S\ref{subsec:DualAnnih} we show that, for arbitrary $\Omega$, equality of $H^s_0(\Omega)$ and $H^s(\Omega)$ is equivalent to a certain subspace of $H^{-s}(\R^n)$ being trivial. From this we deduce a number of necessary and sufficient conditions for equality, many of which appear to be new;
in particular our results linking the equality of $H^s_0(\Omega)$ and $H^s(\Omega)$ to the fractal dimension of $\partial\Omega$ improve related results presented in \cite{Ca:00}.

 \paragraph{The restriction operator.}
 One feature of this paper is that we take care to distinguish between spaces of distributions defined on $\R^n$ (including $H^s(\R^n)$, $\tH^s(\Omega)$,$\cirHs(\Omega)$, $H^s_{\overline\Omega}$) and spaces of distributions defined on $\Omega$ (including $H^s_0(\Omega)$, $H^s(\Omega)$).
 The link between the two is provided by the restriction operator $|_\Omega:H^s(\R^n)\to H^s(\Omega)$. In \S\ref{subsec:restriction} we collect results from \cite{Hs0paper} on its mapping properties (injectivity, surjectivity, unitarity).
In Remark~\ref{rem:LionsMagenes} we briefly mention the relationship of $\tH^s\OO$ and $H^s_0\OO$ with the classical Lions--Magenes space $H^s_{00}\OO$ (defined by interpolation), using results recently derived in \cite{InterpolationCWHM}.

\paragraph{Sequences of subsets.}
Many of the best-known fractals (for example Cantor sets, Cantor dusts, the Koch snowflake, the Sierpinski carpet, and the Menger sponge) are defined by taking the union or intersection of an infinite sequence of simpler, nested ``prefractal'' sets. In \S\ref{subsec:Seqs+Eqs} we determine which of the Sobolev spaces defined on the limiting set naturally emerges as the limit of the spaces defined on the approximating sets.
This question is %
relevant when the different spaces on the limit set do not coincide, e.g.\ when $\tH^s\OO\subsetneqq H^s_{\overline\Omega}$.
In this case the correct function space setting depends on whether the limiting set is to be approximated from ``inside'' (as a union of nested open sets), or from the ``outside'' (as an intersection of nested closed sets).

\paragraph{Boundary integral equations on fractal screens.}
\S\ref{sec:BIE} contains the major application of the paper, namely the BIE formulation of acoustic (scalar) wave scattering by fractal screens. We show how the Sobolev spaces $H^s(\Omega)$, $\tH^s(\Omega)$, $H^s_{F}$ all arise naturally in such problems,  pulling together many of the diverse results proved in the other sections of the paper.  In particular, we study the limiting behaviour as $j\to\infty$ of the solution in the fractional Sobolev space $\tH^{\pm 1/2}(\Gamma_j)$ of the BIE on the sequence of regular screens $\Gamma_j$, focussing particularly on cases where $\Gamma_j$ is a sequence of prefractal approximations to a limiting screen $\Gamma$ that is fractal or has fractal boundary.

\section{Preliminary Hilbert space results} \label{sec:hs}

In this section we summarise the elementary Hilbert space theory which underpins our later discussions.

We say that a mapping $\iota: H_1\to H_2$ between topological vector spaces $H_1$ and $H_2$ is an \emph{embedding} if it is linear, continuous, and injective, and indicate this by writing $H_1 \hookrightarrow_\iota H_2$, abbreviated as $H_1 \hookrightarrow H_2$ when the embedding $\iota$ is clear from the context.
We say that a mapping $\iota: H_1\to H_2$
is an \emph{isomorphism} if $\iota$ is linear and a homeomorphism. %
If $H_1$ and $H_2$ are Banach spaces and, additionally, the mapping is isometric (preserves the norm) then we say that $\iota$ is an \emph{isometric isomorphism}. If $H_1$ and $H_2$ are Hilbert spaces and, furthermore, $\iota$ preserves the inner product,
then we say that $\iota$ is a \emph{unitary isomorphism} (the terms $H$-isomorphism and Hilbert space isomorphism are also commonly used), and we write $H_1 \cong_{\iota} H_2$.
We recall that an isomorphism between Hilbert spaces is unitary if and only if it is isometric \cite[Proposition 5.2]{Conway}.

From now on let $H$ denote a complex Hilbert space with inner product $(\cdot,\cdot)_H$, and $H^*$ its dual space (all our results hold for real spaces as well, with the obvious adjustments).
Following, e.g., Kato \cite{Ka:95} we take $H^*$ to be the space of {\em anti-linear} continuous functionals on $H$ (sometimes called the {\em anti-dual}), this choice simplifying some of the notation and statement of results. The space $H^*$ is itself a Banach space with the usual induced operator norm. Further, it is an elementary result %
that the so-called {\em Riesz isomorphism}, the mapping $R:H\to H^*$  which maps $\phi\in H$ to the anti-linear functional $\ell_\phi\in H^*$, given by $\ell_\phi(\psi)= (\phi,\psi)_H$, for $\psi\in H$, is an isometric isomorphism. This provides a natural identification of the Banach space $H^*$ with $H$ itself.
Moreover, this mapping allows us to define an inner product $(\cdot,\cdot)_{H^*}$ on $H^*$, by the requirement that $(\phi,\psi)_H = (\ell_\phi,\ell_\psi)_{H^*}$, $\phi,\psi\in H$, and this inner product is compatible with the norm on $H^*$. With this canonical inner product $H^*$ is itself a Hilbert space and the Riesz isomorphism is a unitary isomorphism%
\footnote{As for Kato \cite{Ka:95}, a large part of our preference for our dual space convention (that our functionals are anti-linear rather than linear) is that the Riesz mapping is an isomorphism. If one prefers to work with linear functionals one can construct an isomorphism between the spaces of continuous linear and anti-linear functionals; indeed, in many important  cases there is a canonical choice for this isomorphism. %
Precisely, if $\psi\mapsto \psi^*$ is any anti-linear isometric involution on $H$ (sometimes called a conjugate map, and easily constructed using an orthogonal basis for $H$, e.g., \cite[Conclusion 2.1.18]{sauter-schwab11}) the map $\phi^*\mapsto\phi$, from the Hilbert space of continuous anti-linear functionals to the space of continuous linear functionals, defined by $\phi(\psi) = \phi^*(\psi^*)$, $\psi \in H$, is a unitary isomorphism. In general there is no natural choice for this conjugate map, but when, as in \S\ref{sec:ss} onwards, $H$ is a space of complex-valued functions the canonical choice is $\psi^*=\overline{\psi}$. When $H$ is real all this is moot; linear and anti-linear coincide.}.

\subsection{Realisations of dual spaces}\label{sec:DualSpaceRealisations}
It is frequently convenient, e.g.\ when working with Sobolev spaces, to identify the dual space $H^*$ not with $H$ itself but with another Hilbert space $\mathcal{H}$.
If $\cI:\cH\to H^*$ is a unitary isomorphism
then we say that $(\cH,\cI)$
is a {\em unitary realisation} of $H^*$, and %
\begin{equation} \label{dp}
\langle\psi, \phi\rangle := \cI\psi(\phi), \quad \phi\in H, \psi\in \cH,
\end{equation}
defines a bounded sesquilinear form on $\cH\times {H}$, called the {\em duality pairing}.

The following lemma shows that, given a unitary realisation $(\cH, \cI)$ of $H^*$, there is a natural unitary isomorphism $\cIstar:H\to \cH^*$, so that $(H, \cIstar)$ is a realisation of $\cH^*$.
The operator $\cIstar$ is the adjoint operator of $\cI$ after the canonical identification of $H$ with its bidual $H^{**}$.
\begin{lem} \label{dual_lem}
If $H$ and $\cH$ are Hilbert spaces and $\cI:\cH\to H^*$ is a unitary isomorphism, then $\cIstar:H\to \cH^*$, given by $\cIstar\phi(\psi) = \overline{\cI\psi(\phi)}$, for $\phi\in H$ and $\psi \in \cH$,
is a unitary isomorphism, and the corresponding duality pairing $\langle \cdot,\cdot\rangle$ on $H\times \cH$ is
\begin{equation*}%
\langle\phi, \psi\rangle   :=\cIstar\phi(\psi) = \overline{\langle \psi, \phi\rangle}, \quad \phi\in H, \psi\in \cH,
\end{equation*}
where the duality pairing on the right hand side is that on $\cH\times {H}$, as defined in \eqref{dp}.
\end{lem}
\begin{proof}
For $\phi\in H$ and $\psi\in \cH$, where ${R}:{H}\to {H}^*$ and ${\mathcal R}:\cH\to \cH^*$ are the Riesz isomorphisms,
\begin{align*} %
\cIstar\phi(\psi) = \overline{\cI\psi(\phi)} = \overline{(R^{-1}\cI\psi,\phi)_H} =(\phi,R^{-1}\cI\psi)_H&= (\cI^{-1}R\phi, \psi)_{\cH}  \\&= {\mathcal R}\cI^{-1}R\phi(\psi),
\end{align*}
so that $\cIstar = {\mathcal R}\cI^{-1}R$ is a composition of unitary isomorphisms, and hence a unitary isomorphism.
\end{proof}
Similarly, there is associated to $(\cH,\cI)$ a natural unitary isomorphism  $j:H\to \cH$ defined by $j= \cI^{-1}R$, where $R:H\to H^*$ is the Riesz isomorphism.

For a subset $V\subset H$, we denote by $V^\perp$ the subset of $H$ orthogonal to $V$, a closed linear subspace of $H$. When $V$ is itself a closed linear subspace, in which case $V^\perp$ is termed the orthogonal complement of $V$, we can define $P:H\to V$ ({\em orthogonal projection onto $V$}) by $P\phi=\psi$, where $\psi$ is the  best approximation to $\phi$ from $V$. This mapping is linear and bounded with $\|P\|=1$ and $P=P^2=P^*$, where $P^*:H\to H$ is the Hilbert-space adjoint operator of $P$.
$P$ has range $P(H)=V$ and kernel $\ker(P)= V^\perp$; moreover $H=V\oplus V^\perp$, and $V^{\perp\perp} = V$.
Furthermore, if $(\cH,\cI)$ is a unitary realisation of $H^*$ and $\langle \cdot, \cdot\rangle$ is the associated duality pairing (as in \eqref{dp}),
we define, %
for any subset $V\subset H$,
\begin{align}\label{eq:AnnihilatorDef}
V^{a,\mathcal{H}} := \{\psi\in \cH:\langle \psi,\phi\rangle = 0, \mbox{ for all }\phi\in V\}\subset\cH,
\end{align}
this the {\em annihilator of $V$ in $\cH$}.  For $\phi,\psi\in H$, $\langle j\psi,\phi \rangle =R\psi(\phi) =(\psi,\phi)_H$, so that $V^{a,\mathcal{H}} = j(V^\perp)$. When $V$ is a closed linear subspace of $H$, since $j$ preserves orthogonality and $V^{\perp\perp}=V$, we have
\begin{equation} \label{eq:AnnihilatorResult}
(V^\perp)^{a,\mathcal{H}}=j(V)= \left(V^{a,\mathcal{H}}\right)^\perp,
\quad \textrm{ and } \quad \left(V^{a,\mathcal{H}}\right)^{a,H} = j^{-1}\big((V^{a,\mathcal{H}})^\perp\big) = V.
\end{equation}

Given a linear subspace $V\subset H$ we can form the {\em quotient space} $H/V$ as $\{\phi+V:\phi\in H\}$. If $V$ is closed then $H/V$ is a Banach space, with norm
\begin{equation} \label{qsn}
\|\phi+V\|_{H/V} := \inf_{\psi\in V} \|\phi+\psi\|_H = \|Q \phi\|_H,
\end{equation}
where $Q:H\to V^\perp$ is orthogonal projection.
The mapping $Q_/:H/V\to V^\perp$, defined by $Q_/(\phi+V) = Q\phi$, is clearly surjective and so an isometric isomorphism. Defining an inner product compatible with the norm on  $H/V$ by $(\phi+V,\psi+V)_{H/V} = (Q\phi,Q\psi)_H$, $H/V$ becomes a Hilbert space and $Q_/$ a unitary isomorphism, i.e.
\begin{equation*} %
H/V \cong_{Q_/} V^\perp.
\end{equation*}

A situation which arises frequently in Sobolev space theory is where we have identified a particular unitary realisation $(\cH,\cI)$ of a dual space $H^*$ and we seek a unitary realisation of $V^*$, where $V$ is a closed linear subspace of $H$.
The following result shows that an associated natural unitary realisation of $V^*$ is $({\mathcal V}, \cI_\cV)$, where $\cV=\left(V^{a,\mathcal{H}}\right)^\perp\subset \cH$ and $\cI_\cV$ is the restriction of $\cI$ to $\cV$. This is actually a special case of a more general Banach space result, e.g.\ \cite[Theorem 4.9]{Ru91}, but since it plays such a key role in later results, for ease of reference we restate it here restricted to our Hilbert space context, and provide the short proof.

\begin{lem} \label{lem:hs_orth} Suppose that $H$ and $\cH$ are Hilbert spaces, $\cI:\cH\to H^*$ is a unitary isomorphism,
and $V\subset H$ is a closed linear subspace.
Set ${\mathcal V} := \left(V^{a,\mathcal{H}}\right)^\perp\subset \cH$,  and define $\cI_\cV: {\mathcal V}\to V^*$  by $\cI_\cV\psi(\phi)=\cI\psi(\phi)$, for $\phi\in V, \psi\in {\mathcal V}$.
Then $({\mathcal V},\cI_\cV)$ is a unitary realisation of $V^*$, with duality pairing
\begin{equation*} %
\langle\psi,\phi\rangle_V:= \cI_\cV\psi(\phi) =\langle\psi,\phi\rangle, \quad \phi\in V, \psi\in {\mathcal V},
\end{equation*}
where $\langle\cdot,\cdot\rangle$ is the duality pairing on $\cH\times H$ given by \eqref{dp}.
\end{lem}
\begin{proof}
As above, let $R:H\to H^*$ be the Riesz isomorphism and $j:= \cI^{-1}R:H\to \cH$, both unitary isomorphisms. $(V,R_V)$ is a unitary realisation of $V^*$,  where $R_V:V\to V^*$ is the Riesz isomorphism. Thus, since ${\mathcal V} = j(V)$ by \eqref{eq:AnnihilatorResult}, another unitary realisation is $({\mathcal V},R_V j^{-1}|_{\mathcal V})$. Further, for $\phi\in V$, $\psi\in {\mathcal V}$,
\begin{align*}
\label{}
 R_Vj^{-1}\psi(\phi) = (j^{-1}\psi, \phi)_V = (j^{-1}\psi, \phi)_H = Rj^{-1}\psi(\phi) = \cI\psi(\phi) &=\langle\psi,\phi\rangle \\ &= \cI_\cV\psi(\phi),
\end{align*}
 so that $\cI_\cV = R_Vj^{-1}|_{\mathcal V}$.
\end{proof}

\begin{rem} \label{rem:orth}
Lemma \ref{lem:hs_orth} gives a natural unitary realisation of the dual space of a closed subspace $V$ of a Hilbert space $H$. This lemma applies in particular to the closed subspace $V^\perp$. In view of \eqref{eq:AnnihilatorResult} and Lemma \ref{lem:hs_orth} we have that
$({\mathcal V}^\perp,\cI_{\cV^\perp})$ is a unitary realisation of $(V^\perp)^*$, with ${\mathcal V}^\perp =V^{a,\mathcal{H}}$ and $\cI_{\cV^\perp}\psi(\phi) =\langle\psi,\phi\rangle$, $\phi\in V^\perp, \psi\in {\mathcal V}^\perp$.
\end{rem}

Figure \ref{fig:Hilbert} illustrates as connected commutative diagrams the spaces in this section and key elements of the proofs of the above lemmas.

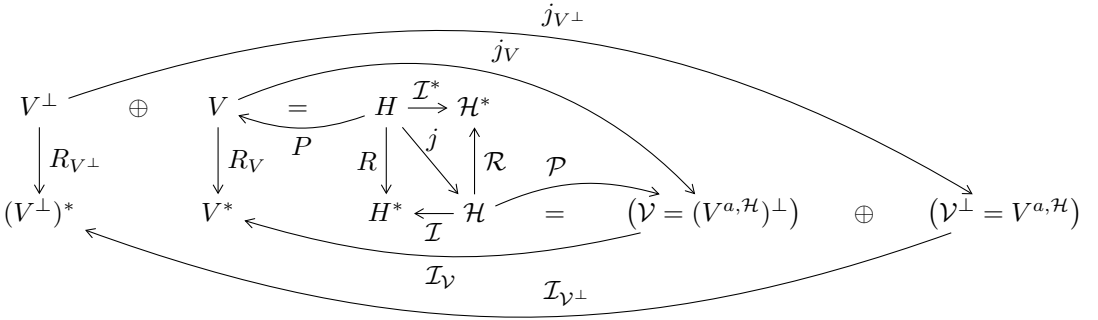
\begin{figure}[htb!]
\begin{center}
\begin{tikzpicture}
\hspace{-15mm}
\matrix[matrix of math nodes,
column sep={13pt},  %
row sep={40pt,between origins}, %
text height=1.5ex, text depth=0.25ex] (s)
{
|[name=Vp]|V^\perp & |[name=plus]|\oplus & |[name=V]|V & |[name=eq]|=&
|[name=H]|H & |[name=cHs]|\mathcal H^*\\
|[name=Vps]|(V^\perp)^* &  & |[name=Vs]|V^* & & |[name=Hs]|H^* & |[name=cH]|\mathcal H
& |[name=ceq]|=&
|[name=cV]|\big(\mathcal{V}=(V^{a,\mathcal H})^\perp\big) & |[name=cplus]|\oplus &
|[name=cVp]|\big(\mathcal V^\perp=V^{a,\mathcal H}\big) \\
};
\draw[->,>=angle 60] %
 (H) edge node[auto] {\(\cIstar\)} (cHs)
 (cH) edge node[auto] {\(\mathcal I\)} (Hs)
 (H) edge node[auto] {\(\!\!\!j\)} (cH)
 (H) edge node[auto,swap] {\(R\)} (Hs)
 (cH) edge node[auto,swap] {\(\mathcal R\)} (cHs)
 (V) edge node[auto] {\(R_V\)} (Vs)
 (Vp) edge node[auto] {\(R_{V^\perp}\)} (Vps)
 (cV) edge[bend left=15] node[auto] {\(\mathcal I_{\mathcal V}\)} (Vs)
 (cVp) edge[bend left=20] node[auto,swap,pos=0.4] {\(\mathcal I_{\mathcal V^\perp}\)} (Vps)
 (V) edge[bend left=35] node[auto] {\(j_V\)} (cV)
 (Vp) edge[bend left=25] node[auto] {\(j_{V^\perp}\)} (cVp);
\draw[->,>=angle 60] (H) edge[bend left=20] node[auto]{\(P\)} (V)
(cH) edge[bend left=20] node[auto]{\(\cP\)} (cV);
\end{tikzpicture}
\end{center}
\caption{
A representation, as two connected commutative diagrams, of the Hilbert spaces and the mappings defined in \S\ref{sec:hs};  here $j_V$ and $j_{V^\perp}$ are the restrictions of $j$ to $V$ and $V^\perp$, respectively.
Every arrow represents a unitary isomorphism, except for the two orthogonal projections $P:H\to V$ and $\cP:\cH\to\cV$.  %
\label{fig:Hilbert}}
\end{figure}

\subsection{Approximation of variational equations in nested subspaces}
\label{subsec:ApproxVar}

Let $H$ be a Hilbert space, with its dual $H^*$ realised unitarily as some Hilbert space $\cH$ and associated duality pairing $\langle\cdot,\cdot\rangle$, as in \S\ref{sec:DualSpaceRealisations}. Fix $f\in \cH$, and suppose that $a(\cdot,\cdot):H\times H\to \C$ is a sesquilinear form that is continuous and coercive, i.e.,
$\exists C,c>0$ such that
\begin{equation} \label{eq:defcoer}
|a(u,v)|\le C\|u\|_H \|v\|_H, \qquad |a(v,v)|\ge c\|v\|^2_H
\qquad \forall u,v\in H.
\end{equation}
For any closed subspace $V\subset H$ the restriction of $a(\cdot,\cdot)$ to $V\times V$ is also continuous and coercive. Thus by the Lax--Milgram lemma there exists a unique solution $u_V\in V$ to the variational equation
\begin{equation}\label{eq:VarEq}
 a(u_V,v) = \langle f,v\rangle \qquad \forall v\in V,
\end{equation}
and the solution is bounded independently of the choice of $V$, by $\|u_V\|_H\le c^{-1}\|f\|_\cH$.
Furthermore, given closed, nested subspaces $V_1\subset V_2\subset H$,
C\'ea's lemma gives the following standard bound:
\begin{align}
\|u_{V_1}-u_{V_2}\|_H
\le\frac{C} c \inf_{v_1\in V_1} \|v_1-u_{V_2}\|_H.
\label{eq:Cea}
\end{align}

Consider increasing and decreasing sequences of closed, nested subspaces indexed by $j\in\N$,
$$V_1\!\subset\!\cdots\!\subset\! V_j\!\subset \!V_{j+1}\!\subset\!\cdots\!\subset \!H \!\quad\text{ and }\quad
 H\!\supset \!W_1\!\supset\!\cdots\!\supset\! W_j\!\supset \!W_{j+1}\!\supset\!\cdots, %
 $$
and define the limit spaces $V:=\overline{\bigcup_{j\in\N} V_j}$ and $W:=\bigcap_{j\in\N} W_j$.
C\'ea's lemma \eqref{eq:Cea} immediately gives convergence of the corresponding solutions of \eqref{eq:VarEq} in the increasing case:
\begin{equation}\label{eq:V-conver}
\|u_{V_j}-u_V\|_H \le \frac{C} c \inf_{v_j\in V_j} \|v_j-u_V\|_H
\xrightarrow{j\to\infty}0.
\end{equation}
In the decreasing case the following analogous result applies.
\begin{lem} \label{lem:dec} Define $\{W_j\}_{j=1}^\infty$ and $W$ as above. Then \mbox{$\|u_{W_j}-u_W\|_H\to 0$} as $j\to\infty$.
\end{lem}
\begin{proof}
The Lax--Milgram lemma gives that
$\|u_{W_j}\|_H \leq c^{-1} \|f\|_{\mathcal H}$,
so that $(u_{W_j})_{j=1}^\infty$ is bounded and has a weakly convergent subsequence, converging to a limit $u_*$. Further, for all $w\in W$, \eqref{eq:VarEq} gives
$$
a(u_W,w)=\langle f, w\rangle = a(u_{W_j},w) \to a(u_*,w),
$$
as $j\to\infty$ through that subsequence, so that $u_*=u_W$. By the same argument every subsequence of $(u_{W_j})_{j=1}^\infty$ has a subsequence converging weakly to $u_W$, so that $(u_{W_j})_{j=1}^\infty$ converges weakly to $u_W$. Finally, we see that%
\begin{align*}
c\|u_{W_j}-u_W\|^2_H &\leq |a(u_{W_j}-u_W,u_{W_j}-u_W)| \\ &= |\langle f,u_{W_j}\rangle - a(u_{W_j},u_W)-a(u_W,u_{W_j}-u_W)|,%
\end{align*}
which tends to 0 as $j\to \infty$, by the weak convergence of $(u_{W_j})_{j=1}^\infty$ and \eqref{eq:VarEq}.
\end{proof}

\section{Sobolev spaces} \label{sec:ss}
\subsection{Main definitions}\label{subsec:SobolevDef}
We now define the Sobolev spaces studied in this paper. Our presentation broadly follows that of \cite{McLean}.

\subsubsection{Distributions, Fourier transform and Bessel potential}
Given $n\in \N$, let $\scrD(\R^n)$ denote the space of compactly supported smooth test functions on $\R^n$, and for any open set $\Omega\subset \R^n$ let
$\scrD(\Omega):=\{u\in\scrD(\R^n):\supp{u}\subset\Omega\}$.
For $\Omega\subset \R^n$ let $\scrD^*(\Omega)$ denote the space of distributions on $\Omega$ (anti-linear continuous functionals on $\mathscr{D}(\Omega)$). With $L^1_{\rm loc}(\Omega)$ denoting the space of locally integrable functions on $\Omega$, the standard embedding $L^1_{\rm loc}(\Omega)\hookrightarrow \scrD^*(\Omega)$ is given by $u(v):=\int_\Omega u \overline{v}$ for $u\in L^1_{\rm loc}(\Omega)$ and $v\in \mathscr{D}(\Omega)$.
Let $\mathscr{S}(\R^n)$ denote the Schwartz space of rapidly decaying smooth test functions on $\R^n$, and  $\mathscr{S}^*(\R^n)$ the dual space of tempered distributions (anti-linear continuous functionals on $\mathscr{S}(\R^n)$).
Since the inclusion $\mathscr{D}(\R^n)\subset \mathscr{S}(\R^n)$ is continuous with dense image, we have $\scrS^*(\R^n)\hookrightarrow \scrD^*(\R^n)$.
For $u\in \mathscr{S}(\R^n)$ we define the Fourier transform $\hat{u}={\mathcal F} u\in \mathscr{S}(\R^n)$ and its inverse $\check{u}={\mathcal F}^{-1} u\in \mathscr{S}(\R^n)$ by
\begin{align*}%
\hat{u}(\bxi)&:= \frac{1}{(2\pi)^{n/2}}\int_{\R^n}\re^{-\ri \bxi\cdot \bx}u(\bx)\,\rd \bx , \;\; \bxi\in\R^n,
\\
\check{u}(\bx) &:= \frac{1}{(2\pi)^{n/2}}\int_{\R^n}\re^{\ri \bxi\cdot \bx}u(\bxi)\,\rd \bxi , \;\;\bx\in\R^n.
\end{align*}
We define the Bessel potential operator $\cJ_s$ on $\mathscr{S}(\R^n)$, for $s\in\R$, by $\cJ_s := {\mathcal F}^{-1}\cM_s{\mathcal F}$, where $\cM_s$ is multiplication by $(1+|\bxi|^2)^{s/2}$.
We extend these definitions to $\mathscr{S}^*(\R^n)$ in the usual way:
for $u\in \mathscr{S}^*(\R^n)$ and $v\in \mathscr{S}(\R^n)$ let
\begin{align}
\label{FTDistDef}
\hat{u}(v) := u(\check{v}),\quad
\check{u}(v) := u(\hat{v}),\quad \cM_su(v) := u(\cM_s v), \quad
(\cJ_s u)(v) := u(\cJ_s v),
\end{align}
Note that for $u\in \mathscr{S}^*(\R^n)$ it holds that $\widehat{\cJ_s u} = \cM_s\hat{u}$.

\subsubsection{Sobolev spaces on \texorpdfstring{$\R^n$}{Rn}}
\label{subsec:SobSpacesOnRn}
We define the Sobolev space $H^s(\R^n)\subset \mathscr{S}^*(\R^n)$ by
\begin{align*}
\mmbox{H^s(\R^n):=\cJ_{-s}\big(L^2(\R^n)\big) = \big\{u\in \mathscr{S}^*(\R^n): \cJ_s u \in \big(L^2(\R^n)\big)\big\},}
\end{align*}
equipped with the inner product
$\left(u,v\right)_{H^{s}(\R^n)}:=\left(\cJ_s u,\cJ_s v\right)_{L^2(\R^n)}$, which makes $H^s(\R^n)$ a Hilbert space and $\cJ_{-s}:L^2(\R^n)\to H^s(\R^n)$ a unitary isomorphism.
Furthermore, for any $s,t\in \R$, the map $\cJ_{t}:H^{s}(\R^n)\to H^{s-t}(\R^n)$ is a unitary isomorphism with inverse $\cJ_{-t}$.
If $u\in H^s(\R^n)$ then the Fourier transform $\hat u\in \scrS^*(\R^n)$ lies in $L^1_{\rm loc}(\R^n)$; that is, $\hat u$ can be identified with a locally integrable function.
Hence we can write %
\begin{align} \label{eq:HsProdNorm}%
\mmbox{ \begin{aligned}
\left(u,v\right)_{H^{s}(\R^n)} &= \int_{\R^n}(1+|\bxi|^2)^{s}\,\hat{u}(\bxi)\overline{\hat{v}(\bxi)}\,\rd \bxi, \\
\norm{u}{H^{s}(\R^n)}^2 &= \norm{\cJ_s u}{L^2(\R^n)}^2 = \int_{\R^n}(1+|\bxi|^2)^{s}|\hat{u}(\bxi)|^2\,\rd \bxi,
\end{aligned} } \qquad u,v\in H^s(\R^n).
\end{align}

For every $s\in \R$, $\scrD(\R^n)$ is a dense subset of $H^s(\R^n)$. Indeed \cite[Lemma 3.24]{McLean}, for all $u\in H^s(\R^n)$ and $\epsilon>0$ there exists $v\in \scrD(\R^n)$ such that
\begin{equation} \label{eq:approx}
\|u-v\|_{H^s(\R^n)} < \epsilon \quad \mbox{and} \quad \supp{v} \subset \{\bx\in \R^n:|\bx-\by|<\epsilon \mbox{ and } \by\in \supp{u}\},
\end{equation}
where $\supp{v}$ denotes the support of the distribution $v$, understood in the standard sense (e.g.~\cite[p.\ 66]{McLean}). A related standard result (this follows, e.g., from \cite[Exercise 3.14]{McLean}) is that,
for all $u\in H^s(\R^n)$ and $\epsilon>0$, there exists a compactly supported $v\in H^s(\R^n)$ such that
\begin{equation} \label{eq:approx2}
\|u-v\|_{H^s(\R^n)} < \epsilon \quad \mbox{and} \quad \supp{v} \subset \supp{u}.
\end{equation}

For any $-\infty<s<t<\infty$, $H^t(\R^n)$ is continuously embedded in $H^s(\R^n)$ with dense image
and $\|u\|_{H^s(\R^n)}<\|u\|_{H^t(\R^n)}$ for all $0\ne u\in H^t(\R^n)$.
When $s>n/2$, elements of $H^{s}(\R^n)$ can be identified with continuous functions (by the Sobolev embedding theorem \cite[Theorem 3.26]{McLean}). At the other extreme,
for any $\bx_0\in\R^n$ the Dirac delta function\footnote{To fit our convention that $H^s(\R^n)\subset \scrS^*(\R^n)$ is a space of {\em anti}-linear functionals on $\scrS(\R^n)$, we understand the action of $\delta_{\bx_0}$ by $\delta_{\bx_0}(\phi)=\overline{\phi(\bx_0)}$, $\phi\in \scrD(\R^n)$.}
\begin{equation}\label{eq:delta}
\delta_{\bx_0}\in H^{s}(\R^n)\qquad \text{if and only if}\qquad s<-n/2.
\end{equation}
Recall that for a multi-index $\boldsymbol\alpha\in\N_0^n$ we have ${\mathcal F}(\partial^{\boldsymbol\alpha}u/\partial \bx^{\boldsymbol\alpha})(\bxi)=(\ri\bxi)^{\boldsymbol\alpha}\hat u(\bxi)$. Then by Plancherel's theorem and \eqref{eq:HsProdNorm} it holds that
$$\|u\|^2_{H^{s+1}(\R^n)}=\|u\|^2_{H^s(\R^n)}+\sum_{j=1}^n\Big\|\frac{\partial u}{\partial x_j}\Big\|^2_{H^s(\R^n)} \qquad \forall u\in H^{s+1}(\R^n), \; s\in\R.$$
In particular, if $m\in\N_0$ then, where $|\boldsymbol\alpha|:=\sum_{j=1}^n\alpha_j$ for $\boldsymbol\alpha\in\N_0^n$,
\begin{align*}%
\|u\|^2_{H^{m}(\R^n)}&=\sum_{\substack{\boldsymbol\alpha\in\N_0^n,\\ |\boldsymbol\alpha|\le m}}
\binom{m}{|\boldsymbol\alpha|}\binom{|\boldsymbol\alpha|}{\boldsymbol\alpha}
\Big\|\frac{\partial^{|\boldsymbol\alpha|} u}{\partial \mathbf x^{\boldsymbol\alpha}}\Big\|_{L^2\Rn}^2
\\
&=\sum_{\substack{\boldsymbol\alpha\in\N_0^n,\\ |\boldsymbol\alpha|\le m}}
\frac{m!}{(m-|\boldsymbol\alpha|)!\alpha_1!\cdots\alpha_n!}
\Big\|\frac{\partial^{|\boldsymbol\alpha|} u}{\partial \mathbf x^{\boldsymbol\alpha}}\Big\|_{L^2\Rn}^2.
\end{align*}
Similar manipulations show that functions with disjoint support are orthogonal in $H^m\Rn$ for $m\in\N_0$. But we emphasize that this is not in general true in $H^s \Rn$ for $s\in \R\setminus\N_0$.

\subsubsection{The duality relation between \texorpdfstring{$H^s(\R^n)$ and $H^{-s}(\R^n)$}{Hs(Rn) and H-s(Rn)}} \label{subsec:dual1}
Where $R_s$ is the Riesz isomorphism $R_s:H^s(\R^n)\to (H^s(\R^n))^*$, the map  $\cI^s:=R_s \cJ_{-2s}$, from $H^{-s}(\R^n)$ to $(H^s(\R^n))^*$, is a unitary isomorphism, so $(H^{-s}(\R^n),\cI^s)$ is a unitary realisation of $(H^s(\R^n))^*$, with the duality pairing given by
\begin{align}\label{DualDef}
\left\langle u, v \right\rangle_{s}:=\cI^s u(v) =(\cJ_{-2s}u,v)_{H^s(\R^n)}=\left(\cJ_{-s} u,\cJ_s v\right)_{L^2(\R^n)} = \int_{\R^n}\hat{u}(\bxi) &\overline{\hat{v}(\bxi)}\,\rd \bxi,
\end{align}
for $u\in H^{-s}(\R^n)$ and $v\in H^s(\R^n)$.
This unitary realisation of $(H^s(\R^n))^*$ is attractive because the duality pairing \rf{DualDef} is simply the $L^2(\R^n)$ inner product when $u,v\in \scrS(\R^n)$, and a continuous extension of that inner product for $u\in H^{-s}(\R^n)$, $v\in H^{s}(\R^n)$. Moreover, if $u\in H^{-s}(\R^n)$ and $v\in \scrS(\R^n)\subset H^s(\R^n)$, then $\langle u, v\rangle_s$ coincides with the action of the tempered distribution $u$ on $v\in \scrS(\R^n)$, since (recalling \eqref{FTDistDef}) %
for $u\in H^{-s}(\R^n)$ and $v\in \scrS(\R^n)$
\begin{equation} \label{dualequiv}
\langle u, v\rangle_s  = (\cJ_{-s}u, \cJ_sv)_{L^2(\R^n)}= \cJ_{-s}u(\cJ_{s}v) = u(v). %
\end{equation}

\subsubsection{Sobolev spaces on closed and open subsets of \texorpdfstring{$\R^n$}{Rn}}
\label{subsec:SobSpacesClosedOpen}
Given $s\in \R$ and a closed set $F\subset \R^n$, we define %
\begin{equation} \label{HsSubFDef}
\mmbox{H_F^s :=\big\{u\in H^s(\R^n): \supp(u) \subset F\big\},}
\end{equation}
i.e.\ $H_F^s=\{u\in H^s\Rn: u(\varphi)=0\;
\forall \varphi\in\scrD(F^c)\}$.
Then $H_F^s$ is a closed subspace of $H^s(\R^n)$, so is a Hilbert space with respect to the inner product inherited from $H^s(\R^n)$.

There are many different ways to define Sobolev spaces on a non-empty open subset $\Omega\subset \R^n$.
We begin by considering three closed subspaces of $H^s(\R^n)$, which are all Hilbert spaces with respect to the inner product inherited from $H^s(\R^n)$. First, we have the space $H^s_{\overline{\Omega}}$, defined as in \rf{HsSubFDef}, i.e.
\begin{equation*} %
\mmbox{H_{\overline{\Omega}}^s := \big\{u\in H^s(\R^n): \supp(u) \subset \overline{\Omega}\big\}.}
\end{equation*}
Second, we consider %
\begin{equation*}
\mmbox{ \tH^s(\Omega):=\overline{\scrD(\Omega)}^{H^s(\R^n)}.}
\end{equation*}
Third, for $s\geq 0$ another natural space to consider is (see also Remark \ref{rem:ZeroExtension})%
\begin{align*}
\mmbox{ \begin{aligned}
\cirHs(\Omega) &:= \big\{u\in H^s(\R^n): u= 0 \mbox{ a.e. in } \Omega^c\big\}\\
&\,\,= \big\{u\in H^s(\R^n): \meas\big(\Omega^c\cap\supp{u}\big) = 0\big\}.
\end{aligned} }
\end{align*}
These three closed subspaces of $H^s(\R^n)$ satisfy the inclusions
\begin{align}
\label{eqn:inclusions}
\tH^s(\Omega)\subset \cirHs(\Omega)\subset H^s_{\overline\Omega}
\end{align}
(with $\cirHs(\Omega)$ present only for $s\geq0$).
If $\Omega$ is sufficiently smooth (e.g.\ $C^0$) then the three sets coincide, but in general all three can be different (this issue will be investigated in \S\ref{subsec:3spaces}).

Another way to define Sobolev spaces on $\Omega$ is by restriction from $H^s(\R^n)$. For $s\in\R$ let
$$
\mmbox{H^s(\Omega):=\big\{u\in \scrD^*(\Omega): u=U|_\Omega \textrm{ for some }U\in H^s(\R^n)\big\},}
$$
where $U|_\Omega$ denotes the restriction of the distribution $U$ to $\Omega$ in the standard sense \cite[p.~66]{McLean}.
We can identify $H^s(\Omega)$ with the quotient space $H^s(\R^n)/H^s_{\Omega^c}$ through the bijection
\begin{equation*} %
q_s:H^s(\R^n)/H^s_{\Omega^c}\to H^s(\Omega)
\quad\text{given by}\quad
q_s(U+H^s_{\Omega^c}) = U|_\Omega, \quad U\in H^s(\R^n).
\end{equation*}
Recalling the discussion of quotient spaces in and below \eqref{qsn}, this allows us to endow $H^s(\Omega)$ with a Hilbert space structure (making $q_s$ a unitary isomorphism), with the inner product given by
\begin{align*}
(u,v)_{H^s(\Omega)} := (q_s^{-1}u,q_s^{-1}v)_{H^s(\R^n)/H^s_{\Omega^c}}
& = (U+H^s_{\Omega^c},V+H^s_{\Omega^c})_{H^s(\R^n)/H^s_{\Omega^c}} \\ &= (Q_sU,Q_sV)_{H^s(\R^n)},
\end{align*}
for $u,v\in H^s(\R^n)$, where $U,V\in H^s(\R^n)$ are such that $U|_\Omega = u$, $V|_\Omega = v$, and
$Q_s$ is orthogonal projection from $H^s(\R^n)$ onto $(H^s_{{\Omega^c}})^\perp$, and the resulting norm given by%
\begin{align}
\mmbox{\|u\|_{H^{s}(\Omega)}=\|Q_sU\|_{H^s(\R^n)} = \min_{\substack{W\in H^s(\R^n)\\ W|_{\Omega}=u}}\normt{W}{H^{s}(\R^n)}.}
\label{eq:InfNorm}
\end{align}

We can also identify $H^s(\Omega)$ with $(H^s_{\Omega^c})^\perp$, by the unitary isomorphism $q_s {Q_s}_/^{-1}:(H^s_{\Omega^c})^\perp\to H^s(\Omega)$, where ${Q_s}_/:H^s(\R^n)/H^s_{\Omega^c} \to (H^s_{\Omega^c})^\perp$ is the quotient map defined from $Q_s$, as in \S\ref{sec:hs}.
In fact, it is easy to check that $q_s {Q_s}_/^{-1}$ is nothing but the restriction operator $|_\Omega$, so
\begin{equation}\label{eq:RestrIsUnitary}
|_\Omega :(H^s_{\Omega^c})^\perp\to H^s(\Omega) \qquad \text{is a unitary isomorphism}
\end{equation}
and the diagram in Figure \ref{fig:qsQs} commutes.
This means we can study the spaces $H^s(\Omega)$ (which, a priori, consist of distributions on $\Omega$) by studying subspaces of $H^s\Rn$; this is convenient, e.g., when trying to compare $H^s(\Omega_1)$ and $H^s(\Omega_2)$
for two different open sets $\Omega_1,\Omega_2$; see \S\ref{subsec:DiffDoms}.

\begin{figure}[tb!]
\begin{center}
\begin{tikzpicture}
\matrix[matrix of math nodes, column sep={80pt}] (s)
{&&|[name=quot]| H^s(\R^n)/H^s_{\Omega^c}\\
|[name=HR]| H^s(\R^n)&|[name=orth]| (H^s_{\Omega^c})^\perp&\\
&&|[name=HO]| H^s(\Omega) \\};
\draw[->,>=angle 60]
(HR) edge node[auto] {\(Q_s\)} (orth);
\draw[->,>=angle 60]
(quot) edge node[auto,swap,pos=0.3] {\(Q_{s/}\)} (orth)
(quot) edge node[auto] {\(q_s\)} (HO)
(orth) edge node[auto,swap,pos=0.7] {%
\(|_\Omega\)} (HO);
\end{tikzpicture}
\end{center}
\caption{The maps between $H^s(\R^n)$ and $H^s(\Omega)$, for $s\in\R$ and an open $\Omega\subset\R^n$, as described in \S\ref{subsec:SobSpacesClosedOpen}.
All the maps depicted are unitary isomorphisms except $Q_s$, which is an orthogonal projection, and this diagram commutes.\label{fig:qsQs}}
\end{figure}

Clearly%
\[\scrD(\overline\Omega):=\big\{u\in C^\infty(\Omega):u=U|_\Omega \textrm{ for some }U\in\scrD(\R^n)\big\}\]
is a dense subspace of $H^s(\Omega)$, since $\scrD(\R^n)$ is dense in $H^s(\R^n)$. The final space we introduce in this section is the closed subspace of $H^s(\Omega)$ defined by
\begin{equation}\label{eq:Hs0}
\mmbox{H^s_0(\Omega):=\overline{\scrD(\Omega)\big|_\Omega}^{H^s(\Omega)}.}
\end{equation}
$\tH^s(\Omega)$ and $H^s_0(\Omega)$ are defined  as closures in certain norms of $\scrD(\Omega)$ and $\scrD(\Omega)|_\Omega$, respectively,
so that the former is a subspace of $H^s(\R^n)\subset \scrS^*(\R^n)$ and the latter of  $H^s(\Omega)\subset\scrS^*(\R^n)|_\Omega \subset\scrD^*(\Omega)$.
For $s>1/2$ and sufficiently uniformly smooth $\Omega$, both $\tH^s\OO$ and $H^s_0\OO$ consist of functions with ``zero trace'' (see \cite[Theorem~3.40]{McLean}  for the case when $\partial \Omega$ is bounded), but this intuition fails for negative $s$: if $\bx_0\in\partial\Omega$, then the delta function $\delta_{\bx_0}$ lies in $\tH^s\OO$ for $s<-n/2$, irrespective of the regularity of $\partial\Omega$; see the proof of Corollary~\ref{cor:Hs0HsEqual}(iv) below.

\begin{rem}
\label{rem:ZeroExtension}
We note that for $s\geq 0$ the restriction of $\cirHs(\Omega)$ to $\Omega$ is precisely the subspace (not necessarily closed)
\begin{align*}%
\Hze^s(\Omega):= \big\{u\in H^s(\Omega): \uze\in H^s(\R^n)\big\}\subset H^s(\Omega),
\end{align*}
where $\uze$ is the extension of $u$ from $\Omega$ to $\R^n$ by zero.
The restriction operator $|_\Omega:\cirHs(\Omega) \to \Hze^s(\Omega)$ is clearly a bijection for all $s\geq 0$, with inverse given by the map $u\mapsto \uze$, and if $\Hze^s(\Omega)$ is equipped with the norm $\|u\|_{\Hze^s(\Omega)}:=\|\uze\|_{H^s(\R^n)}$ (as in e.g.\ \cite[Equation~(1.3.2.7)]{Gri}, where $\Hze^s(\Omega)$ is denoted $\tilde W_2^s(\Omega)$) then $|_\Omega:\cirHs(\Omega) \to \Hze^s(\Omega)$ is trivially a unitary isomorphism for all $s\geq 0$.
\end{rem}

For clarity, we repeat a fundamental fact: the natural norm on $H^s_F$, $\tH^s(\Omega)$, $\cirHs\OO$ and $H^s_{\overline{\Omega}}$ is the $H^s\Rn$-norm (defined in \eqref{eq:HsProdNorm}), while the norm on $H^s\OO$ and $H^s_0\OO$ is the minimal $H^s\Rn$-norm among the extensions of $u\in H^s\OO$ to $\R^n$ (defined in \eqref{eq:InfNorm}).
\subsection{Dual spaces}
\label{subsec:DualAnnih}

In this section we construct concrete unitary realisations (as Sobolev spaces) of the duals of the Sobolev spaces defined in \S\ref{subsec:SobolevDef}.
Our constructions are based on the abstract Hilbert space result of Lemma \ref{lem:hs_orth}, and are valid for any non-empty open set $\Omega\subset\R^n$, irrespective of its regularity.

We first note the following lemma, which characterises the annihilators (as defined in \rf{eq:AnnihilatorDef}) of the subsets $\tH^s(\Omega)$ and $H^s_{\Omega^c}$ of $H^s(\R^n)$, with
$(H^s(\R^n))^*$ realised as $H^{-s}(\R^n)$ through the unitary isomorphism $\cI^s=R_s\cJ_{-2s}$ (see \S\ref{subsec:dual1}) with associated duality pairing \rf{DualDef}.

\begin{lem} \label{lem:orth_lem} Let $\Omega$ be any non-empty open subset of $\R^n$, and $s\in\R$. %
Then
\begin{align}
\label{eqn:anni}
H^{-s}_{\Omega^c} = \left(\tH^{s}(\Omega)\right)^{a,H^{-s}(\R^n)}
\qquad
\textrm{ and }
\qquad
\tH^{-s}(\Omega) = \left(H^s_{\Omega^c}\right)^{a,H^{-s}(\R^n)}.
\end{align}
Furthermore, the Bessel potential operator is a unitary isomorphism between the following pairs of subspaces:
\[\cJ_{2s}:\tH^s(\Omega)\to(H^{-s}_{\Omega^c})^{\perp} \qquad \textrm{and} \qquad
\cJ_{2s}: H^s_{\Omega^c} \to (\tH^{-s}\OO)^{\perp}.\]
\end{lem}
\begin{proof}
From the definition of the support of a distribution, \eqref{dualequiv}, the definition of $\tH^{s}(\Omega)$, and the continuity of the sesquilinear form $\langle\cdot,\cdot\rangle_s$, it follows that, for $s\in \R$,
\begin{align*}
H^{-s}_{\Omega^c} &= \{u\in H^{-s}(\R^n): \supp(u)\subset {\Omega^c}\} \\ &= \{u\in H^{-s}(\R^n): u(v) = 0 \mbox{ for all } v\in \scrD(\Omega)\}\\
& =  \{u\in H^{-s}(\R^n): \langle u,v\rangle_s = 0 \mbox{ for all } v\in \scrD(\Omega)\}
= \left(\tH^{s}(\Omega)\right)^{a,H^{-s}(\R^n)},
\end{align*}
which proves the first statement in \rf{eqn:anni}. The second statement in \rf{eqn:anni} follows immediately from the first, after replacing $s$ by $-s$, by \rf{eq:AnnihilatorResult}. The final statement of the lemma also follows by \eqref{eq:AnnihilatorResult}, noting that $j$ in \eqref{eq:AnnihilatorResult} is given explicitly as $j=(\cI^s)^{-1}R_s = \cJ_{2s}$.
\end{proof}

Combining Lemma~\ref{lem:orth_lem} with Lemmas \ref{dual_lem} and \ref{lem:hs_orth} gives unitary realisations for $(\tH^s(\Omega))^*$ and $(H^{-s}(\Omega))^*$, expressed in Theorem~\ref{thm:DualityTheorem} below.
These unitary realisations, precisely the result that the operators $\cI_s$ and $\cI_s^*$ in \rf{def_embed1App} are unitary isomorphisms, are well known when $\Omega$ is sufficiently regular.
For example,
in \cite[Theorem~3.30]{McLean} and in \cite[Theorem~2.15]{Steinbach} the result
is claimed for $\Omega$ Lipschitz with bounded boundary.
(In fact, \cite[Theorems~3.14 and 3.29(ii)]{McLean} together imply the result when $\Omega$ is $C^0$ with bounded boundary, but this is not highlighted in \cite{McLean}.)
However, it is not widely  appreciated, at least in the numerical PDEs community, that this result holds without any constraint on the geometry of $\Omega$.

\begin{thm}\label{thm:DualityTheorem}
Let $\Omega$ be any non-empty open subset of $\R^n$, and $s\in\R$. Then %
\begin{align}
\label{isdual}
H^{-s}(\Omega)\cong_{\cI_s} \big(\tH^s(\Omega)\big)^* \; \mbox{ and }\;
\tH^{s}(\Omega)\cong_{\cI_s^*}\big(H^{-s}(\Omega)\big)^*,
\end{align}
where $\cI_s:H^{-s}(\Omega)\to(\tH^s(\Omega))^*$ and $\cI_s^*:\tH^{s}(\Omega)\to(H^{-s}(\Omega))^*$, defined by
\begin{align}
\label{def_embed1App}
\cI_s u (v)= \langle U,v \rangle_{s}
\;\; \mbox{ and } \;\;
\cI_s^*v(u) = \langle v,U \rangle_{-s}, %
\quad \mbox{ for } u\in H^{-s}(\Omega), \,v\in\tH^s(\Omega),
\end{align}
where $U\in H^{-s}(\R^n)$ denotes \textit{any} extension of $u$ with $U|_\Omega=u$, %
are unitary isomorphisms.
Furthermore, the associated duality pairings %
\begin{equation*} %
\langle u,v \rangle_{H^{-s}(\Omega)\times \tH^{s}(\Omega)} := \cI_s u(v) \qquad \mbox{ and } \qquad \langle v,u \rangle_{\tH^{s}(\Omega)\times {H}^{-s}(\Omega)}:= \cI_s^*v(u),
\end{equation*}
satisfy
$$
\langle v,u \rangle_{\tH^{s}(\Omega)\times H^{-s}(\Omega)} = \overline{\langle u,v \rangle}_{H^{-s}(\Omega)\times \tH^s(\Omega)}, \quad v\in \tH^{s}(\Omega), \; u\in H^{-s}(\Omega).
$$
\end{thm}
\begin{proof}
By Lemma~\ref{lem:orth_lem}, it follows from Lemma \ref{lem:hs_orth}, applied with $H=H^s(\R^n)$, $\cH = H^{-s}(\R^n)$ and $V=\tH^s(\Omega)$, that
$\cIhat_s:(H^{-s}_{\Omega^c})^\perp \to (\tH^s(\Omega))^*$, defined by $\cIhat_s u(v) = \langle u,v\rangle_s$, is a unitary isomorphism.
By Lemma \ref{dual_lem}, $\cIhat_s^*:\tH^s(\Omega) \to ((H^{-s}_{\Omega^c})^\perp)^*$, defined by $\cIhat_s^* v(u) = \langle v,u\rangle_{-s} =\overline{\cIhat_s u(v)}$ is also a unitary isomorphism.
Thus the dual space of $\tH^s(\Omega)$ can be realised in a canonical way %
by $(H^{-s}_{\Omega^c})^\perp$, and vice versa.
But we can say more. Since (cf.\ \eqref{eq:RestrIsUnitary}) the restriction operator $|_\Omega$ is a unitary isomorphism from $(H^{-s}_{\Omega^c})^\perp$ onto $H^{-s}(\Omega)$, the composition $\cI_s:=\cIhat_s(|_\Omega)^{-1}:H^{-s}(\Omega) \to (\tH^s(\Omega))^*$ is a unitary isomorphism.
And, again by Lemma \ref{dual_lem}, $\cI_s^*:\tH^s(\Omega) \to (H^{-s}(\Omega))^*$, defined by $\cI_s^* v(u) := \overline{\cI_s u(v)}$ is also a unitary isomorphism. Hence we can realise the dual space of $\tH^s(\Omega)$ by $H^{-s}(\Omega)$, and vice versa.
Moreover, it is easy to check that $\cI_s$ and $\cI_s^*$ can be evaluated as in \rf{def_embed1App}.
Thus $\cI_s$ and $\cI_s^*$ coincide with the natural embeddings of $H^{-s}(\Omega)$ and $\tH^{s}(\Omega)$ into $(\tH^s(\Omega))^*$ and $(H^{-s}(\Omega))^*$, respectively (as in e.g.\ \cite[Theorem 3.14]{McLean}).
\end{proof}

\begin{cor}\label{cor:DualityTheorem2}
Let $F$ be any closed subset of $\R^n$ (excepting $\R^n$ itself), and $s\in\R$.
Then
\begin{align*}%
\big(\tH^{-s}(F^c)\big)^\perp\cong_{\cItilde_s}(H^s_{F})^* \; \mbox{ and }\; H^s_{F} \cong_{\cItilde_s^*}\Big( \big(\tH^{-s}(F^c)\big)^\perp\Big)^*,
\end{align*}
where $\cItilde_s: (\tH^{-s}(F^c))^\perp \to (H^s_{F})^*$ and $\cItilde_s^*:H^s_{F} \to ( (\tH^{-s}(F^c))^\perp)^*$, defined by
\begin{align*}%
\cItilde_s u(v):=\langle u,v\rangle_s, \; \mbox{ and }\; \cItilde_s^* v(u) = \langle v,u\rangle_{-s} =\overline{\cItilde_s u(v)},
\end{align*}
for $u\in \big(\tH^{-s}(F^c)\big)^\perp$ and $v\in H^s_{F}$,
are unitary isomorphisms.
\end{cor}
\begin{proof}
Setting $\Omega:=F^c$, the result follows from Theorem \ref{thm:DualityTheorem} and its proof and Remark \ref{rem:orth}.
\end{proof}

\begin{rem}
It is also possible to realise $(\tH^s(\Omega))^*$ and $(H^s_{F})^*$ using quotient spaces, by composition of $\cIhat_s$ and $\cItilde_s$ with the appropriate quotient maps.
For example, $(\tH^s(\Omega))^*$ can be realised as $(H^{-s}(\R^n)/H^{-s}_{\Omega^c},\cIcheck_s)$, where $\cIcheck_s=\cIhat_s {Q_{-s}}_/ = \cI_s q_{-s}$, and $q_s$ and ${Q_s}_/$ are defined as in \S\ref{subsec:SobSpacesClosedOpen}.
\end{rem}

\begin{rem}
\label{rem:DualOfCircle}
Corollary \ref{cor:DualityTheorem2}, coupled with Remark \ref{rem:orth} or with the results in the proof of Theorem \ref{thm:DualityTheorem},
implies that, for a non-empty open set $\Omega$, $(\tH^s(\Omega))^*$ and $(H^s_{\overline{\Omega}})^*$ can be canonically realised as subspaces of $H^{-s}(\R^n)$, namely as $(H^{-s}_{\Omega^c})^\perp$ and $(\tH^{-s}(\overline{\Omega}^c))^\perp$ respectively. For $s\geq 0$, we know that $(\cirHs\OO)^*$ can similarly be realised as the subspace $(X^{-s}(\Omega))^\perp\subset H^{-s}(\R^n)$, where $\tH^{-s}(\overline{\Omega}^c)\subset X^{-s}(\Omega):=(\cirHs\OO)^{a,H^{-s}(\R^n)}\subset H^{-s}_{\Omega^c}$. But, as far as we know, providing an explicit description of the space $X^{-s}(\Omega)\subset H^{-s}(\R^n)$ is an  open problem.
\end{rem}

The following lemma realises the dual space of $H^s_0(\Omega)\subset H^s(\Omega)$ as a subspace of $\tH^{-s}(\Omega)$.
\begin{lem}
\label{lem:Hs0Dual}
Let $\Omega$ be any non-empty open subset of $\R^n$ and $s\in\R$.
Then the dual space of $H^s_0(\Omega)$ can be unitarily realised as
$(\tH^{-s}(\Omega) \cap H^{-s}_{\partial\Omega})^{\perp, \tH^{-s}(\Omega)}$,
with the duality pairing inherited from ${\tH}^{-s}(\Omega) \times H^s(\Omega)$.
\end{lem}
\begin{proof}
Since $H^s_0(\Omega)$ is a closed subspace of $H^s(\Omega)$, by Lemma \ref{lem:hs_orth} $(H^s_0(\Omega))^*$ can be unitarily realised as a closed subspace of $(H^s(\Omega))^*$, which we identify with $\tH^{-s}(\Omega)$ using the operator $\cI^*_{-s}$ of Theorem~\ref{thm:DualityTheorem}.
Explicitly, $(H^s_0(\Omega))^*$ is identified with the orthogonal complement of the annihilator of $H^s_0(\Omega)$ in $\tH^{-s}(\Omega)$, which annihilator satisfies
\begin{align*}
H^s_0(\Omega)^{a,\tH^{-s}(\Omega)}
&=\big(\scrD(\Omega)|_\Omega\big)^{a,\tH^{-s}(\Omega)}
=\tH^{-s}(\Omega) \cap\big(\scrD(\Omega)\big)^{a,H^{-s}(\R^n)}\\
&=\tH^{-s}(\Omega)\cap H^{-s}_{\Omega^c}
=\tH^{-s}(\Omega) \cap H^{-s}_{\partial\Omega}.
\end{align*}
\end{proof}

\begin{table}[tb!]
\begin{center}\begin{tabular}{c|c|c}
The dual of & is isomorphic to & via the isomorphism\\ \hline &&\\[-3mm]
$H^s(\R^n)$ & $H^{-s}(\R^n)$ & $\cI^s$\\
$\tH^s(\Omega)$ & $(H^{-s}_{\Omega^c})^\perp$ & $\cIhat_s$\\ %
& $H^{-s}(\Omega)$ & $\cI_s%
$\\
&$H^{-s}(\R^n)/(H^{-s}_{\Omega^c})%
$ & $\cIcheck_s%
$\\
$H^s(\Omega)$ & $\tH^{-s}(\Omega)$ & $\cI_{-s}^*$\\
$H^{s}_{\Omega^c}$ & $(\tH^{-s}(\Omega))^\perp$ &$\cItilde_s$\\
$(H^{s}_{\Omega^c})^\perp$ & $\tH^{-s}(\Omega)$ &$\cIhat_{-s}^*$\\
$\big(\tH^{s}(\Omega)\big)^\perp$ & $H^{-s}_{\Omega^c}$ & $\cItilde_{-s}^*$\\
$H^s_0(\Omega)$ &
$(\tH^{-s}(\Omega) \cap H^{-s}_{\partial\Omega})^{\perp, \tH^{-s}(\Omega)}$
\end{tabular}\end{center}
\caption{A summary of the duality relations proved in \S\ref{subsec:dual1} and \S\ref{subsec:DualAnnih}.\label{tab:duals}}
\end{table}

\begin{figure}[htb!]
\begin{center}
\begin{tikzpicture}
\hspace{-15mm}
\matrix[matrix of math nodes,
column sep={12pt},
row sep={40pt,between origins}, %
text height=1.5ex, text depth=0.25ex] (s)
{
|[name=DO]|  \scrD(\Omega)&
|[name=DoO]|  &%
|[name=DR]|  \scrD(\R^n)&
|[name=SR]|  \scrS(\R^n)&&
|[name=L2R]| L^2(\R^n)
\\
|[name=tH]| \tH^s(\Omega)  &
|[name=cH]| \cirHs(\Omega) &
|[name=sH]| H^s_{\overline{\Omega}} &
|[name=HR]| \hspace{3mm}H^s(\R^n)\hspace{2mm} =\hspace{-5mm} &
|[name=Pr]|  (H^s_{\Omega^c})^\perp \; \oplus \; H^s_{\Omega^c}
&
|[name=SsR]|  \scrS^*(\R^n)
\\
& |[name=H0O]| H^s_0(\Omega)  &&
|[name=HO]| H^s(\Omega)  &
|[name=DsO]| \scrD^*(\Omega) &
\\
|[name=HOs]| \big(H^{-s}(\Omega)\big)^* &
|[name=HmR]| H^{-s}(\R^n)&
|[name=HRs]| \big(H^{-s}(\R^n)\big)^* &
|[name=tHOs]| \big(\tH^{-s}(\Omega)\big)^*&&
|[name=QOs]| %
\hspace{-2mm}\big((\tH^{-s}(\Omega))^\perp\big)^*
\\
};
\draw[right hook->]
  (DO) edge (DR)
  (DR) edge (SR)
  (SR) edge (L2R)
  (tH) edge (cH)
  (cH) edge (sH)
  (sH) edge (HR)
  (Pr) edge (SsR)   %
  (HO) edge  (DsO)
  (H0O) edge (HO)
  (DO) edge node[auto] {\(\iota\)} (tH)
  (SR) edge node[auto] {\(\iota\)} (HR)
  (L2R) edge node[auto] {\(\iota\)} (SsR)
    (tH) edge node[swap] {\(\hs{-8}|_\Omega\)} (H0O)
  ;
\draw[->>]
  (HR) edge node[auto] {\(\hs{-0.7}|_\Omega\)} (HO)
;
\draw[right hook->>]
  (tH) edge node[auto] {\(\cI_{-s}^*\)} (HOs)
  (HO) edge node[auto,swap] {\(\cI_{-s}\)} (tHOs)
  (Pr.south east) to node[auto,pos=0.5] {\(\cItilde_s^*\hs{-2}\)} (QOs)
  (HmR) edge node[auto,swap] {\(R_{-s}\)} (HRs)
;
\draw[left hook->>]
  (Pr) edge node[pos=0.5] {\(\hs{-10}|_\Omega\)} (HO)  %
  (HR) edge node[pos=0.77] {\(\hs{9}\cI^{-s}\)}(HRs)
  (HR) edge node[pos=0.77] {\(\hs{11}\cJ_{2s}\)}  (HmR)
  (Pr) edge node[pos=0.77] {\(\hs{12}\cIhat_{-s}\)} (tHOs)
;
\end{tikzpicture}
\end{center}
\caption{A representation, as a commutative diagram, of the relationships between the Sobolev spaces and the isomorphisms between them described in \S\ref{subsec:SobolevDef} and \S\ref{subsec:DualAnnih}.
Here $s\in\R$, $\Omega\subset\R^n$ is open,
${\Omega^c}:=\R^n\setminus\Omega$,
$\hookrightarrow$ denotes an embedding,
$\twoheadrightarrow$ a surjective mapping, %
$\hookdoubleheadrightarrow$ a unitary isomorphism, and
$\iota$ denotes the standard identification of Lebesgue functions with distributions, namely $\iota:L^2\Rn\to\scrS^*\Rn$, with $\iota u(v):=(u,v)_{L^2\Rn}$, for $u\in L^2\Rn$, $v\in\scrS\Rn$.
Note that $\cirHs(\Omega)$ is defined only when $s\ge0$, see \S\ref{subsec:3spaces}.
In this diagram the first row contains spaces of functions, the second distributions on $\R^n$, and the third distributions on $\Omega$.
\label{fig:Sobolev}}
\end{figure}
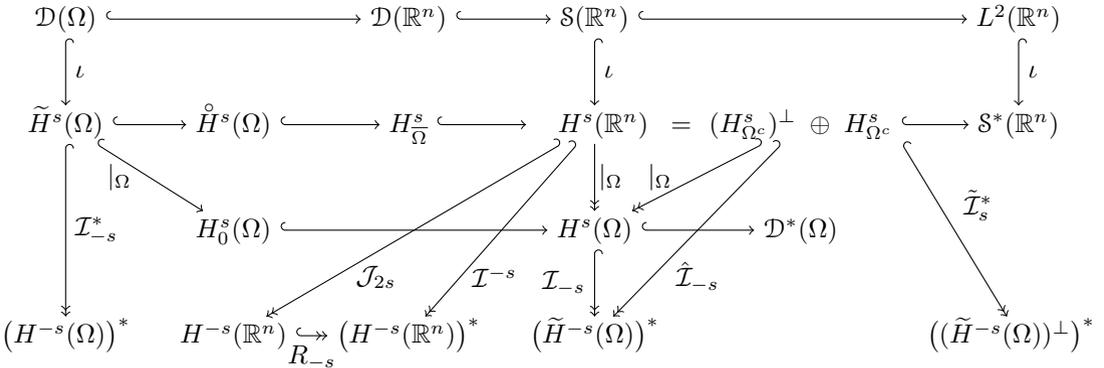
\subsection{\texorpdfstring{$s$}{s}-nullity}
\label{subsec:Polarity}
In order to compare Sobolev spaces defined on different open sets (which we do in \S\ref{subsec:DiffDoms}), and to study the relationship between the different spaces (e.g.\ $\tH^s(\Omega)$, $\cirHs(\Omega)$ and $H^s_{\overline{\Omega}}$) on a given open set $\Omega$ (which we do in \S\ref{subsec:3spaces}), we require the concept of $s$-nullity of subsets of $\R^n$. %
\begin{defn}
For $s\in\R$ we say that a set $E\subset\R^n$ is $s$-null if there are no non-zero elements of $H^{s}(\R^n)$ supported entirely in $E$ (equivalently, if $H^{s}_{F}=\{0\}$ for every closed set $F\subset E$).
\end{defn}
\noindent We make the trivial remark that if $F$ is closed then $F$ is $s$-null if and only if $H^s_F=\{0\}$.
\begin{rem}\label{rem:polarity}
While the terminology ``$s$-null'' is our own, the concept it describes has been studied previously, apparently first by H\"{o}rmander and Lions in relation to properties of Sobolev spaces normed by Dirichlet integrals \cite{HoLi:56}, and then subsequently by other authors in relation to the removability of singularities for elliptic partial differential operators \cite{Li:67a,Maz'ya}, and to the approximation of functions by solutions of the associated elliptic PDEs \cite{Po:72}.
For integer $s<0$, $s$-nullity is referred to as $(-s)$-polarity in \cite[Definition 2]{HoLi:56}, ``$2$-$(-s)$ polarity'' in \cite{Li:67a} and ``$(2,-s)$-polarity'' in \cite[\S 13.2]{Maz'ya}. For $s>0$ and $E$ closed, $s$-nullity coincides with the concept of ``sets of uniqueness'' for $H^s(\R^n)$, as considered in \cite[\S11.3]{AdHe} and \cite[p.~692]{Maz'ya}. For $s>0$ and $E$ with empty interior, $s$-nullity coincides with the concept of $(s,2)$-stability, discussed in \cite[\S11.5]{AdHe}.
For a more detailed comparison with the literature see \cite[\S2.2]{HewMoi:15}.
\end{rem}

To help us throughout the paper interpret characterisations in terms of $s$-nullity, the following lemma collects useful results relating $s$-nullity to topological and geometrical properties of a set.
The results in Lemma \ref{lem:polarity} are a special case of those recently presented in \cite{HewMoi:15} (where $s$-nullity is called $(s,2)$-nullity) in the more general setting of the Bessel potential spaces $H^{s,p}(\R^n)$, $s\in\R$, $1<p<\infty$. Many results in \cite{HewMoi:15} are derived using the equivalence between $s$-nullity and the vanishing of certain set capacities from classical potential theory, drawing heavily on results in \cite{AdHe} and \cite{Maz'ya}.
\cite{HewMoi:15} also contains a number of concrete examples and counterexamples illustrating the general results.
Regarding point \rf{gg0} of the lemma, following \cite[\S3]{Triebel97FracSpec}, given $0\leq d\leq n$ we call a closed set $F\subset \R^n$ with $\dimH(F)=d$ a $d$-set if there exist constants $c_1,c_2>0$ such that
\begin{align}
\label{eq:dSet}
0<c_1 r^d \leq \mathcal H^d(B_r(\bx)\cap F) \leq c_2 r^d<\infty, \qquad \textrm{for all } \bx\in F, \; 0<r<1,
\end{align}
where $\mathcal H^d$ is the $d$-dimensional Hausdorff measure on $\R^n$. 
Condition \eqref{eq:dSet} may be understood as saying that $d$-sets are everywhere locally $d$-dimensional.
Note that the definition of $d$-set includes as a special case all Lipschitz $d$-dimensional manifolds, $d\in\{0,1,\ldots, n\}$.
\begin{lem}[\cite{HewMoi:15}]
\label{lem:polarity}
Let $E,E'\subset \R^n$ be arbitrary, $\Omega\subset \R^n$ be non-empty and open, and $s\in\R$.
\begin{enumerate}[(i)]
\item \label{aa}If $E$ is $s$-null and $E'\subset E$ then $E'$ is $s$-null.
\item \label{bb}If $E$ is $s$-null and $t>s$ then $E$ is $t$-null.
\item \label{cc}If $E$ is $s$-null then ${\rm int}(E)=\emptyset$.
\item \label{dd}If $s>n/2$ then $E$ is $s$-null if and only if ${\rm int}(E)=\emptyset$.
\item \label{ll} Let $E$ be $s$-null and let $F\subset\R^n$ be closed and $s$-null. Then $E\cup F$ is $s$-null.
\item \label{ll0} If $s\leq 0$ then a countable union of Borel $s$-null sets is $s$-null.
\item \label{ee}If $s\geq0$ and $E$ is Lebesgue-measurable with $m(E)=0$, then $E$ is $s$-null.
\item \label{ff}If $E$ is Lebesgue-measurable then $E$ is $0$-null if and only if $m(E)=0$.
\item \label{mm} There exists a compact set $K\subset\R^n$ with ${\rm int}(K)=\emptyset$ and $m(K)>0$, which is not $s$-null for any $s\leq n/2$.
\item \label{jj}If $s<-n/2$ there are no non-empty $s$-null sets.
\item \label{ii} A non-empty countable set is $s$-null if and only if $s\ge-n/2$.
\item \label{hh} If $-n/2<s\leq 0$ and ${\rm dim_H}(E)< n+2s$, then $E$ is $s$-null.
\item \label{gg} If $-n/2\leq s<0$ and $E$ is Borel and $s$-null, then ${\rm dim_H}(E)\leq n+2s$.
\item \label{gg1} For each $0\leq d\leq n$ there exist compact sets $K_1,K_2\subset \R^n$ with $\dimH(K_1)$ $=\dimH(K_2)=d$, such that $K_1$ is $(d-n)/2$-null and $K_2$ is not $(d-n)/2$-null.
\item \label{gg0} If $0<d<n$ and $F\subset \R^n$ is a compact $d$-set, or a $d$-dimensional hyperplane (in which case $d$ is assumed to be an integer) then $F$ is $(d-n)/2$-null.
\item \label{kk1}
If ${\rm int}(\Omega^c)\neq \emptyset$, then $\deO$ is not $s$-null for $s<-1/2$. (In particular this holds if $\Omega\neq \R^n$ is $C^0$.)
\item \label{kk0}
If $\Omega$ is $C^0$ and $s\geq 0$, then $\partial\Omega$ is $s$-null.
Furthermore, for $n\geq2$ there exists a bounded $C^0$ open set whose boundary is not $s$-null for any $s<0$.
\item \label{kk2}
If $\Omega$ is $C^{0,\alpha}$ for some $0<\alpha<1$ and $s> -\alpha/2$, then $\partial\Omega$ is $s$-null.
Furthermore, for $n\geq2$ there exists a bounded $C^{0,\alpha}$ open set whose boundary is not $s$-null for any $s<-\alpha/2$.
\item \label{kk} If $\Omega$ is Lipschitz then $\partial\Omega$ is $s$-null if and only if $s\geq -1/2$.
\end{enumerate}
\end{lem}

\subsection{Equality of spaces defined on different subsets of  \texorpdfstring{$\R^n$}{Rn}}\label{subsec:DiffDoms}
The concept of $s$-nullity defined in \S\ref{subsec:Polarity} provides a characterization of when Sobolev spaces defined on different open or closed sets are or are not equal. For two subsets $E_1$ and $E_2$ of $\R^n$ we use the notation $E_1\ominus E_2$ to denote the symmetric difference between $E_1$ and $E_2$, i.e.
\begin{align*}%
E_1\ominus E_2:=(E_1\setminus E_2)\cup( E_2\setminus E_1)=(E_1\cup E_2)\setminus (E_1\cap E_2).
\end{align*}

The following elementary result is a special case of \cite[Proposition 2.11]{HewMoi:15}.
\begin{thm}[{\cite[Proposition 2.11]{HewMoi:15}}]
\label{thm:Hs_equality_closed}
Let $F_1,F_2$ be closed subsets of $\R^n$, and  let $s\in\R$. Then the following statements are equivalent:
\begin{enumerate}[(i)]
\item \label{a}$F_1\ominus F_2$ is $s$-null.
\item \label{b} $F_1\setminus F_2$ and $ F_2\setminus F_1$ are both $s$-null.
\item \label{c} $H^s_{F_1}=H^s_{F_2}$.
\end{enumerate}
\end{thm}

By combining Theorem \ref{thm:Hs_equality_closed} with the duality result of Theorem \ref{thm:DualityTheorem} one can deduce a corresponding result about spaces defined on open subsets. The following theorem generalises \cite[Theorem~13.2.1]{Maz'ya}, which concerned the case $\Omega_1\subset\Omega_2=\R^n$.
The special case where $\R^n\setminus\Omega_1$ is a $d$-set was considered in \cite{Tri:08}. (That result was used in \cite{HewMoi:15} to prove item~\rf{gg0} in Lemma~\ref{lem:polarity} above.)
\begin{thm}
\label{thm:Hs_equality_open}
Let $\Omega_1,\Omega_2$ be non-empty, open subsets of $\R^n$, and let $s\in\R$. Then the following statements are equivalent:
\begin{enumerate}[(i)]
\item \label{a0}$\Omega_1\ominus\Omega_2$ is $s$-null.
\item \label{b0}$\Omega_1\setminus\Omega_2$ and $\Omega_2\setminus\Omega_1$ are both $s$-null.
\item \label{c0} $H^{s}(\Omega_1)=H^{s}(\Omega_2)$, in the sense that
$\!\big( H^s_{\Omega_1^c}\big)^\perp \!
=\!\big( H^s_{\Omega_2^c}\big)^\perp \!$
(recall from \eqref{eq:RestrIsUnitary} that $(H^s_{\Omega^c})^\perp\cong H^s(\Omega)$ for any non-empty open $\Omega\subset \R^n$).
\item \label{d0} $\tH^{-s}(\Omega_1)=\tH^{-s}(\Omega_2)$.
\end{enumerate}
\end{thm}
\begin{proof}
The result follows from Theorem \ref{thm:DualityTheorem} and Theorem \ref{thm:Hs_equality_closed} with $F_j:=(\Omega_j)^c$, $j=1,2$.
\end{proof}

\begin{rem}
For non-empty open $\Omega_1,\Omega_2\subset \R^n$, the set $\Omega_1\ominus\Omega_2$ has empty interior if and only if
$\overline{\Omega_1} =\overline{\Omega_2}$.
Hence, by Lemma \ref{lem:polarity}\rf{cc},\rf{dd}, $\overline{\Omega_1} =\overline{\Omega_2}$ %
 is a necessary condition for the statements \rf{a0}--\rf{d0} of Theorem \ref{thm:Hs_equality_open} to hold, and a sufficient condition when $s>n/2$.
But sufficiency does not extend to $s\leq n/2$: a counter-example is provided by $\Omega_1=\R^n$ and $\Omega_2=K^c$, where $K$ is any compact non-$(n/2)$-null set (cf.\ Lemma \ref{lem:polarity}\rf{mm}).
\end{rem}

For the $\cirHs(\Omega)$ spaces, $s\geq0$, the following sufficient (but not necessary) condition for equality is trivial.
\begin{lem}
\label{lem:CircEquality}
If $\Omega_1,\Omega_2\subset \R^n$ are non-empty and open, with $m(\Omega_2\ominus \Omega_1)=0$, then $\cirHs(\Omega_1)=\cirHs(\Omega_2)$ for all $s\geq 0$.
\end{lem}

\subsection{Comparison of the ``zero trace'' subspaces of \texorpdfstring{$H^s(\R^n)$}{Hs(Rn)}}
\label{subsec:3spaces}

In \S\ref{subsec:SobSpacesClosedOpen} we defined three closed subspaces of $H^s(\R^n)$ associated with a non-empty open set $\Omega\subset\R^n$, namely $H^s_{\overline{\Omega}}$ and $\tH^s(\Omega)$ (both defined for all $s\in \R$) and $\cirHs(\Omega)$ (defined for $s\geq0$), which can all be viewed in some sense as ``zero trace'' spaces. %
We already noted (cf.\ \rf{eqn:inclusions}) the inclusions
\begin{align}\label{eqn:inclusionsRepeat}
\tH^s(\Omega)\subset \cirHs(\Omega) \subset H^s_{\overline{\Omega}},
\end{align}
for all $s\in\R$ (with $\cirHs(\Omega)$ present only for $s\geq 0$).
In this section we investigate conditions on $\Omega$ and $s$ under which the inclusions in \rf{eqn:inclusionsRepeat} are or are not equalities, and construct explicit counterexamples demonstrating that equality does not hold in general.

When $\Omega$ is a $C^0$ open set, both inclusions in \rf{eqn:inclusionsRepeat} are equalities.
The following result is proved in \cite[Theorem~3.29]{McLean} for $C^0$ sets with bounded boundary\footnote{We note however that the partition of unity argument appears not quite accurate in the proof of \cite[Theorem~3.29]{McLean}. For an alternative method of handling this part of the argument see the proof of Theorem \ref{thm:new2} below.}; the extension to general $C^0$ sets (as defined in \cite[Definition~1.2.1.1]{Gri}) follows from \eqref{eq:approx2} (cf.\ the proof of Theorem \ref{thm:new2} below).
We note that a proof of the equality $\tH^s(\Omega) = \cirHs(\Omega)$ for $s>0$ and $\Omega$ a $C^0$ open set can also be found in \cite[Theorem 1.4.2.2]{Gri}.
\begin{lem}[{\cite[Theorems 3.29, 3.21]{McLean}}]
\label{lem:sob_equiv}
Let $\Omega\subset \R^n$ be $C^0$ and let $s\in\R$.
Then $\tH^s(\Omega)= \cirHs(\Omega) =H^s_{\overline\Omega}$ (with $\cirHs(\Omega)$ present only for $s\geq 0$).
\end{lem}

When $\Omega$ is not $C^0$ the situation is more complicated.
We first note the following elementary results concerning the case $s\geq 0$, part \rf{c1} of which makes it clear that Lemma \ref{lem:sob_equiv} does not extend to general open $\Omega$.

\begin{lem}\label{lem:CircleSpace}
Let $\Omega\subset\R^n$ be non-empty and open. Then
\begin{enumerate}[(i)]
\item \label{c1} $\tH^0(\Omega)=\cirHzero(\Omega)$; while
$\cirHzero(\Omega) = H^0_{\overline{\Omega}}$ if and only if $m(\partial \Omega)=0$.
\item \label{d1} For $s\geq 0$, if $\meas(\partial \Omega)=0$ then $\cirHs(\Omega)= H^s_{\overline{\Omega}}$.
\item \label{g1} For $t>s\geq0$, if $\cirHs(\Omega) = H^{s}_{\overline{\Omega}}$ then $\cirHt(\Omega) = H^{t}_{\overline{\Omega}}$.%
\end{enumerate}
\end{lem}
\begin{proof} %
\rf{c1} The equality $\tH^0(\Omega)=\cirHzero(\Omega)$ holds because the restriction operator is a unitary isomorphism from $\cirHzero(\Omega)$ onto $H^0(\Omega) = L^2(\Omega)$, in particular $\|u\|_{L^2(\R^n)}=\|u|_\Omega\|_{L^2(\Omega)}$ for $u\in \cirHzero(\Omega)$, and because $\scrD(\Omega)$ is dense in $L^2(\Omega)$ \cite[Theorem 2.19]{Adams}.
The second statement in \rf{c1}, and \rf{d1}, follow straight from the definitions.
If the hypothesis of part \rf{g1} is satisfied, then every $u\in H^{t}_{\overline{\Omega}}\subset H^{s}_{\overline{\Omega}}\cap H^{t}(\R^n)=\cirHs(\Omega)\cap H^{t}(\R^n)$ is equal to zero a.e.\ in $\Omega^c$, %
and hence belongs to $\cirHt(\Omega)$.
\end{proof}

Open sets for which
$\Omega\subsetneqq \mathrm{int}(\overline{\Omega})$ are a source of counterexamples to equality in \rf{eqn:inclusionsRepeat}.
The following lemma relates properties of the inclusions \rf{eqn:inclusionsRepeat} to properties of the set $\mathrm{int}(\overline{\Omega})\setminus\Omega$.
\begin{lem}
\label{lem:equalityNullity}
Let $\Omega\subset \R^n$ be non-empty and open, and let $s\in \R$.%
\begin{enumerate}[(i)]
\item \label{a7}
For $s\geq 0$, if $\meas(\mathrm{int}(\overline{\Omega})\setminus \Omega)>0$ then $\cirHs(\Omega) \subsetneqq H^s_{\overline{\Omega}}$.
\item \label{b7}
For $s>n/2$, $\cirHs(\Omega) = H^s_{\overline{\Omega}}$ if and only if $m(\mathrm{int}(\overline{\Omega})\setminus \Omega)=0$.
\item \label{aaa} If
$\IntClosOm\setminus\Omega$ is not $(-s)$-null then $\tH^s(\Omega)\subsetneqq H^s_{\overline\Omega}$.
\item \label{bbb} If $\mathrm{int}(\overline{\Omega})\setminus\Omega$ is not $(-s)$-null, $s>0$, and $m(\mathrm{int}(\overline{\Omega})\setminus\Omega)=0$, then $\tH^s\OO\subsetneqq\cirHs\OO$.
\item \label{ddd}
If $\tH^s(\mathrm{int}(\overline{\Omega})) = H^s_{\overline{\Omega}}$ (e.g.\ if $\mathrm{int}(\overline{\Omega})$ is $C^0$),
then $\tH^s(\Omega) = H^s_{\overline{\Omega}}$ if and only if $\mathrm{int}(\overline{\Omega})\setminus\Omega$ is $(-s)$-null.
\end{enumerate}
\end{lem}

\begin{proof}
\rf{a7} If $\meas(\mathrm{int}(\overline{\Omega})\setminus \Omega)>0$ then there exists an open ball $B\subset \mathrm{int}(\overline{\Omega})$ such that $\meas(B\setminus \Omega)=\epsilon>0$.
(To see this first write $\mathrm{int}(\overline{\Omega})$ as the union of balls. Then use the fact that $\R^n$ is a separable metric space, so second countable, so that, by Lindel\"of's theorem (see e.g.\ \cite[p.~100]{Simmons}), $\mathrm{int}(\overline{\Omega})$ can be written as the union of a countable set of balls, i.e., as $\mathrm{int}(\overline{\Omega})=\bigcup_{n=1}^\infty B_n$. Then $0<m(\mathrm{int}(\overline{\Omega})\setminus \Omega)\leq \sum_{n=1}^\infty m(B_n\setminus \Omega)$, so that $m(B_n\setminus \Omega)>0$ for some $n$.) Choose $\chi\in \scrD(B)$ such that $0\leq \chi\leq 1$ and $\int \chi dx > m(B)-\epsilon$. Then $\chi \in \tH^s(\mathrm{int}(\overline{\Omega}))\subset H^s_{\overline{\Omega}}$, but $\chi\not\in \cirHs(\Omega)$, for if $\chi\in \cirHs(\Omega)$ then $\chi=0$ a.e. in $\Omega^c$, so that $\int \chi dx \leq m(B\cap \Omega)\leq m(B)-\epsilon$.
\rf{b7} If $u\in H^s_{\overline{\Omega}}$ then $u=0$ a.e.\ in $\overline\Omega^c$.
Since $s>n/2$, the Sobolev embedding theorem says that $u\in C^0\Rn$, so $u=0$ a.e.\ in $\overline{\overline{\Omega}^c}$.
But $\Omega^c\setminus\overline{\overline{\Omega}^c}=\mathrm{int}(\overline{\Omega})\setminus \Omega$, which has zero measure by assumption.
Thus $u=0$ a.e.\ in $\Omega^c$, so $u\in\cirHs\OO$.
The ``only if'' part of the statement is provided by \rf{a7}.
\rf{aaa} If $\IntClosOm\setminus\Omega$ is not $(-s)$-null then, by Theorem \ref{thm:Hs_equality_open}, $\tH^s(\Omega)\subsetneqq \tH^s(\IntClosOm) \subset H^s_{\overline\Omega}$.
Part \rf{bbb} follows similarly, by noting that $\tH^s\OO\subsetneqq \tH^s(\mathrm{int}(\overline{\Omega}))\subset
\cirHs(\mathrm{int}(\overline{\Omega}))=\cirHs\OO$, the latter equality following from Lemma \ref{lem:CircEquality}.
\rf{ddd} Lemma \ref{lem:sob_equiv} (applied to $\mathrm{int}(\overline{\Omega})$) implies that
$\tH^s\OO\subset\tH^s(\mathrm{int}(\overline{\Omega}))
=H^s_{\overline{\mathrm{int}(\overline{\Omega})}}
=H^s_{\overline\Omega}$,
and the assertion then follows by Theorem~\ref{thm:Hs_equality_open} (with $\Omega_1=\Omega$ and $\Omega_2=\mathrm{int}(\overline{\Omega})$).
\end{proof}

In particular, Lemma~\ref{lem:equalityNullity}\rf{ddd}, combined with Lemmas \ref{lem:sob_equiv} and \ref{lem:polarity}, provides results about the case where $\Omega$ is an $C^0$ open set from which a closed, nowhere dense set has been removed. A selection of such results is given in the following proposition.
\begin{prop}
\label{prop:TildeSubscript}
Suppose that $\Omega\subsetneqq \mathrm{int}(\overline{\Omega})$ and that $\mathrm{int}(\overline{\Omega})$ is  $C^0$. Then:
\begin{enumerate}[(i)]

\item \label{ts0} $\tH^s(\Omega)= H^s_{\overline\Omega}$ for all $s<-n/2$.

\item \label{ts1}
If $\mathrm{int}(\overline{\Omega})\setminus\Omega$ is a subset of the boundary of a Lipschitz open set $\Upsilon$, with $\mathrm{int}(\overline{\Omega})\setminus\Omega$ having non-empty relative interior in $\partial\Upsilon$, then $\tH^s(\Omega) = H^s_{\overline{\Omega}}$ if and only if $s\leq 1/2$.
(A concrete example in one dimension is where $\Omega$ is an open interval with an interior point removed.
An example in two dimensions is where $\Omega$ is an open disc with a slit cut out.
Three-dimensional examples relevant for computational electromagnetism are the ``pseudo-Lipschitz domains'' of \cite[Definition~3.1]{ABD98}.)

\item If $0<d:=\dimH(\mathrm{int}(\overline{\Omega})\setminus\Omega)<n$ then $\tH^s(\Omega) = H^s_{\overline{\Omega}}$ for all $s<(n-d)/2$ and $\tH^s(\Omega) \subsetneqq H^s_{\overline{\Omega}}$ for all $s>(n-d)/2$.

\item If $\mathrm{int}(\overline{\Omega})\setminus\Omega$ is countable then $\tH^s(\Omega) = H^s_{\overline{\Omega}}$ if and only if $s\leq n/2$.

\item If $\tH^t(\Omega)=H^t_{\overline\Omega}$ for some $t\in\R$ then $\tH^s(\Omega)=H^s_{\overline\Omega}$ for all $s<t$. %
(Whether the assumption that $\mathrm{int}(\overline\Omega)$ is $C^0$ is necessary here appears to be an open question.
Lemma \ref{lem:CircleSpace}\rf{g1} shows that if $\tH$ is replaced by $\cirH$ the opposite result holds (without assumptions on $\mathrm{int}(\overline\Omega)$)).
\end{enumerate}
\end{prop}

Parts \rf{aaa} and \rf{bbb} of Lemma \ref{lem:equalityNullity}, combined with Lemma \ref{lem:CircleSpace}, provide a way of constructing bounded open sets for which all the spaces considered in this section are different from each other for $s\geq-n/2$. (Note that the statement of Lemma \ref{lem:equalityNullity}\rf{aaa} is empty if $s<-n/2$ as $\IntClosOm\setminus\Omega$ is necessarily $(-s)$-null in this case (cf.\ Lemma \ref{lem:polarity}\rf{dd}).
One might speculate that if $s<-n/2$ then $\tH^s(\Omega)= H^s_{\overline\Omega}$ for every open $\Omega\subset\R^n$, not just when $\mathrm{int}(\overline\Omega)$ is~$C^0$ (see Proposition \ref{prop:TildeSubscript}\rf{ts0} above). But proving this in the general case is an open problem.
\begin{thm}\label{thm:notequalbig}
For every $n\in\N$, there exists a bounded open set $\Omega\subset \R^n$ such that,
for every $s>0$, $\tH^s(\Omega) \subsetneqq \cirHs(\Omega) \subsetneqq H^s_{\overline{\Omega}}$,
and
for every $s\geq-n/2$, $\tH^s(\Omega) \subsetneqq H^s_{\overline{\Omega}}$.
\end{thm}
\begin{proof}
Let $\Omega_1$ be any bounded open set for which $\mathrm{int}(\overline\Omega_1)\setminus\Omega_1$ has positive measure and is not $n/2$-null, for example an open ball minus a compact set of the type considered in Lemma \ref{lem:polarity}\rf{mm}. Let $\Omega_2$ be any bounded open set for which $\mathrm{int}(\overline\Omega_2)\setminus\Omega_2$ has zero measure and is not $s$-null for any $s<0$, for example an open ball minus the Cantor set $F^{(n)}_{n,\infty}$ from \cite[Theorem 4.5]{HewMoi:15}.
Then, by Lemmas~\ref{lem:CircleSpace} and \ref{lem:equalityNullity},
\begin{align*}
\tH^s(\Omega_1)&\subsetneqq  H^s_{\overline{\Omega_1}}, && \mbox{ for all }s\geq-n/2,\nonumber\\
\cirHs(\Omega_1)&\subsetneqq  H^s_{\overline{\Omega_1}}, && \mbox{ for all }s\ge0,
\\
\tH^s(\Omega_2) &\subsetneqq \cirHs(\Omega_2) , && \mbox{ for all }s>0.\nonumber
\end{align*}
 Provided $\Omega_1$ and $\Omega_2$ have disjoint closure (this can always be achieved by applying a suitable translation if necessary) the open set $\Omega:=\Omega_1\cup\Omega_2$ has the properties claimed in the assertion.
\end{proof}

For bounded open sets with $\Omega=\mathrm{int}(\overline{\Omega})$, the equality $\tH^s(\Omega)=H^s_{\overline\Omega}$ is equivalent to $\overline\Omega$ being ``$(s,2)$-stable'', in the sense of \cite[Definition 11.5.2]{AdHe} and \cite[Definition 3.1]{BaCa:01}.
(We note that the space $L^{s,2}_0(E)$ appearing in \cite[Definition 11.5.2]{AdHe} is equal to $\tH^s(E)$ when $E$ is open (see \cite[Equation (11.5.2)]{AdHe}), and equal to $H^s_E$ when $E$ is compact (see \cite[\S10.1]{AdHe}).)
Then, results in \cite[\S11]{AdHe} -- specifically, the remark after Theorem~11.5.3,
Theorem~11.5.5 (noting that the compact set $K$ constructed therein satisfies $K=\overline{\mathrm{int}(K)}$) and Theorem~11.5.6 -- provide the following results, which show that, at least for $m\in\N$, $\Omega=\mathrm{int}(\overline{\Omega})$ is not a sufficient condition for $\tH^m(\Omega)=H^m_{\overline\Omega}$ unless $n=1$.
Part \rf{AdHe1} of Lemma \ref{lem:stability} also appears in \cite[Theorem 7.1]{BaCa:01}. We point out that references \cite{AdHe} and \cite{BaCa:01} also collect a number of technical results from the literature, not repeated here, relating $(s,2)$-stability to certain ``polynomial'' set capacities (e.g.\ \cite[Theorem 11.5.10]{AdHe} and \cite[Theorem 7.6]{BaCa:01}) %
and spectral properties of partial differential operators (e.g.\ \cite[Theorem 6.6]{BaCa:01}).
\begin{lem}[{\cite[$\S$11]{AdHe}}]
\label{lem:stability}
\begin{enumerate}[(i)]
\item \label{AdHe1} If $n=1$ and $\Omega\subset \R$ is open, bounded and satisfies $\Omega=\mathrm{int}(\overline{\Omega})$, then $\tH^m(\Omega)=H^m_{\overline\Omega}$ for all $m\in\N$.
\item \label{AdHe2} If $n\geq 2$ and $m\in\N$, there exists a bounded open set $\Omega\subset \R^n$ for which $\Omega=\mathrm{int}(\overline{\Omega})$ but $\tH^m(\Omega)\neq H^m_{\overline\Omega}$.
\item \label{AdHe3} If $n\geq 3$ then the set $\Omega$ in point \rf{AdHe2} can be chosen so that $\overline{\Omega}^c$ is connected.
\end{enumerate}
\end{lem}

We now consider the following question: if $\Omega$ is the disjoint union of finitely many open sets $\{\Omega_\ell\}_{\ell=1}^L$, each of which satisfies $\tH^s(\Omega_\ell) = H^s_{\overline{\Omega_\ell}}$, then is $\tH^s(\Omega) = H^s_{\overline{\Omega}}$?
Certainly this will be the case when the closures of the constituent sets are mutually disjoint. But what about the general case when the closures intersect nontrivially?
A first answer, valid for a narrow range of regularity exponents, is given by the following lemma, which is a simple consequence of a standard result on pointwise Sobolev multipliers.
\begin{lem}\label{lem:TildeSubscriptHalf}
Let $\Omega\subset\R^n$ be the disjoint union of finitely many bounded Lipschitz open sets $\Omega_1,\ldots,\Omega_L$.
Then $\tH^s(\Omega) = H^s_{\overline{\Omega}}$ for $0\leq s<1/2$.
\end{lem}
\begin{proof}
Let $0\leq s<1/2$ and $u\in H^s_{\overline{\Omega}}$.
By \cite[Proposition~5.3]{Tri:02} and Lemma~\ref{lem:sob_equiv}, where $\chi_{\Omega_\ell}$ is
the characteristic function of $\Omega_\ell$, $u\chi_{\Omega_\ell}\in H^s_{\overline{\Omega_\ell}}=\tH^s(\Omega_\ell)\subset \tH^s(\Omega)$. Thus $\sum_{\ell=1}^L u\chi_{\Omega_\ell}\in\tH^s(\Omega)$, and $\sum_{\ell=1}^L u\chi_{\Omega_\ell}=u$ since $m(\partial \Omega)\leq \sum_{\ell=1}^L m(\partial \Omega_\ell) = 0$.
\end{proof}
Lemma~\ref{lem:TildeSubscriptHalf} can be extended to disjoint unions of some classes of non-Lipschitz open sets using \cite[Definition 4.2, Theorem~4.4]{Sickel99a}, leading to the equality $\tH^s(\Omega) = H^s_{\overline{\Omega}}$ for $0\leq s<t/2$ for some $0<t<1$ related to the boundary regularity (cf.\ also \cite[Theorem 6]{Sickel99b} and \cite[Theorem 3, p.\ 216]{RunstSickel}). However, %
the technique of Lemma~\ref{lem:TildeSubscriptHalf}, namely using characteristic functions as pointwise multipliers, cannot be extended to $s\geq 1/2$, no matter how regular the constituent sets are;
indeed, \cite[Lemma~3.2]{Sickel99a} states that $\chi_\Omega\notin H^{1/2}\Rn$ for any non-empty open set $\Omega\subset\R^n$.

We now state and prove a general result, which allows us to prove $\tH^s\OO=H^s_{\overline\Omega}$, for $|s|\leq 1$ if $n\geq 2$, $|s|\leq 1/2$ if $n=1$, for a class of open sets which are in a certain sense ``regular except at a countable number of points''. This result depends on the following lemma that is inspired by results in \cite[\S17]{Tartar}, whose proof we defer to later in this section.

\begin{lem} \label{lem:vj} Suppose that $n\geq 2$, that $N\in \N$ and $\bx_1,...,\bx_N\in \R^n$ are distinct, and that
\begin{equation} \label{eq:Rbound}
0<R<\min_{i,j\in \{1,...,N\}}\frac{|\bx_i-\bx_j|}{6}.
\end{equation}
Then there exists a family $(v_j)_{j\in\N}\subset C^\infty(\R^n)$ and a constant $C>0$ such that, for all $j\in \N$: (i) $0\leq v_j(\bx)\leq 1$, for $\bx\in \R^n$;
(ii) $v_j(\bx) = 0$, if $|\bx-\bx_i|< R/(2j)$ for some $i\in \{1,...,N\}$; (iii) $v_j(\bx)=1$,  if $|\bx-\bx_i|> 5R/2$ for all $i\in \{1,...,N\}$; (iv) $\|v_j\phi\|_{H^s(\R^n)} \leq C\|\phi\|_{H^s(\R^n)}$, for all $\phi\in H^s(\R^n)$ with $|s|\leq 1$; (v) $\|v_j\phi-\phi\|_{H^s(\R^n)} \to 0$ as $j\to\infty$, for all $\phi\in H^s(\R^n)$ with $|s|\leq 1$. For $n=1$ the same result holds, but with $s$ restricted to $|s|\leq 1/2$.
\end{lem}

\begin{thm} \label{thm:new} Suppose that $|s|\leq 1$ if $n\geq 2$, $|s|\leq 1/2$ if $n=1$, that $\Omega\subset \R^n$ is open, and that: (i) $P\subset \partial \Omega$ is closed and countable with at most  finitely many limit points in every bounded subset of $\partial\Omega$; (ii) $\Omega$ has the property that, if $u\in H_{\overline{\Omega}}^s$ is compactly supported with $\supp(u)\cap P=\emptyset$, then $u\in \tH^s(\Omega)$. Then $\tH^s(\Omega)=H_{\overline{\Omega}}^s$.
\end{thm}
\begin{proof} Suppose that $|s|\leq 1$ if $n\geq 2$, $|s|\leq 1/2$ if $n=1$. Since the set of compactly supported $v\in H^s_{\overline{\Omega}}$ is dense in $H^s_{\overline{\Omega}}$ by \eqref{eq:approx2} and $\tH^s(\Omega)$ is closed, it is enough to show that $v\in \tH^s(\Omega)$ for every compactly supported $v\in H^s_{\overline{\Omega}}$. So suppose that $v\in H^s_{\overline \Omega}$ is compactly supported, and let $Q$ be the (finite) set of limit points of $P$ that lie in the support of $v$. Let $(v_j)_{j\in \N}\subset C^\infty(\R^n)$ be a family constructed as in Lemma \ref{lem:vj}, such that $v_jv\to v$ as $j\to\infty$ in $H^s(\R^n)$, and each $v_j=0$ in a neighbourhood of $Q$. For each $j\in \N$, $P_j := P\cap \supp(v_jv)$ is finite. For each $j\in \N$, let $(v_{j,\ell})_{\ell\in \N}\subset C^\infty(\R^n)$ be a family constructed as in Lemma \ref{lem:vj}, such that $v_{j,\ell}v_jv\to v_jv$ as $\ell\to\infty$ in $H^s(\R^n)$, and each $v_{j,\ell}=0$ in a neighbourhood of $P_j$. Then $w_{j,\ell}:= v_{j,\
ell}v_jv\in \tH^s(\Omega)$, for all $j,\ell\in \N$, by hypothesis. Since $\tH^s(\Omega)$ is closed it follows that $v\in \tH^s(\Omega)$.
\end{proof}

In the next theorem, when we say that the open set $\Omega\subset \R^n$ is $C^0$ except at the points $P\subset \partial \Omega$, we mean that its  boundary $\partial\Omega$ can, in a neighbourhood of each point in $\partial \Omega\setminus P$, be locally represented as the graph (suitably rotated) of a $C^0$ function from $\R^{n-1}$ to $\R$, with $\Omega$ lying only on one side of $\partial\Omega$.
(In more detail we mean that $\Omega$ satisfies the conditions of \cite[Definition 1.2.1.1]{Gri}, but for every $\bx \in \partial \Omega \setminus P$ rather than for every $\bx \in \partial \Omega$.)

\begin{thm} \label{thm:new2}  Suppose that $|s|\leq 1$ if $n\geq 2$, $|s|\leq 1/2$ if $n=1$. Suppose further that $\Omega\subset\R^n$ is open, and that $\Omega$ is $C^0$ except at a set of points $P$ satisfying condition (i) of Theorem \ref{thm:new}. Then $\tH^s(\Omega)=H_{\overline{\Omega}}^s$. In particular, $\tH^s(\Omega)=H_{\overline{\Omega}}^s$ if $\Omega$ is the union of disjoint $C^0$ open sets, whose closures intersect only at a set of points $P$ that satisfies condition (i) of Theorem \ref{thm:new}.
\end{thm}
\begin{proof} The first two sentences of this result will follow from Theorem \ref{thm:new} if we can show that $\Omega$ satisfies condition (ii) of Theorem \ref{thm:new}.  We will show that this is true (for all $s\in \R$) by a partition of unity argument, adapting the argument used to prove Lemma \ref{lem:sob_equiv} in \cite[Theorem 3.29]{McLean}.

Suppose that $u\in H_{\overline{\Omega}}^s$ is compactly supported with $\supp(u)\cap P=\emptyset$. For each $\bx \in \supp(u)$, let $\epsilon(\bx)>0$ be such that $\partial \Omega$ is the rotated graph of a $C^0$ function and $\Omega$ the rotated hypograph of that function in $B_{3\epsilon(\bx)}(\bx)$ if $\bx\in \partial \Omega$, and such that $B_{3\epsilon(\bx)}(\bx)\subset \Omega$ if $\bx \in \Omega$. Then $\{B_{\epsilon(\bx)}(\bx):\bx \in \supp(u)\}$ is an open cover for $\supp(u)$. Since $\supp(u)$ is compact we can choose a finite sub-cover $\mathcal{W}=\{B_{\epsilon(\bx_i)}(\bx_i):i \in \{1,...,N\}\}$. Choose a partition of unity $(\chi_i)_{i=1}^N$ for $\supp(u)$ subordinate to $\mathcal{W}$, with $\supp(\chi_i)\subset B_{\epsilon(\bx_i)}(\bx_i)$ for $1\leq i\leq N$, this possible by \cite[Theorem 2.17]{grubb}. Given $\eta>0$, for $i=1,...,N$ choose $\phi_i\in \scrD(\Omega)$ such that $\|\chi_iu-\phi_i\|_{H^s(\R^n)}\leq \eta/N$. This is possible by \eqref{eq:approx} if $\bx_i\in \Omega$. To see that 
this is possible if $\bx_i\in \partial \Omega \cap \supp(u)$ we argue as in the proof of Lemma \ref{lem:sob_equiv} given in \cite[Theorem 3.29]{McLean}, first making a small shift of $\chi_iu$ to move its support into $\Omega$, and then approximating by \eqref{eq:approx}. Then $\phi = \sum_{i=1}^N \phi_i\in \scrD(\Omega)$ and $\|u-\phi\|_{H^s(\R^n)}\leq\eta$. Since $\eta>0$ is arbitrary, this implies that $u\in \tH^s(\Omega)$.

The last sentence of the theorem is an immediate corollary.
\end{proof}

The above theorem applies, in particular, whenever $\Omega$ is $C^0$ except at a finite number of points. The following remark notes applications of this type.
\begin{rem}
Theorem \ref{thm:new2} implies that $\tH^s\OO=H^s_{\overline\Omega}$, for $|s|\leq 1$, for a number of well-known examples of non-$C^0$ open sets. In particular we note the following examples, illustrated in Figure \ref{fig:TSExamples}, all of which are $C^0$ except at a finite number of points:
\begin{enumerate}
\item any finite union of polygons (in $\R^2$) or $C^0$ polyhedra (in $\R^3$) where the closures of the constituent polygons/polyhedra intersect only at a finite number of points, for example the standard prefractal approximations to the Sierpinski triangle (see Figure \ref{subfig:Sierpinski});
\item the double brick domain of \cite[p.~91]{McLean} (see Figure \ref{subfig:DoubleBrick});
\item sets with ``curved cusps'', either interior or exterior,
e.g.\ $\{(x,y)\in \R^2 : x^2+y^2<1 \textrm{ and } x^2+(y+1/2)^2>1/2\}$ or its complement (see Figure \ref{subfig:Cusps});
\item spiral domains, e.g.\ $\{(r\cos{\theta},r\sin\theta)\in \R^2:2^{\theta/(2\pi)}<r<\frac{3}{2}2^{\theta/(2\pi)}, \, \theta\in\R\}$ (see Figure \ref{subfig:Spiral});
\item the ``rooms and passages'' domain of \cite[\S2.1]{Fr:79} (see Figure \ref{subfig:RoomsAndPassages}).
\end{enumerate}
\end{rem}

\begin{figure}[!t]
\begin{center}
\subfigure[\label{subfig:Sierpinski} The first four prefractal approximations to the Sierpinski triangle]{
\hspace{-7.5mm}
\includegraphics[height=25mm]{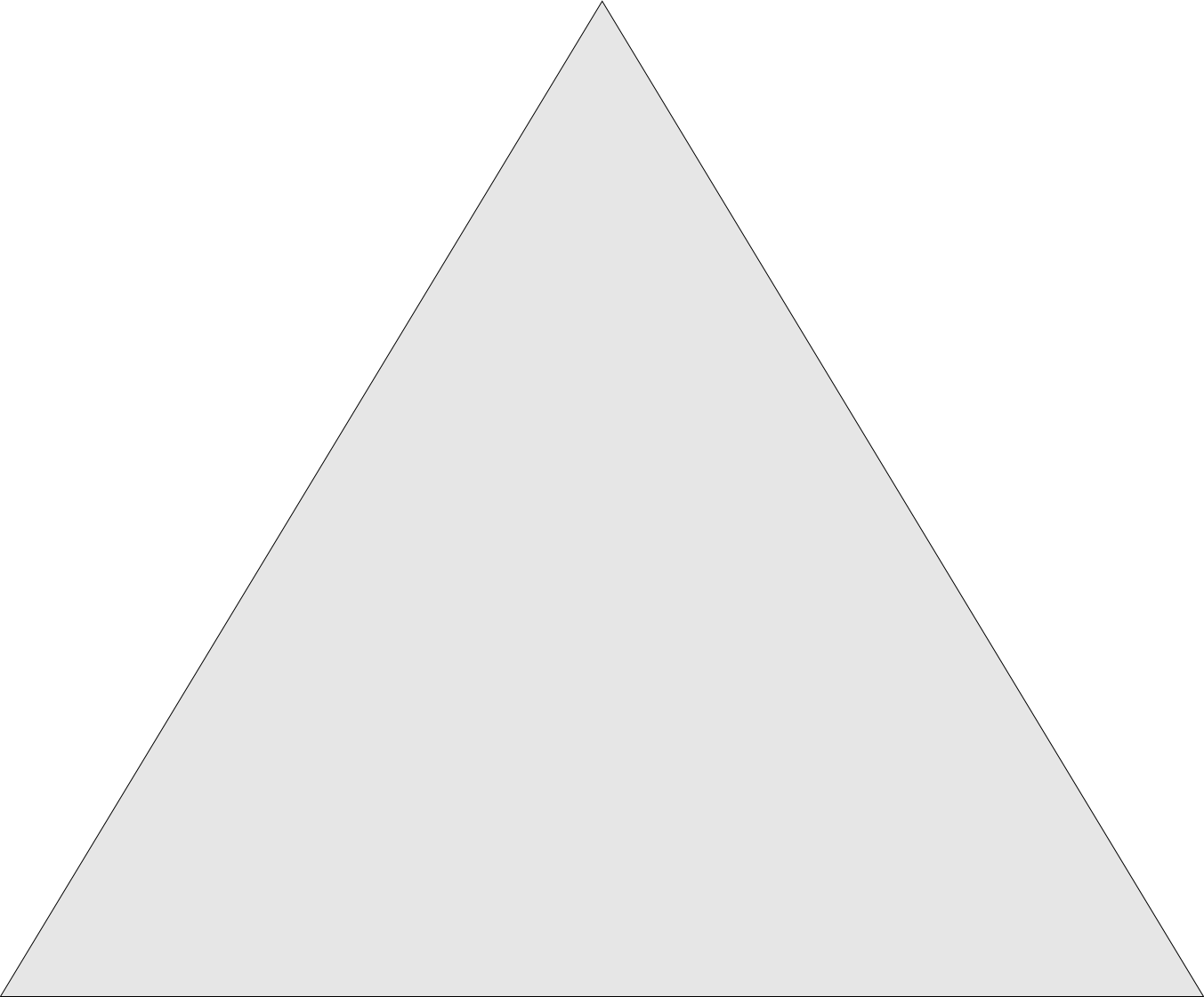}\hs{2}
\includegraphics[height=25mm]{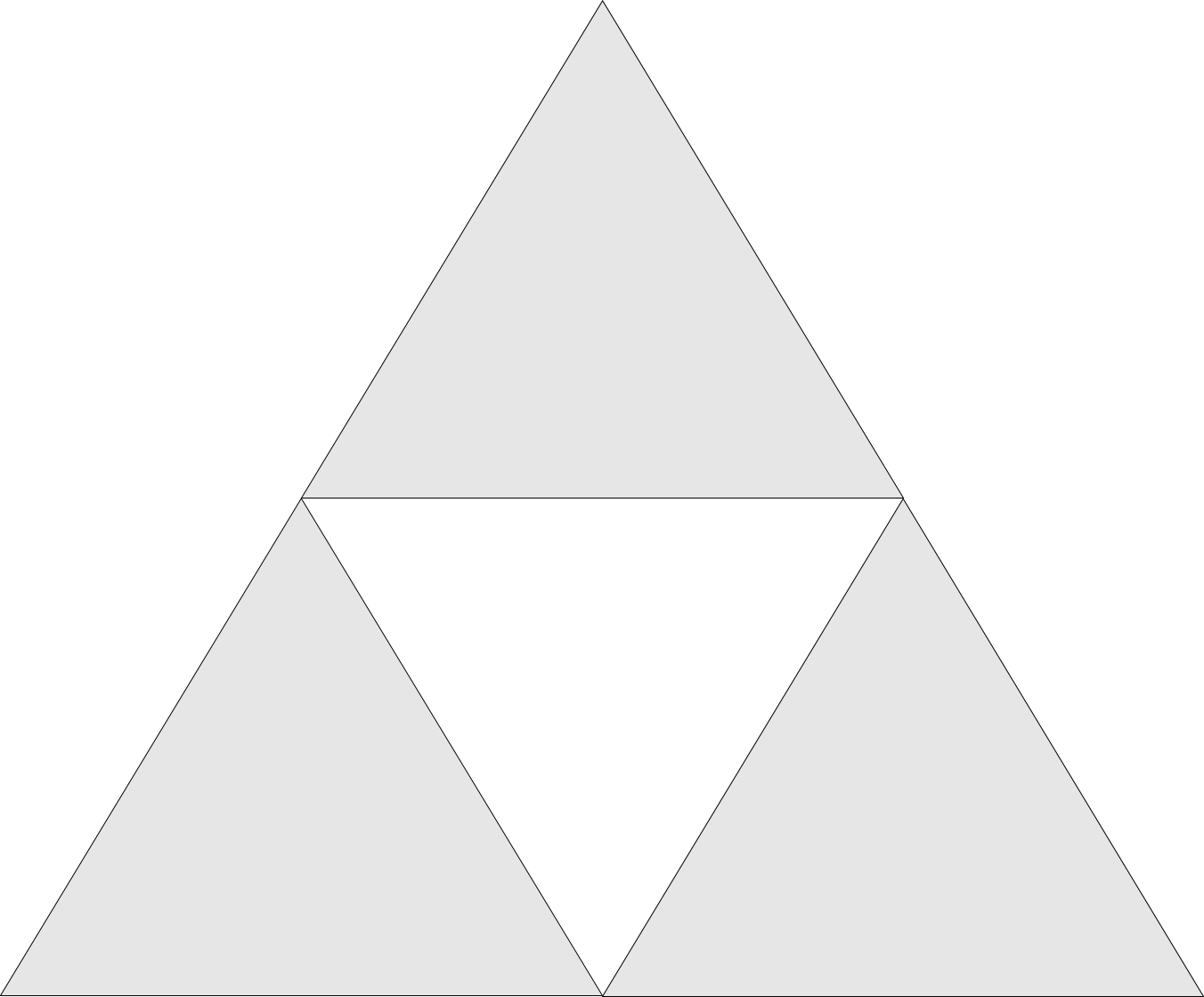}\hs{2}
\includegraphics[height=25mm]{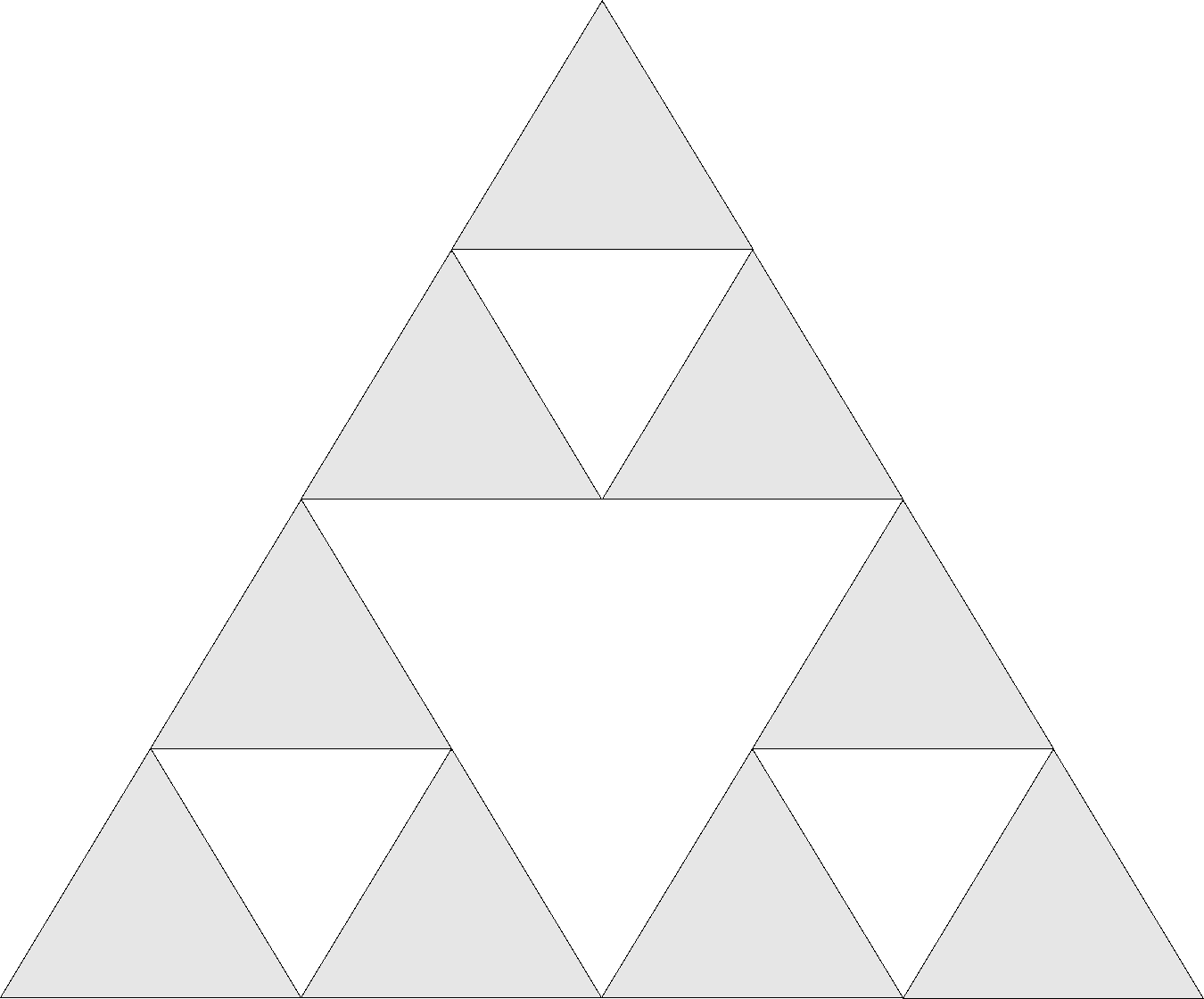}\hs{2}
\includegraphics[height=25mm]{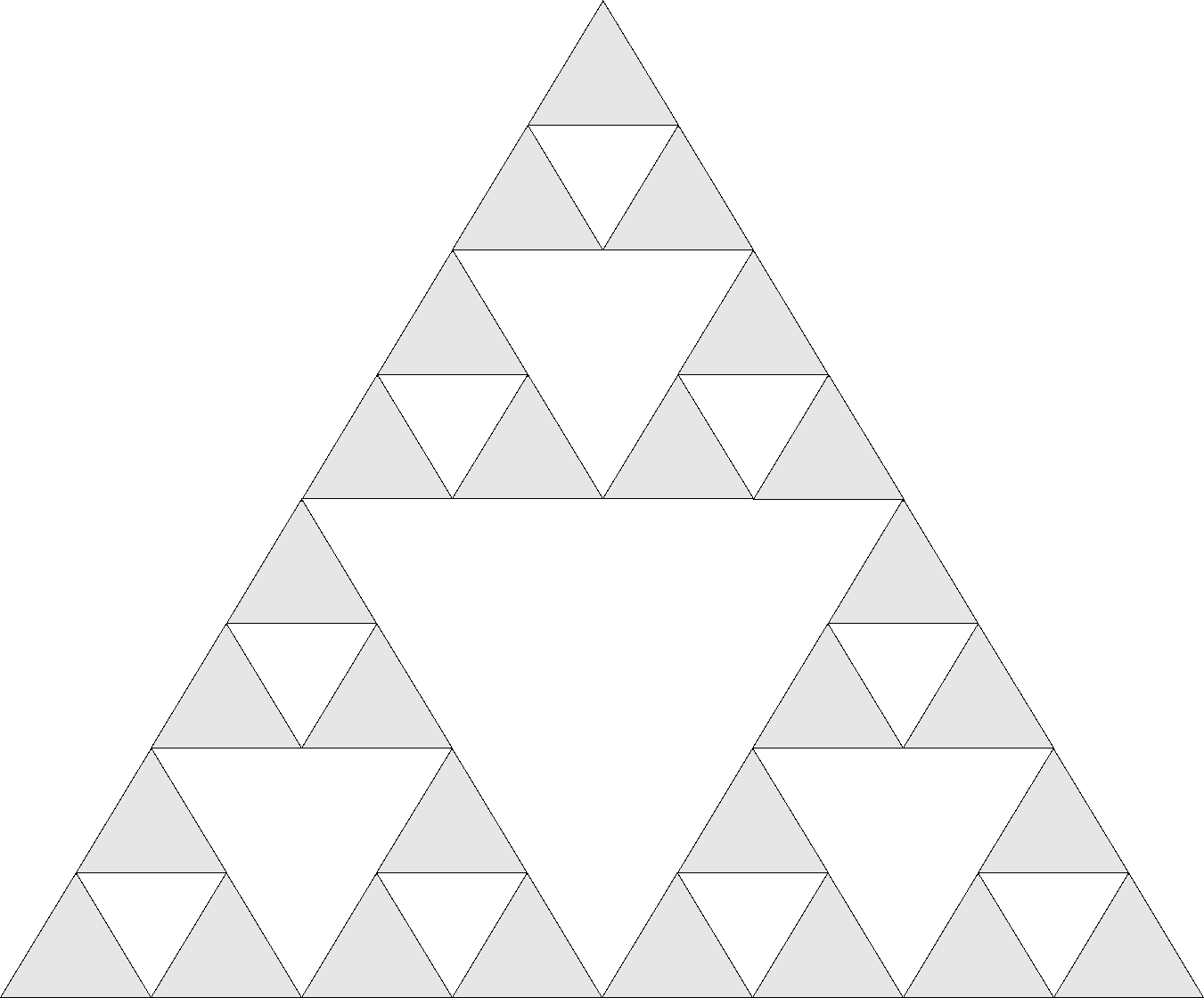}
}
\subfigure[\label{subfig:DoubleBrick} Double brick]{
\includegraphics[height=25mm]{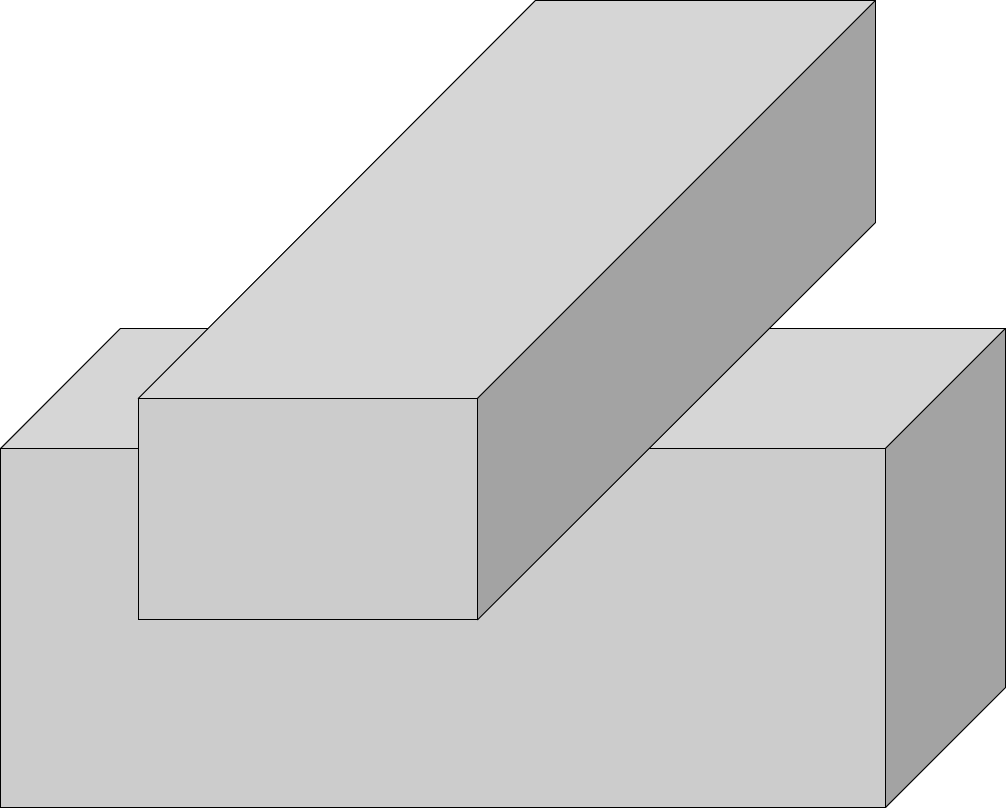}
}
\hs{3}
\subfigure[\label{subfig:Cusps} Curved cusps]{
\includegraphics[height=30mm]{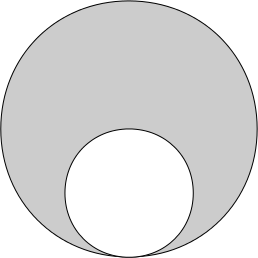}
}
\hs{3}
\subfigure[\label{subfig:Spiral} Spiral]{
\includegraphics[height=30mm]{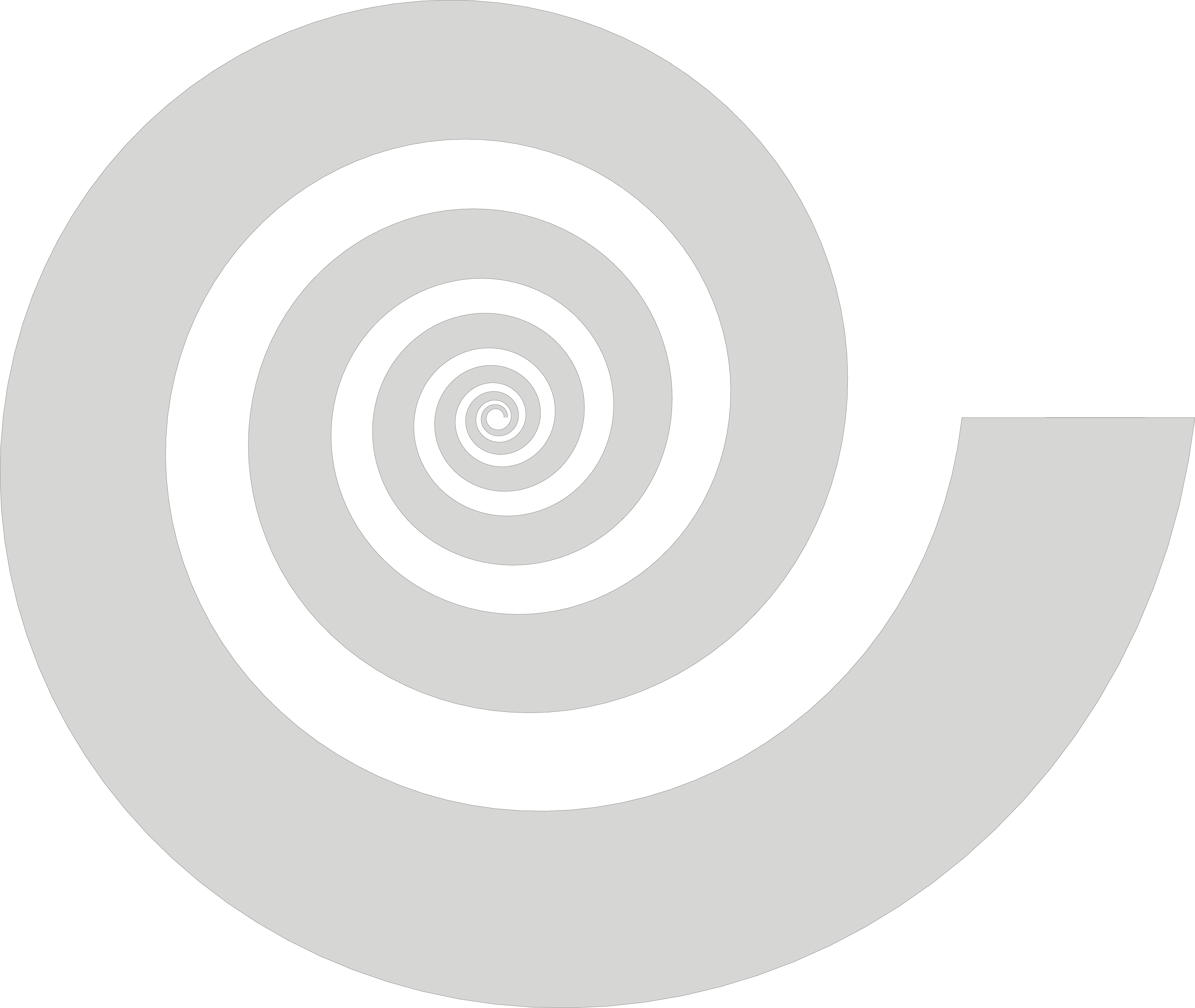}
}
\subfigure[\label{subfig:RoomsAndPassages} ``Rooms and passages'']{
\includegraphics[height=30mm]{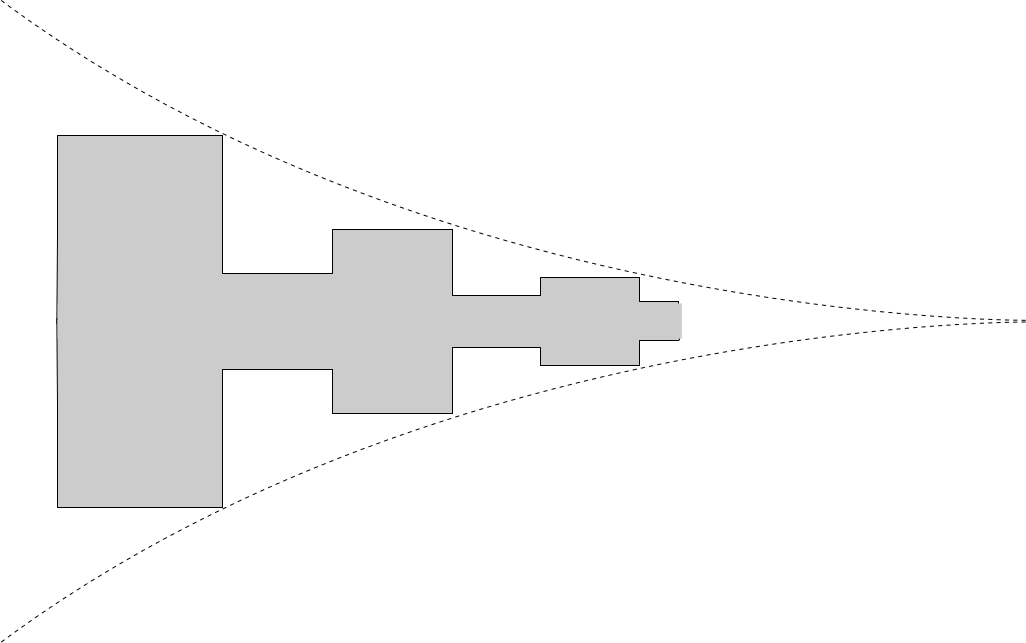}
}
\end{center}
\caption{\label{fig:TSExamples} Examples of non-$C^0$ open sets to which Theorem \ref{thm:new2} applies.}
\end{figure}

\begin{proof}[Proof of Lemma \ref{lem:vj}.]
Choose $R>0$ to satisfy \eqref{eq:Rbound}. The case $n=2$ is the hardest so we start with that. For $j\in \N$, define $\Phi_j\in C(\R)$ by
\begin{align*}%
\Phi_j(r):=
\begin{cases}
0, & r\leq R/j,\\
1-\dfrac{\log(r/(2R))}{\log(1/(2j))}, & R/j<r\leq 2R,\\
1, & r>2R,
\end{cases}
\end{align*}
and note that $\Phi_j^\prime(r) = (r\log(2j))^{-1}$, for $R/j<r<2R$.
We define by mollification a smoothed version $\Psi_j$ of $\Phi_j$. Choose $\chi\in \scrD(\R)$ with $0\leq \chi(t)\leq 1$ for $t\in\R$, $\chi(t)=0$ if $|t|\geq 1$, and $\int_{-\infty}^\infty \chi(t) \rd t = 1$. Define $\chi_j(t) := (2j/R)\chi(2jt/R)$, $t\in \R$, and
$$
\Psi_j(r) := \int_{-\infty}^\infty \chi_j(r-t)\Phi_j(t)\, \rd t = \int_{-\infty}^\infty \chi_j(t)\Phi_j(r-t)\, \rd t, \quad r\in \R.
$$
Then $\Psi_j\in C^\infty(\R)$, $0\leq \Psi_j(r)\leq 1$ for $r\in \R$, $\Psi_j(r)=0$ if $r\leq R/(2j)$, $\Psi_j(r)=1$ if $r\geq 2R + R/(2j)$, and
\begin{equation} \label{eq:Psibound}
0\leq\Psi^\prime_j(r)\leq \max_{|t-r|\leq R/(2j)}\, \Phi^\prime_j(t) \leq \frac{3}{2r\log(2j)}, \quad \mbox{for }\frac{R}{2j}< r\leq 2R+ \frac{R}{2j}.
\end{equation}

For $n=2$ we define the sequence $(v_j)_{j\in \N}$ by
\begin{equation} \label{eq:vjdef}
v_j(\bx) := \prod_{i=1}^N \Psi_j(|\bx-\bx_i|), \quad \bx\in \R^2.
\end{equation}
Clearly $v_j\in C^\infty(\R^2)$ and satisfies conditions (i)-(iii).  Noting that
\begin{equation} \label{eq:vjbound}
|\nabla v_j(\bx)| = \left\{\begin{array}{cc}
                         \Psi_j^\prime(|\bx-\bx_i|),   & \mbox{for } |\bx-\bx_i|\leq 5R/2, \;\; i\in \{1,...,N\}, \\
                           0, & \mbox{otherwise},
                         \end{array}\right.
\end{equation}
 it follows from \eqref{eq:Psibound} that $\|\nabla v_j\|_{L^2(\R^2)}\to 0$ as $j\to\infty$, and hence by the dominated convergence theorem that (v) holds for all $\phi\in \scrD(\R^2)$ and $s=1$ (and so also for $s<1$). Thus, if (iv) holds, (v) follows by density arguments.

We will prove (iv) first for $s=1$, then for $s=-1$ by a duality argument, then for $s\in[-1,1]$ by interpolation. 
Choose $\varphi\in \scrD(\R^2)$ with support in $\cup_{i=1}^N B_R(\bx_i)$ and such that $\varphi=1$ in a neighbourhood of $\{\bx_1,...,\bx_N\}$. It is clear from \eqref{eq:vjbound} and \eqref{eq:Psibound} that the operation of multiplication by $(1-\varphi)v_j$ is bounded on $H^1(\R^2)$, uniformly in $j$. It follows from the same bounds and the fact (for $n=2$) that (cf.\ \cite[Lemma 17.4]{Tartar})
\begin{align}
\label{eqn:Tartar17p4}
\int_{B_{R}}\frac{|u|^2}{|\bx|^2\log^2(|\bx|/R)}\, \rd\bx\leq 4\int_{B_{R}}|\nabla u|^2\, \rd\bx, \qquad u\in \widetilde{H}^1(B_{R}),
\end{align}
that the operation of multiplication by $\varphi v_j$ is bounded on $H^1(\R^2)$, uniformly in $j$. Thus (iv) holds for some constant $C>0$ for $s=1$. Abbreviating $H^{\pm 1}(\R^2)$ by $H^{\pm 1}$ and $\langle \cdot,\cdot\rangle_{H^{-1}\times H^{1}}$ by $\langle \cdot,\cdot\rangle$, since $H^{-1}$ is a unitary realisation of $(H^{1})^*$ it holds for $\phi\in H^{-1}$ that
\begin{align*}%
\|v_j\phi\|_{H^{-1}} = \sup_{\substack{v\in H^1\\ \|v\|_{H^1}=1}}\left|\langle v_j\phi,v\rangle\right|
= \sup_{\substack{v\in H^1\\ \|v\|_{H^1}=1}}\left|\langle \phi,v_jv\rangle\right| \leq C\|\phi\|_{H^{-1}},
\end{align*}
i.e.\ (iv) holds also for $s=-1$ with the same constant $C$, and hence also for $s\in [-1,1]$ by interpolation (e.g.\ \cite[(1) and Theorem 4.1]{InterpolationCWHM}).

If $n\ge3$ we argue and define $v_j$ as above, but with the simpler choice $\Psi_j(r):=\psi(jr/R)$, where $\psi$ is any function in $C^\infty(\R)$ with $\psi(r)=0$ for $r<1$, $\psi(r)=1$ for $r>2$, and $0\leq \psi(r)\leq 1$ for $r\in\R$. To prove (iv) for $s=1$ one uses instead of \rf{eqn:Tartar17p4} the bound (cf.\ \cite[Lemma 17.1]{Tartar})
\begin{align*}%
\displaystyle{\int_{\R^n}\frac{|u|^2}{|\bx|^2} \, \rd\bx\leq \frac{2}{n-2}\int_{\R^n}|\nabla u|^2\,\rd\bx, \qquad u\in H^1(\R^n).}
\end{align*}

 If $n=1$ then the result follows by embedding $\R$ in $\R^2$, trace theorems, interpolation and duality. In more detail, if $\bx_1,...,\bx_N\in \R\subset \R^2$ are distinct, and $(v_j)_{j\in \N}\subset C^\infty(\R^2)$, satisfying (i)-(v) for $n=2$ and $|s|\leq 1$, is defined by \eqref{eq:vjdef}, then $(v_j|_\R)_{j\in\N}$ satisfies (i)-(iii) for $n=1$. To see (iv) holds, note that $(v_j|_\R)$ is uniformly bounded in $L^\infty(\R)$. Moreover, let $c$ denote the norm of the trace operator $\gamma:H^1(\R^2)\to H^{1/2}(\R)$, defined by $\gamma v = v|_\R$ for $v\in \scrD(\R^2)$, and $c^\prime$ the norm of a right inverse $E:H^{1/2}(\R)\to H^1(\R^2)$ of $\gamma$. Then
\begin{equation} \label{eq:ineq}
\|v_j|_\R\phi\|_{H^{1/2}(\R)} \leq c \|v_jE\phi\|_{H^1(\R^2)} \leq cC \|E\phi\|_{H^1(\R^2)}\leq c^\prime cC \|\phi\|_{H^{1/2}(\R^2)},
\end{equation}
for $\phi\in H^{1/2}(\R)$. Thus $(v_j|_\R)$ satisfies (iv) for $n=1$ for $s=0$ and $s=1/2$, and hence for $0\leq s\leq 1/2$ by interpolation, and then for $-1/2\leq s<0$ by duality arguments as above.  Finally, (v) follows by density, as in the case $n=2$, if we can show that (v) holds for $s=1/2$ and all $\phi\in \scrD(\R)$. But, arguing as in \eqref{eq:ineq}, this follows from (v) for $n=2$.
\end{proof}

We end this section with a result linking the inclusions in \rf{eqn:inclusionsRepeat} to taking complements. This result generalises \cite[Theorem~1.1]{Po:72}, where the same result is proved for the special case where $s\in \N$ and $\Omega$ is the interior of a compact set.
\begin{lem}
\label{lem:TSComplements}
Let $\Omega\subset \R^n$ be open and non-empty, and let $s\in \R$. %
Then $\tH^s(\Omega) = H^s_{\overline{\Omega}}$ if and only if $\tH^{-s}(\overline{\Omega}^c) = H^{-s}_{\Omega^c}$.
\end{lem}
\begin{proof}
Applying Lemma~\ref{lem:orth_lem} twice, and using $V_2^{a,H^{s}\Rn}\subset V_1^{a,H^{s}\Rn}$ for all closed spaces $V_1\subset V_2\subset H^{-s}\Rn$, we have
$\tH^s\OO=(H^{-s}_{\Omega^c})^{a,H^{s}\Rn}
\subset (\tH^{-s}(\overline\Omega^c))^{a,H^{s}\Rn}=H^s_{\overline\Omega}$.
The assertion follows noting that $V_1^{a,H^{s}\Rn}= V_2^{a,H^{s}\Rn}$ if and only if $V_1= V_2$.
\end{proof}

\begin{rem}
If $\mathrm{int}(\overline{\Omega})\setminus\Omega$ is $(-s)$-null (for example if $\Omega=\inter(\overline\Omega)$) then $H^{-s}_{\Omega^c} = H^{-s}_{ \overline{\overline{\Omega}^c}}$, by Theorem~\ref{thm:Hs_equality_closed}, and the fact that $ \inter(\overline\Omega)\setminus \Omega=\Omega^c\setminus \overline{\overline\Omega^c}$. In this case, Lemma \ref{lem:TSComplements} says that $\tH^s(\Omega) = H^s_{\overline{\Omega}}$ if and only if $\tH^{-s}(U) = H^{-s}_{\overline{U}}$, where $U=\overline{\Omega}^c$.
\end{rem}

\subsection{When is \texorpdfstring{$H^s_0(\Omega)=H^s(\Omega)$}{Hs0(Omega)=Hs(Omega)}?}
\label{subsec:Hs0vsHs}

The space $H^s_0(\Omega)$ was defined in \eqref{eq:Hs0} as a closed subspace of $H^s(\Omega)$.
In this section we investigate the question of when these two spaces coincide, or, equivalently, when
 $\scrD(\Omega)|_\Omega$ is dense in $H^s(\Omega)$.
One classical result (see \cite[Theorem 1.4.2.4]{Gri} or \cite[Theorem 3.40]{McLean}) is that if $\Omega$ is Lipschitz and bounded, then $H^s_0(\Omega)=H^s(\Omega)$ for $0\leq s\leq 1/2$.
In Corollary \ref{cor:Hs0HsEqual} we extend this slightly, by showing that equality in fact extends to $s<0$ (in fact this holds for any open set $\Omega$, see parts \rf{f3} and \rf{e3} below),
as well as presenting results for non-Lipschitz $\Omega$.
The proofs of the results in Corollary \ref{cor:Hs0HsEqual} are based on the following lemma, which states that the condition $H^s_0(\Omega)=H^s(\Omega)$ is equivalent to a certain subspace of $H^{-s}(\R^n)$ being trivial.
This seemingly new characterisation follows directly from the dual space realisations derived in \S\ref{subsec:DualAnnih}.
\begin{lem}
\label{lem:Hs0HsEquiv}
Let $\Omega\subset \R^n$ be non-empty and open, and let $s\in\R$. Then $H^s_0(\Omega)=H^s(\Omega)$ if and only if $\tH^{-s}(\Omega) \cap H^{-s}_{\partial\Omega} =\{0\}$.
\end{lem}
\begin{proof}
This follows from Theorem \ref{thm:DualityTheorem} and Lemma \ref{lem:Hs0Dual}, which together imply that, by duality, $H^s_0(\Omega)=H^s(\Omega)$ if and only if
$(\tH^{-s}(\Omega) \cap H^{-s}_{\partial\Omega})^{\perp, \tH^{-s}(\Omega)}=\tH^{-s}(\Omega)$, which holds if and only if $\tH^{-s}(\Omega) \cap H^{-s}_{\partial\Omega} =\{0\}$.
\end{proof}

\begin{cor}
\label{cor:Hs0HsEqual}
Let $\Omega\subset \R^n$ be non-empty, open and different from $\R^n$ itself, and let $s\in\R$.
\begin{enumerate}[(i)]
\item \label{lem:Hs0Hs_monot}
If $H^s_0(\Omega)=H^s(\Omega)$ then $H_0^{t}(\Omega)=H^{t}(\Omega)$ for all $t<s$.
\item \label{f3}
If $s\leq 0$ then $H^s_0(\Omega)=H^s(\Omega)$. %
\item \label{a3}
If $\partial\Omega$ is $(-s)$-null then $H^s_0(\Omega)=H^s(\Omega)$.
\item \label{g3}
If $s>n/2$, %
then $H^s_0(\Omega)\subsetneqq H^s(\Omega)$. %
\item \label{b3}
For $0<s<n/2$, if $\dimH{\partial\Omega}<n-2s$ then $H^s_0(\Omega)=H^s(\Omega)$.
\item \label{c3}
If $\tH^{-s}(\Omega)=H^{-s}_{\overline\Omega}$ (e.g. if $\Omega$ is $C^0$) then $H^s_0(\Omega)=H^s(\Omega)$ if and only if $\partial\Omega$ is $(-s)$-null.
\item \label{d3}
If $\Omega$ is $C^0$ then $H^s_0(\Omega)\subsetneqq H^s(\Omega)$ for $s>1/2$.
\item \label{h3}
If $\Omega$ is $C^{0,\alpha}$ for some $0<\alpha<1$ then $H^s_0(\Omega)=H^s(\Omega)$ for $s<\alpha/2$.
\item \label{e3}
If $\Omega$ is Lipschitz then $H^s_0(\Omega)=H^s(\Omega)$ if and only if $s\leq 1/2$.
\end{enumerate}
\end{cor}

\begin{proof}
Our proofs all use the characterization provided by Lemma \ref{lem:Hs0HsEquiv}.
\rf{lem:Hs0Hs_monot} holds because, for $t<s$, $\tH^{-t}(\Omega)\subset \tH^{-s}(\Omega)$ and $H_{\partial\Omega}^{-t}\subset H_{\partial\Omega}^{-s}$.
\rf{f3} holds because, for $s\leq 0$, $\tH^{-s}(\Omega)\cap H^{-s}_{\partial \Omega} \subset \cirH{}^{-s}(\Omega)\cap H^{-s}_{\partial \Omega}= \{0\}$.
\rf{a3} is immediate from Lemma \ref{lem:Hs0HsEquiv}.
To prove \rf{g3}, we first note that, for any $\bx_0\in\deO$, there exists a sequence of points $\{\by_j\}_{j\in\N}\subset\Omega$ such that $\lim_{j\to\infty}\by_j=\bx_0$, and the corresponding Dirac delta functions satisfy $\delta_{\bx_0}\in H^{-s}_\deO$ and  $\delta_{\by_j}\in H^{-s}_{\{\by_j\}}\subset\tH^{-s}(\Omega)$, by \eqref{eq:delta} and \eqref{eq:approx}.
Then, since $\tH^{-s}(\Omega)\subset H^{-s}(\R^n)$ is closed, to show that $\tH^{-s}(\Omega)\cap H^{-s}_\deO\neq \{0\}$ it suffices to prove that  $\{\delta_{\by_j}\}_{j\in\N}$ converges to $\delta_{\bx_0}$ in $H^{-s}(\R^n)$.
Recall that the dual space of $H^{-s}(\R^n)$ is realised as $H^s(\R^n)$, which (since $s>n/2$) is a subspace of $C^0(\R^n)$, the space of %
continuous functions (see, e.g.\ \cite[Theorem 3.26]{McLean}). Hence the duality pairing \eqref{dualequiv} gives $\langle \delta_{\bx_0}-\delta_{\by_j}, \phi\rangle_s=\overline{\phi(\bx_0)}-\overline{\phi(\by_j)}\xrightarrow{j\to\infty}0$ for all $\phi\in H^s(\R^n)\subset C^0(\R^n)$, i.e.\ $\{\delta_{\by_j}\}_{j\in\N}$ converges to $\delta_{\bx_0}$ weakly in $H^{-s}(\R^n)$.
But by \cite[Theorem 3.7]{Brezis}, $\tH^{-s}(\Omega)$ is weakly closed, so $\delta_{\bx_0}\in \tH^{-s}(\Omega)$ as required. %
\rf{b3} follows from \rf{a3} and Lemma \ref{lem:polarity}\rf{hh}. For \rf{c3}, note that if $\tH^{-s}(\Omega)=H^{-s}_{\overline\Omega}$ then $\tH^{-s}(\Omega)\cap H^{-s}_{\partial \Omega}
= H^{-s}_{\partial \Omega}$.
\rf{d3}--\rf{e3} follow from \rf{c3}, Lemma~\ref{lem:sob_equiv}, and Lemma~\ref{lem:polarity}\rf{kk1}--\rf{kk}.
\end{proof}

\begin{rem}
\label{rem:HsHs0}
Parts \rf{lem:Hs0Hs_monot}, \rf{f3} and \rf{g3} of Corollary \ref{cor:Hs0HsEqual} imply that for any non-empty open $\Omega\subsetneqq\R^n$, there exists
$0\leq s_0(\Omega)\leq n/2$
such that
$$
H_0^{s_-}(\Omega)= H^{s_-}(\Omega)\quad\text{and}\quad
H_0^{s_+}(\Omega)\subsetneqq H^{s_+}(\Omega)
\qquad \text{for all }\;  s_-<s_0(\Omega)<s_+.
$$
We can summarise most of the remaining results in Corollary \ref{cor:Hs0HsEqual} as follows:
\begin{itemize}
\item $s_0(\Omega)\ge \sup\{s: \deO $ is $(-s)$-null$\}$, with equality if $\Omega$ is $C^0$.
\item If $\Omega$ is $C^0$, then $0\le s_0(\Omega)\le 1/2$.
\item If $\Omega$ is $C^{0,\alpha}$ for some $0<\alpha<1$, then $\alpha/2\le s_0(\Omega)\le 1/2$.
\item If $\Omega$ is Lipschitz, then $s_0(\Omega)=1/2$.
\end{itemize}
Moreover, the above bounds on $s_0(\Omega)$ can all be achieved: by Corollary \ref{cor:Hs0HsEqual}(vi) for the first two cases, (iii) and (iv) for the third case:%
\begin{itemize}
\item For $2\le n\in\N$ the bounded $C^0$ open set of \cite[Lemma 4.1(vi)]{HewMoi:15}
satisfies $s_0(\Omega)=0$.
\item For $2\le n\in\N$ and $0<\alpha<1$, the bounded $C^{0,\alpha}$ open set
of \cite[Lemma 4.1(v)]{HewMoi:15}
satisfies $s_0(\Omega)=\alpha/2$.
\item If $\Omega=\R^n\setminus\{\mathbf0\}$, $s_0(\Omega)=n/2$.
\end{itemize}
\end{rem}

To put the results of this section in context we give a brief comparison with the results presented by Caetano in \cite{Ca:00}, where the question of when $H^s_0(\Omega)=H^s(\Omega)$ is considered within the more general context of Besov--Triebel--Lizorkin spaces.
Caetano's main positive result \cite[Proposition 2.2]{Ca:00} is that if $0<s<n/2$, $\Omega$ is bounded, and $\overline{\dimB}\partial\Omega<n-2s$, then $H^s_0(\Omega)=H^s(\Omega)$ (here $\overline{\dimB}$ denotes the upper box dimension, cf.\ \cite[\S3]{Fal}).
Our Corollary \ref{cor:Hs0HsEqual}\rf{b3} sharpens this result, replacing $\overline{\dimB}$ with $\dimH$
(note that $\dimH(E)\leq \overline{\dimB}(E)$ for all bounded $E\subset\R^n$, cf.\ \cite[Proposition 3.4]{Fal})
and removing the boundedness assumption.
Caetano's main negative result \cite[Proposition 3.7]{Ca:00} says that if $0<s<n/2$, $\Omega$ is ``interior regular'', $\partial\Omega$ is a $d$-set (see \eqref{eq:dSet}) for some $d>n-2s$, then $H^s_0(\Omega)\subsetneqq H^s(\Omega)$. Here ``interior regular'' is a smoothness assumption that, in particular, excludes outward cusps in $\partial\Omega$. Precisely, it means \cite[Definition 3.2]{Ca:00} that there exists $C>0$ such that for all $\bx\in\partial\Omega$ and all cubes $Q$ centred at $\bx$ with side length $\leq 1$, $m(\Omega\cap Q)\geq C m(Q)$.
This result of Caetano's is similar to our Corollary \ref{cor:Hs0HsEqual}\rf{c3}, which, when combined with our Lemma \ref{lem:polarity}\rf{gg}, implies that if $0<s<n/2$ and $\tH^{-s}(\Omega)=H^{-s}_{\overline\Omega}$ (e.g.\ if $\Omega$ is $C^0$) with $\dimH\deO>n-2s$, then $H^s_0(\Omega)\subsetneqq H^s(\Omega)$.
In some respects our result is more general than \cite[Proposition 3.7]{Ca:00} because we allow cusp domains and we do not require a uniform Hausdorff dimension.
However, it is difficult to make a definitive comparison because we do not know of a characterisation of when $\tH^{-s}(\Omega)=H^{-s}_{\overline\Omega}$ for interior regular $\Omega$. %
Certainly, not every interior regular set whose boundary is a $d$-set belongs to the class of sets for which we can prove $\tH^{-s}(\Omega)=H^{-s}_{\overline\Omega}$; a concrete example is the Koch snowflake \cite[Figure~0.2]{Fal}.

\subsection{Some properties of the restriction operator \texorpdfstring{$|_\Omega: H^s(\R^n)\to H^s(\Omega)$}{}}
\label{subsec:restriction}
In \S\ref{subsec:3spaces} we have studied the relationship between the spaces $\tH^s(\Omega)$, $\cirHs(\Omega)$, and $H^s_{\overline\Omega}\subset H^s(\R^n)$, whose elements are distributions on $\R^n$, and in \S\ref{subsec:Hs0vsHs} the relationship between $H^s(\Omega)$ and $H^s_0(\Omega)$, whose elements are distributions on $\Omega$. To complete the picture we explore in this section the connections between these two types of spaces, which amounts to studying mapping properties of the restriction operator $|_\Omega:H^s(\R^n)\to H^s(\Omega)$.
These properties, contained in the following lemma, are rather straightforward consequences of the results obtained earlier in the paper and classical results such as \cite[Theorem 3.33]{McLean}, but for the sake of brevity we relegate the proofs to \cite{Hs0paper}.

\begin{lem}
\label{lem:RestrictionCollection}
Let $\Omega\subset\R^n$ be non-empty and open, and $s\in\R$.

\begin{enumerate}[(i)]
\item
$|_\Omega:H^s(\R^n)\to H^s(\Omega)$ is continuous with norm one;
\item
$|_\Omega:(H^s_{\Omega^c})^\perp\to H^s(\Omega)$ is a unitary isomorphism;%
\item If $\Omega$ is a finite union of disjoint Lipschitz open sets, $\partial \Omega$ is bounded, and $s> -1/2$, $s\not\in \{1/2,3/2,\ldots\}$, then $|_\Omega:\tH^s(\Omega) \to H^s_0(\Omega)$ is an isomorphism;
\item \label{b4}
$|_\Omega:H^s_{\overline \Omega}\to H^s(\Omega)$ is injective if and only if $\partial\Omega$ is $s$-null;
in particular,
\begin{itemize}
\item
$|_\Omega:H^s_{\overline \Omega}\to H^s(\Omega)$ is always injective for $s>n/2$ and never injective for $s<-n/2$;
\item if $\Omega$ is Lipschitz then $|_\Omega:\tH^s(\Omega)=H^s_{\overline \Omega}\to H^s(\Omega)$ is injective if and only if $s\ge -1/2$;
\item for every $-1/2\leq s_*\leq 0$ there exists a $C^0$ open set $\Omega$ for which $|_\Omega:\tH^s(\Omega)=H^s_{\overline \Omega}\to H^s(\Omega)$ is injective for all $s>s_*$ and not injective for all $s<s_*$;
\end{itemize}
\item \label{e4}
For $s\geq 0$, $|_\Omega:\cirHs(\Omega) \to H^s(\Omega)$ is injective; if $s\in\N_0$ then it is a unitary isomorphism onto its image in $H^s(\Omega)$;
\item \label{f4}
For $s\geq 0$, $|_\Omega:\tH^s(\Omega)\to H^s_0(\Omega)$ is injective and has dense image; if $s\in\N_0$ then it is a unitary isomorphism;
\item \label{a5}
$|_\Omega:\tH^s(\Omega)\to H^s(\Omega)$ is bijective if and only if $|_\Omega:\tH^{-s}(\Omega)\to H^{-s}(\Omega)$ is bijective;
\item \label{b5} $|_\Omega:\tH^{-s}(\Omega)\to H^{-s}(\Omega)$ is injective if and only if $|_\Omega:\tH^{s}(\Omega)\to H^{s}(\Omega)$ has dense image; i.e.\ if and only if $H^s_0(\Omega)=H^s(\Omega)$;
\item \label{bbb5} The following are equivalent:
\begin{itemize}
\item
$|_\Omega:\tH^s(\Omega)\to H^s_0(\Omega)$ is a unitary isomorphism;
\item
$\big\|\phi|_\Omega\big\|_{H^s(\Omega)} = \|\phi\|_{H^s(\R^d)}$ for all $\phi\in \scrD(\Omega)$;
\item
$\scrD(\Omega) \subset (H^s_{\Omega^c})^\perp$;
\end{itemize}
\item If $\Omega$ is bounded, or $\Omega^c$ is bounded with non-empty interior, then the three equivalent statements in \rf{bbb5} hold if and only if $s\in\N_0$;
\item
If the complement of $\Omega$ is $s$-null, then $|_\Omega:\tH^s(\Omega)\to H^s_0(\Omega)$ is a unitary isomorphism.
\end{enumerate}
\end{lem}

\begin{rem}\label{rem:LionsMagenes}
A space often used in applications is the Lions--Magenes space $H^s_{00}\OO$, defined as the interpolation space between $H^m_0\OO$ and $H^{m+1}_0\OO$, where $m\in\N_0$ and $m\le s<m+1$, see e.g.\ \cite[Chapter~1, Theorem~11.7]{LiMaI}
(the choice of interpolation method, e.g.\ the $K$-, the $J$- or the complex method, does not affect the result, as long it delivers a Hilbert space, see \cite[\S3.3]{InterpolationCWHM}).

Since $|_\Omega:\tH^m(\Omega)\to H^m_0(\Omega)$ is an isomorphism for all $m\in\N_0$ by Lemma \ref{lem:RestrictionCollection}\rf{f4} above, $H^s_{00}\OO$ is the image under the restriction operator of the space obtained from the interpolation of $\tH^m\OO$ and $\tH^{m+1}\OO$. %
Thus by \cite[Corollary~4.9]{InterpolationCWHM}, $H^s_{00}\OO$ is a subspace (not necessarily closed) of $H^s_0\OO$, for all $s\ge0$ and all open $\Omega$.

Furthermore, if $\Omega$ is Lipschitz and $\deO$ is bounded, \cite[Corollary~4.10]{InterpolationCWHM} ensures that $\{\tH^s\OO: s\in\R\}$ is an interpolation scale,
hence in this case we can characterise the Lions--Magenes space as $H^s_{00}\OO=\tH^s\OO|_\Omega$.
In particular, by \cite[Theorem 3.33]{McLean}, this implies that $H^s_{00}\OO=H^s_0\OO$ if $s\notin\{1/2,3/2,\ldots\}$.
This observation extends \cite[Chapter~1, Theorem~11.7]{LiMaI}, which was stated for $C^\infty$ bounded~$\Omega$.

That $H^{m+1/2}_{00}\OO\subsetneqq H^{m+1/2}_0\OO$ for $m\in\N_0$ was proved for all $C^\infty$ bounded $\Omega$ in \cite[Chapter~1, Theorem~11.7]{LiMaI}.
For general Lipschitz bounded $\Omega$, $H^{1/2}_{00}\OO\subsetneqq H^{1/2}_0\OO$ because the constant function $1$ belongs to the difference between the two spaces, as shown in \cite[p.~5]{NOS13}.
\end{rem}

\subsection{Sobolev spaces on sequences of subsets of \texorpdfstring{$\R^n$}{Rn}}
\label{subsec:Seqs+Eqs}

We showed in \S\ref{subsec:3spaces} that the Sobolev spaces $\tH^s(\Omega)$, $\cirHs(\Omega)$ (for $s\geq0$) and $H^s_{\overline\Omega}$ are in general distinct. These spaces arise naturally in the study of Fredholm integral equations and elliptic PDEs on rough (non-Lipschitz) open sets (a concrete example is the study of BIEs on screens, see \S\ref{sec:BIE} and \cite{CoercScreen}). When formulating such problems using a variational formulation, one must take care to choose the correct Sobolev space setting to ensure the physically correct solution.

Any arbitrarily ``rough'' open set $\Omega$ can be represented as a nested union of countably many ``smoother'' (e.g.\ Lipschitz) open sets $\{\Omega_j\}_{j=1}^\infty$ \cite[p.317]{Kellogg}. One can also consider closed sets $F$ that are nested intersections of a collection of closed sets $\{F_j\}_{j=1}^\infty$.
Significantly, many well-known fractal sets and sets with fractal boundary are constructed in this manner as a limit of prefractals.
We will apply the following propositions that consider such constructions to BIEs on sequences of prefractal sets in \S\ref{sec:BIE} below. Precisely, we will use these results together with those from \S\ref{subsec:ApproxVar} to deduce the correct fractal limit of the sequence of solutions to the prefractal problems, and the correct variational formulation and Sobolev space setting for the limiting solution.

\begin{prop}
\label{prop:nestedopen}
Suppose that $\Omega=\bigcup_{j=1}^\infty \Omega_j$, where $\{\Omega_j\}_{j=1}^\infty$ is a nested sequence of non-empty open subsets of $\R^n$ satisfying $\Omega_j\subset\Omega_{j+1}$ for $j=1,2,\ldots$.
Then $\Omega$ is open and
\begin{align}
\label{eq:HtildeUnion}
\tH^s(\Omega)
=\overline{\bigcup_{j=1}^\infty \tH^s(\Omega_j)}.
\end{align}
\end{prop}
\begin{proof}
We will show below that
\begin{align}
\label{eq:DUnion}
\scrD(\Omega) = \bigcup_{j=1}^\infty \scrD(\Omega_j).
\end{align}
Then \rf{eq:HtildeUnion} follows easily from \rf{eq:DUnion} because
\begin{align*}%
\tH^s(\Omega)
= \overline{\scrD(\Omega)}
= \overline{\bigcup_{j=1}^\infty \scrD(\Omega_j)}
= \overline{\bigcup_{j=1}^\infty \overline{\scrD(\Omega_j)}}
=\overline{\bigcup_{j=1}^\infty \tH^s(\Omega_j)}.
\end{align*}
To prove \rf{eq:DUnion}, we first note that the inclusion $\bigcup_{j=1}^\infty \scrD(\Omega_j)\subset  \scrD(\Omega)$ is obvious.
To show the reverse inclusion, let $\phi\in\scrD(\Omega)$.
We have to prove that $\phi\in\scrD(\Omega_j)$ for some $j\in \N$.
Denote $K$ the support of $\phi$; then $K$ is a compact subset of $\Omega$, thus $\{\Omega_j\}_{j=1}^\infty$ is an open cover of $K$.
As $K$ is compact there exists a finite subcover $\{\Omega_j\}_{j=j_1,\ldots,j_\ell}$. Thus $K\subset\Omega_{j_\ell}$ and $\phi\in \scrD(\Omega_{j_\ell})$.
\end{proof}

It is easy to see that the analogous result, with $\tH^s(\Omega)$ replaced by $\cirHs(\Omega)$ (with $s\geq 0$), or with $\tH^s(\Omega)$ replaced by $H^s_{\overline{\Omega}}$, does not hold in general. Indeed, as a counterexample we can take any $\Omega$ which is a union of nested $C^0$ open sets, but for which $\tH^s(\Omega)\neq \cirHs(\Omega)$. Then the above result and \eqref{eqn:inclusions} gives
\begin{align*}%
\overline{\bigcup_{j=1}^\infty \cirHs(\Omega_j)}
=\overline{\bigcup_{j=1}^\infty \tH^s(\Omega_j)}
=\tH^s(\Omega)
\subsetneqq \cirHs(\Omega) \subset H^s_{\overline{\Omega}}.
\end{align*}
A concrete example is $\Omega=(-1,0)\cup(0,1)\subset\R$ and $\Omega_j=(-1,-1/j)\cup(1/j,1)$, with $s>1/2$, for which $\tH^s(\Omega)\neq \cirHs(\Omega)=H^s_{\overline\Omega}$ by Lemma \ref{lem:CircleSpace}\rf{d1}, Lemma \ref{lem:equalityNullity}\rf{aaa} and Lemma \ref{lem:polarity}\rf{jj}.

The following is a related and obvious result.
\begin{prop}
\label{prop:nestedclosed}
Suppose that $F=\bigcap_{j\in\scrJ}  F_j$,
where $\scrJ$ is an index set and $\{F_j\}_{j\in\scrJ}$ is a collection of closed subsets of $\R^n$. Then $F$ is closed and
\begin{align*}%
H^s_F=\bigcap_{j\in\scrJ} H^s_{F_j}.
\end{align*}
\end{prop}

We will apply both the above results in \S\ref{sec:BIE} on BIEs. The following remark makes clear that Proposition \ref{prop:nestedopen} applies also to the FEM approximation of elliptic PDEs on domains with fractal boundaries.
\begin{rem}\label{rem:FEM}
Combining the abstract theory developed in \S\ref{subsec:ApproxVar} with Proposition~\ref{prop:nestedopen} allows us to prove the convergence of Galerkin methods on open sets with fractal boundaries.
In particular, we can easily identify which limit a sequence of Galerkin approximations converges to.
Precisely, let $\Omega=\bigcup_{j=1}^\infty \Omega_j$, where $(\Omega_j)_{j=1}^\infty$ is a  sequence of non-empty open subsets of $\R^n$ satisfying $\Omega_j\subset\Omega_{j+1}$ for $j\in\N$.
Fix $s\in\R$. For each $j\in\N$, define a sequence of nested closed spaces $V_{j,k}\subset V_{j,k+1}\subset \tH^s(\Omega_j)$, $k\in\N$, such that
$\tH^s(\Omega_j)=\overline{\bigcup_{k=1}^\infty V_{j,k}}$, and such that
 the sequences are a refinement of each other, i.e.\ $V_{j,k}\subset V_{j+1,k}$. Suppose that
$a(\cdot,\cdot)$ is a continuous and coercive sesquilinear form on some space $H$ satisfying $\tH^s\OO\subset H \subset H^s\Rn$. Then, for all $f\in H^{-s}\Rn$ the discrete and continuous variational problems: find $u_{V_{j,k}}\in V_{j,k}$ and $u_{\tH^s\OO}\in \tH^s\OO$ such that
\begin{align} \label{eq:vps}
a(u_{V_{j,k}},v)=\langle f,v\rangle_s\quad\forall v\in V_{j,k},\qquad
a(u_{\tH^s\OO},v')=\langle f,v'\rangle_s\quad\forall v'\in \tH^s\OO,
\end{align}
have exactly one solution,  and moreover the sequence
$(u_{V_{j,j}})_{j=1}^\infty$ converges to $u_{\tH^s\OO}$ in the $H^s\Rn$ norm, because the sequence $(V_{j,j})_{j=1}^\infty$ is dense in $\tH^s\OO$. (Here we use Proposition \ref{prop:nestedopen} and \eqref{eq:V-conver}.)

As a concrete example, take $\Omega\subset\R^2$ to be the Koch snowflake \cite[Figure~0.2]{Fal}, $\Omega_j$ the prefractal set of level $j$ (which is a Lipschitz polygon with $3\cdot 4^{j-1}$ sides),
$s=1$ and $a(u,v)=\int_{B_R}\nabla u\cdot \nabla \overline v\rd \bx$ the sesquilinear form for the Laplace equation, which is continuous and coercive on $\tH^s(B_R)$, where $B_R$ is any open ball containing $\overline\Omega$.
The $V_{j,k}$ spaces can be taken as nested sequences of standard finite element spaces defined on the polygonal prefractals.
Then the solutions $u_{V_{j,j}}\in V_{j,j}$ of the discrete variational problems, which are easily computable with a finite element code, converge in the $H^1(\R^2)$ norm to $u_{\tH^1\OO}$, the solution to the variational problem on the right hand side in \eqref{eq:vps}.
\end{rem}

\section{Boundary integral equations on fractal screens} \label{sec:BIE}

This section contains the paper's major application, which has motivated much of the earlier theoretical analysis. The problem we consider is itself motivated by the widespread use in telecommunications of electromagnetic antennas that are designed as good approximations to fractal sets. The idea of this form of antenna design, realised in many applications, is that the self-similar, multi-scale fractal structure leads naturally to good and uniform performance over a wide range of wavelengths, so that the antenna has effective wide band performance \cite[\S18.4]{Fal}. Many of the designs proposed take the form of thin planar devices that are approximations to bounded fractal subsets of the plane, for example the Sierpinski triangle \cite{PBaRomPouCar:98} and sets built using Cantor-set-type constructions \cite{SriRanKri:11}.
These and many other fractals sets $F$ are constructed by an iterative procedure: a sequence of ``regular'' closed sets $F_1\supset F_2 \supset \ldots$ (which we refer to as ``prefractals'') is constructed recursively, with the fractal set $F$ defined as the limit $F=\cap_{j=1}^\infty F_j$.  Of course,
practical engineered antennae are not true fractals but rather a prefractal approximation $F_j$ from the recursive sequence. So an interesting mathematical question of potential practical interest is: how does the radiated field from a prefractal antenna $F_j$ behave in the limit as $j\to\infty$ and $F_j\to F$?

We will not address this problem in this paper, which could be studied, at a particular radiating frequency, via a consideration of boundary value problems for the time harmonic Maxwell system in the exterior of the antenna, using for example the BIE formulation of \cite{BuCh:03}. Rather, we shall consider analogous time harmonic acoustic problems, modelled by boundary value problems for the Helmholtz equation. These problems can be considered as models of many of the issues and potential behaviours, and we will discuss, applying the results of \S\ref{subsec:ApproxVar} and other sections above, the limiting behaviour of sequences of solutions to BIEs, considering as illustrative examples two of several possible set-ups.

For the Dirichlet screen problem we will consider the limit $\Gamma_j\to F$ where the closed set $F=\cap_{j=1}^\infty \Gamma_j$ may be fractal and each $\Gamma_j$ is a regular Lipschitz screen. For the Neumann screen problem we will consider the limit $\Gamma_j\to \Gamma$ where the open set $\Gamma = \cup_{j=1}^\infty \Gamma_j$, and $\overline{\Gamma}\setminus \Gamma$ may be fractal. In the Dirichlet case we will see that the limiting solution may be non-zero even when $m(F)=0$ ($m$ here 2D Lebesgue measure), provided the fractal dimension of $F$ is $>1$. In the Neumann case we will see that in cases where $\Gamma^* := \inter(\overline{\Gamma})$ is a regular Lipschitz screen the limiting solution can differ from the solution for $\Gamma^*$ if   the fractal dimension of $\partial \Gamma$ is $>1$.

The set-up is as follows. For $\bx=(x_1,x_2,x_3)\in \R^3$ let $\tilde \bx=(x_1,x_2)$ and let $\Gamma_\infty = \{(\tilde \bx,0):\tilde \bx\in \R^2\}\subset \R^3$, which we identify with $\R^2$ in the obvious way. Let $\Gamma$ be a bounded open Lipschitz subset of $\Gamma_\infty$, choose $k\in \C$ (the {\em wavenumber}), with $k\neq 0$ and\footnote{Our assumption here that $k$ has a positive imaginary part corresponds physically to an assumption of some energy absorption in the medium of propagation. While making no essential difference to the issues we consider, a positive imaginary part for $k$ simplifies the mathematical formulation of our screen problems slightly.} $0< \arg(k)\leq \pi/2$, and consider the following Dirichlet and Neumann screen problems for the Helmholtz equation (our notation $W^1_2(\R^3)$ here is as defined in \S\ref{sec:intro}):
\begin{equation*} %
\begin{split}
 \mbox{Find }u\in C^2(\R^3\setminus\overline{\Gamma})\cap W^1_2(\R^3\setminus\overline{\Gamma})  \mbox{ such that }\Delta u + k^2u=0 \mbox{ in }\R^3\setminus\overline{\Gamma} \mbox{ and}\\
 u = f\in H^{1/2}(\Gamma) \mbox{ on }\Gamma \mbox{ (Dirichlet) or}\\
  \frac{\partial u}{\partial \bn} = g\in H^{-1/2}(\Gamma) \mbox{ on }\Gamma \mbox{ (Neumann)}.
\end{split}
\end{equation*}
Where $U_+ := \{\bx\in \R^3:x_3>0\}$ and $U_-:= \R^3\setminus \overline{U_+}$ are  the upper and lower half-spaces,  by $u=f$ on $\Gamma$ we mean precisely that $\gamma_\pm u|_\Gamma=f$, where $\gamma_\pm$ are the standard trace operators $\gamma_\pm:H^1(U_\pm)=W^1_2(U_\pm)\to H^{1/2}(\Gamma_\infty)$. Similarly, by $\partial u/\partial \bn=g$ on $\Gamma$ we mean precisely that $\partial_\bn^\pm u|_\Gamma=g$, where $\partial_\bn^\pm$ are the standard normal derivative operators $\partial_\bn^\pm: W^1_2(U_\pm;\Delta) \to H^{1/2}(\Gamma_\infty)$; here $W^1_2(U_\pm;\Delta)= \{u\in W^1_2(U_\pm):\Delta u\in L^2(U_\pm)\}$, and for definiteness we take the normal in the $x_3$-direction, so that $\partial u/\partial \bn = \partial u/\partial x_3$.

These screen problems are uniquely solvable: one standard proof of this is via BIE methods \cite{sauter-schwab11}. The following theorem, reformulating these screen problems as BIEs,  is standard (e.g.~\cite{sauter-schwab11}), dating back to \cite{stephan87} in the case when $\Gamma$ is $C^\infty$ (the result in \cite{stephan87} is for $k\geq 0$, but the argument is almost identical and slightly simpler for the case $\Im(k)>0$). The notation in this theorem is that $[u]:= \gamma_+u-\gamma_-u \in H^{1/2}_{\overline{\Gamma}}\subset H^{1/2}(\Gamma_\infty)$ and $[\partial_\bn u]:= \partial_\bn^+u-\partial_\bn^-u \in H^{-1/2}_{\overline{\Gamma}}\subset H^{-1/2}(\Gamma_\infty)$ (and recall that $H^{s}_{\overline{\Gamma}}=\tH^s(\Gamma)$, $s\in \R$, since $\Gamma$ is Lipschitz; see \cite[Theorem 3.29]{McLean} or Lemma \ref{lem:sob_equiv} above). Further, for every compactly supported $\phi\in H^{-1/2}(\Gamma_\infty)$, $\mathcal{S}\phi \in H^1(\R^3)=W^1_2(
\R^3)
$ denotes the standard acoustic single-layer
potential (e.g.~\cite{McLean,ChGrLaSp:11}), defined explicitly in the case that $\phi$ is continuous by
\begin{equation*}%
\mathcal{S}\phi(\bx) = \int_{\Gamma_\infty} \Phi(\bx,\by) \phi(\by) \,\rd s(\by), \quad \bx\in \R^3,
\end{equation*}
where $\Phi(\bx,\by) := \exp(\ri k |\bx-\by|)/(4\pi |\bx-\by|)$ is the fundamental solution for the Helmholtz equation. Similarly \cite{McLean,ChGrLaSp:11}, for compactly supported $\psi\in H^{1/2}(\Gamma_\infty)$, $\mathcal{D}\psi \in W^1_2(\R^3\setminus \supp{\psi})$ is the standard acoustic double-layer potential, defined by
\begin{equation*}%
\mathcal{D}\psi(\bx) = \int_{\Gamma_\infty} \frac{\partial \Phi(\bx,\by)}{\partial \bn(\by)} \psi(\by) \,\rd s(\by), \quad \bx\in \R^3\setminus \supp{\psi}.
\end{equation*}

\begin{thm} [E.g., \cite{stephan87,sauter-schwab11}.] \label{thm:bie}
If $u$ satisfies the Dirichlet screen problem then
$$
u(\bx) = -\mathcal{S}[\partial_\bn u ](\bx), \quad \bx\in \R^3\setminus \overline{\Gamma},
$$
and $[\partial_\bn u]\in \tH^{-1/2}(\Gamma)$ is the unique solution of
\begin{equation} \label{eq:single}
S_\Gamma [\partial_\bn u] = f,
\end{equation}
where the isomorphism $S_\Gamma:\tH^{-1/2}(\Gamma)\to H^{1/2}(\Gamma)$ is the standard acoustic single-layer boundary integral operator, defined by
$$
S_\Gamma \phi:= \gamma_\pm \mathcal{S} \phi\big|_\Gamma, \quad \phi \in \tH^{-1/2}(\Gamma).
$$
Similarly, if $u$ satisfies the Neumann screen problem then
$$
u(\bx) = \mathcal{D}[u](\bx), \quad \bx\in \R^3\setminus \overline{\Gamma},
$$
and $[u]\in \tH^{1/2}(\Gamma)$ is the unique solution of
\begin{equation} \label{eq:hyp}
T_\Gamma [u] = -g,
\end{equation}
where the isomorphism $T_\Gamma:\tH^{1/2}(\Gamma)\to H^{-1/2}(\Gamma)$ is the standard acoustic hypersingular integral operator, defined by
$$
T_\Gamma \phi:= \partial_\bn^\pm \mathcal{D} \phi\big|_\Gamma, \quad \phi \in \tH^{1/2}(\Gamma).
$$
\end{thm}

The standard analysis of the above BIEs, in particular the proof that $S_\Gamma$ and $T_\Gamma$ are isomorphisms, progresses via a variational formulation. Recalling from Theorem \ref{thm:DualityTheorem} that $H^{-s}(\Gamma)$ is (a natural unitary realisation of) the dual space of $\tH^s(\Gamma)$, we define sesquilinear
forms $a_{\rm D}$ on $\tH^{-1/2}(\Gamma)$  and $a_{\rm N}$ on $\tH^{1/2}(\Gamma)$ by
\begin{align*}
\label{}
a_{\rm D}(\phi,\psi) &= \langle S_\Gamma \phi,\psi \rangle, \quad \phi,\psi\in \tH^{-1/2}(\Gamma),\\
a_{\rm N}(\phi,\psi) &= \langle T_\Gamma \phi,\psi \rangle, \quad \phi,\psi\in \tH^{1/2}(\Gamma),
\end{align*}
where in each equation $\langle.,.\rangle$ is the appropriate duality pairing. Equation \eqref{eq:single} is equivalent to the variational formulation: find $[\partial_\bn u]\in \tH^{-1/2}(\Gamma)$ such that
\begin{equation} \label{eq:wfD}
a_{\rm D}([\partial_\bn u],\phi) = \langle f,\phi\rangle, \quad \phi \in \tH^{-1/2}(\Gamma).
\end{equation}
Similarly \eqref{eq:hyp} is equivalent to: find $[u]\in \tH^{1/2}(\Gamma)$ such that
\begin{equation} \label{eq:wfN}
a_{\rm N}([u],\psi) = -\langle g,\psi\rangle, \quad \psi \in \tH^{1/2}(\Gamma).
\end{equation}
These sesquilinear forms (see \cite{stephan87,Ha-Du:92,Co:04})
are continuous and coercive in the sense of \eqref{eq:defcoer}. It follows from the Lax--Milgram theorem
that \eqref{eq:wfD} and \eqref{eq:wfN} (and so also \eqref{eq:single} and \eqref{eq:hyp}) are uniquely solvable.
\begin{rem} \label{rem:finite_union}
It is not difficult to show (see \cite{CoercScreen,ScreenPaper} for details) that Theorem \ref{thm:bie} holds, and the Dirichlet and Neumann screen problems are uniquely solvable, for a rather larger class of open sets than the open Lipschitz sets. Precisely, the Dirichlet problem is uniquely solvable, and Theorem \ref{thm:bie} holds for the Dirichlet problem, if and only if $\partial\Gamma$ is $1/2$-null (as defined in \S\ref{subsec:Polarity}) and $\tH^{-1/2}(\Gamma)=H^{-1/2}_{\overline \Gamma}$. In particular, by Lemma \ref{lem:polarity}(xvii), (v) and Theorem \ref{thm:new2}, and relevant to our discussion of prefractals below, these conditions hold in the case that $\Gamma= \Gamma_1 \cup \ldots \cup \Gamma_M$ is a finite union of bounded $C^0$ open sets, $\Gamma_1$, \ldots, $\Gamma_M$, with $\overline{\Gamma_i}\cap \overline{\Gamma_j}$ a finite set for $1\leq i,j\leq M$. Similarly, the Neumann problem is uniquely solvable, and Theorem \ref{thm:bie} holds for the Neumann problem, if and only if $\partial\
Gamma$ is $(-1/2)$-null and $\tH^{1/2}(\Gamma)=H^{1/2}_{\overline \Gamma}$; in particular, by Lemma \ref{lem:polarity}(xix), (v) and Theorem \ref{thm:new2}, these conditions hold in the case that $\Gamma= \Gamma_1 \cup \ldots \cup \Gamma_M$ is a finite union of bounded Lipschitz open sets, $\Gamma_1$, \ldots, $\Gamma_M$, with $\overline{\Gamma_i}\cap \overline{\Gamma_j}$ finite for $1\leq i,j\leq M$.
\end{rem}

Domain-based variational formulations of screen problems are also standard. In particular, an equivalent formulation of the Dirichlet problem is to find $u\in H^1(\R^3)= W_2^1(\R^3)$ such that $\gamma_\pm u = f$ on $\Gamma$ and such that
\begin{equation} \label{eq:dom_weak}
a_{\mathrm{dom}}(u,\psi) := \int_{\R^3} (\nabla u \cdot \nabla \bar v - k^2 u \bar v) \,\rd \bx = 0, \quad \forall v\in H^1_0(\R^3\setminus \overline{\Gamma}),%
\end{equation}
with $a_{\mathrm{dom}}(\cdot,\cdot)$ continuous and coercive on $H_0^1(\R^3\setminus\overline{\Gamma})$, so that this formulation is also uniquely solvable by the Lax--Milgram lemma.
In the case that $\Re(k)=0$, so that $k^2<0$, $a_{\mathrm{dom}}(\cdot,\cdot)$ is also Hermitian, and the solution to this variational problem is also the unique solution to the minimisation problem: find $u\in H^1(\R^3)$ that minimises $a_{\mathrm{dom}}(u,u)$ subject to the constraint $\gamma_\pm u = f$.

This leads to a connection to certain set capacities from potential theory.
For an open set $\Omega\subset \R^n$ and $s>0$ we define the capacity
\begin{align*}
\label{}
\mathrm{cap}_{s,\R^n}(\Omega):=\sup_{\substack{K\subset \Omega\\ K \textrm{ compact}}}
\inf \big\{\|u \|^2_{H^{s}(\R^n)}
\big\},
\end{align*}
where the infimum is over all $u\in \scrD(\R^n)$ such that $u\geq 1$ in a neighbourhood of $K$.
Then, in the special case when $k=\ri$ (so that $a_\mathrm{dom}(u,u) = \|u\|^2_{H^1(\R^3)}$ for $u\in H^1(\R^3)$) and $f=1$, the solution $u$ of the above minimisation problem
satisfies
(viewing $\Gamma$ as a subset of $\R^3$)
\begin{equation} \label{eq:capbie}
\mathrm{cap}_{1,\R^3}(\Gamma) = a_{\mathrm{dom}}(u,u) = a_D([\partial_n u],[\partial_n u])= \langle 1, [\partial_n u]\rangle,
\end{equation}
where $[\partial_n u]\in H^{-1/2}(\Gamma)$ is the unique solution of \eqref{eq:wfD} and $u= -\mathcal{S}[\partial_n u]$ is the unique solution of \eqref{eq:dom_weak}. Note that in \rf{eq:capbie} the first equality follows from standard results on capacities (see, e.g., \cite[Proposition 3.4, Remark 3.14]{HewMoi:15}), the third from \eqref{eq:wfD}, and the second equality follows because $a_D(\phi,\phi) = a_{\mathrm{dom}}(\mathcal{S}\phi,\mathcal{S}\phi)$, for all $\phi\in \tH^{-1/2}(\Gamma)$ (cf.\ the proof of \cite[Theorem 2]{Costabel88}).

We are interested in sequences of screen problems, with a sequence of screens $\Gamma_1, \Gamma_2, \ldots$ converging in some sense to a limiting screen. We assume that there exists $R>0$ such that the open set $\Gamma_j\subset \Gamma^R := \{\bx\in \Gamma_\infty: |\bx|<R\}$ for every $j\in \N$. Let $a_\mathrm{D}^R$ and $a_\mathrm{N}^R$ denote the sesquilinear forms $a_\mathrm{D}$ and $a_\mathrm{N}$ when $\Gamma=\Gamma^R$. We note that for any $R>0$ and open $\Gamma\subset \Gamma^R$ it holds that
\begin{equation*} %
S_\Gamma \phi = \left.\left(S_{\Gamma^R} \phi\right)\right|_\Gamma \quad \mbox{ and } \quad T_\Gamma \psi = \left.\left(T_{\Gamma^R} \psi\right)\right|_\Gamma, %
\end{equation*}
for $\phi\in \tH^{-1/2}(\Gamma)$ and $\psi\in \tH^{1/2}(\Gamma)$.
Hence
\begin{equation} \label{eq:rest2}
a_\mathrm{D}(\phi,\psi) = a_\mathrm{D}^R(\phi,\psi), \quad \phi,\psi\in \tH^{-1/2}(\Gamma)\subset \tH^{-1/2}(\Gamma^R),
\end{equation}
 i.e. $a_\mathrm{D}$ is the restriction of the sesquilinear form $a_\mathrm{D}^R$ from $\tH^{-1/2}(\Gamma^R)$ to its closed subspace $\tH^{-1/2}(\Gamma)$. Similarly, $a_\mathrm{N}$ is the restriction of $a_\mathrm{N}^R$ to $\tH^{1/2}(\Gamma)$.

 Focussing first on the Dirichlet problem, consider a sequence of Lipschitz screens $\Gamma_1, \Gamma_2, \ldots$ with $\Gamma_1 \supset \Gamma_2 \supset \ldots$ (or equivalently $\overline{\Gamma_1} \supset \overline{\Gamma_2} \supset \ldots$). Suppose that $f_j\in H^{1/2}(\Gamma_j)$ and let $\phi_j$ denote the solution $[\partial_\bn u]$ to \eqref{eq:wfD} (equivalently to \eqref{eq:single}) when $\Gamma=\Gamma_j$ and $f=f_j$. The question we address is what can be said about $\phi_j$ in the limit as $j\to\infty$. For this question to be meaningful, we need some control over the sequence $f_j$: a natural assumption, relevant to many applications, is that
\begin{equation} \label{eq:assum}
\mbox{there exists }f_\infty\in H^{1/2}(\Gamma_\infty) \;\mbox{ such that } f_j = f_\infty|_{\Gamma_j}, \; \mbox{for } j\in \N.
\end{equation}
We shall study the limiting behaviour under this assumption using the general theory of \S\ref{subsec:ApproxVar}.

To this end choose $R>0$ so that $\Gamma_1\subset \Gamma^R$, let $H = \tH^{-1/2}(\Gamma^R)$, $W_j = \tH^{-1/2}(\Gamma_j)$, so that $H\supset W_1\supset W_2 \supset \ldots$, and set
$$
W = \bigcap_{j=1}^\infty W_j = \bigcap_{j=1}^\infty H^{-1/2}_{\overline{\Gamma_j}}  = \bigcap_{j=1}^\infty \tH^{-1/2}(\Gamma_j).
$$
Then, by Proposition \ref{prop:nestedclosed}, $W = H^{-1/2}_F$, where $F = \cap_{j=1}^\infty \overline{\Gamma_j}$. Further, by \eqref{eq:rest2}, and where $f= f_\infty|_{\Gamma^R}$, we see that $\phi_j$ is the solution of
$$
a_\mathrm{D}^R(\phi_j,\psi) = \langle f,\psi\rangle, \quad \psi\in W_j.
$$
Applying Lemma \ref{lem:dec} we obtain immediately the first part of the following result. The remainder of the theorem follows from Lemma \ref{lem:polarity}(\ref{hh}) and (\ref{gg}).
\begin{thm} \label{thm:dec}
Assuming \eqref{eq:assum},  $\|\phi_j-\phi\|_{H^{-1/2}(\Gamma_\infty)}=\|\phi_j-\phi\|_{\tH^{-1/2}(\Gamma^R)}\to 0$ as $j\to\infty$, where $\phi\in H^{-1/2}_F$ is the unique solution of
$$
a_\mathrm{D}^R(\phi,\psi) = \langle f,\psi\rangle, \quad \psi\in H^{-1/2}_F.
$$
Further, if $F$ is $(-1/2)$-null (which holds in particular if $\dimH(F) < 1$) then $\phi=0$. If $F$ is not $(-1/2)$-null (which holds in particular if $\dimH(F) > 1$), then there exists $f_\infty\in H^{1/2}(\Gamma_\infty)$ such that $\langle f,\psi\rangle\neq 0$, for some $\psi\in H^{-1/2}_F$, in which case $\phi\neq 0$.
\end{thm}

\begin{example}
Theorem \ref{thm:dec} applies in particular to cases in which $F$ is a fractal set. One such example is where
\begin{equation*}%
\overline{\Gamma_j} = \big\{(\tilde \bx, 0):\tilde \bx\in E_{j-1}^2\big\},
\end{equation*}
and $\Gamma_j = \mathrm{int}(\overline{\Gamma_j})$,
with (cf.\ \cite[Example 4.5]{Fal}) $E_0\supset E_1\supset\ldots$ the standard recursive sequence generating the one-dimensional ``middle-$\lambda$'' Cantor set, $0<\lambda<1$, so that $E_j^2\subset \R^2$ is the closure of a Lipschitz open set that is the union of $4^j$ squares of side-length $l_j=\alpha^j$, where $\alpha=(1-\lambda)/2\in (0,1/2)$. (Figure \ref{fig:CantorDust} visualises $E_0^2,\ldots, E_4^2$ in the classical ``middle third'' case $\alpha =\lambda = 1/3$.) In this case the limit set is %
\begin{equation*}%
F = \big\{(\tilde \bx, 0):\tilde \bx\in E^2\big\},
\end{equation*}
where $E=\cap_{j=0}^\infty E_j$ is the middle-$\lambda$ Cantor set and $E^2$ is the associated two-dimensional Cantor set (or ``Cantor dust''), which has Hausdorff dimension
$\dimH(E^2) = 2\log 2/\log(1/\alpha) \in (0,2)$. It is known that $E^2$ is $s$-null if and only if $s\geq (\dimH (E^2)-n)/2$ (see \cite[Theorem 4.5]{HewMoi:15}, where $E^2$ is denoted $F^{(2)}_{2\log 2/\log(1/\alpha),\infty}$).
Theorem \ref{thm:dec} applied to this example shows that if $1/4<\alpha < 1/2$ then there exists $f_\infty\in H^{1/2}(\Gamma_\infty)$ such that the limiting solution $\phi\in H^{-1/2}_F$ to the sequence of screen problems is non-zero.
On the other hand, if $0<\alpha\leq 1/4$ then the theorem tells us that the limiting solution $\phi=0$.
\end{example}

\begin{figure}

{\hspace{-12mm}\includegraphics[scale=0.75]{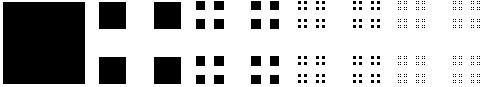}}
\caption{The first five terms in the recursive sequence of prefractals converging to the standard two-dimensional middle-third Cantor set (or Cantor dust).}%
\label{fig:CantorDust}
\end{figure}

It is clear from Theorem \ref{thm:dec} that whether or not the solution to the limiting sequence of screen problems is zero depends not on whether the limiting set $F$, thought of as a subset of $\Gamma_\infty$ which we identify with $\R^2$, has Lebesgue measure zero, but rather on whether this set $F$ is $(-1/2)$-null. %
From a physical perspective this may seem surprising: thinking of the screen as having a certain mass per unit area, a screen with zero surface Lebesgue measure is a screen with zero mass, in some sense a screen that is not there! But to those familiar with potential theory (e.g., \cite{AdHe}) this will be less surprising. In particular from \eqref{eq:capbie}, in the case $k=\ri$ and choosing $f_\infty$ so that $f_\infty=1$ in a neighbourhood of $\Gamma^R$, it holds that
$$
\mathrm{cap}_{1,\R^3}(\Gamma_j) = \langle 1, \phi_j\rangle.
$$
Taking the limit as $j\to\infty$, and applying elementary capacity theoretic arguments (see, e.g.,  \cite[Proposition 3.4]{HewMoi:15}), it follows that
$$
\mathrm{cap}_{1,\R^3}(F) = \langle 1, \phi\rangle.
$$
Moreover, for $\widetilde G \subset \R^2$, defining $G = \{(x_1,x_2,0)\in \R^3: (x_1,x_2)\in \widetilde G\}$, it is clear from the definition of capacity (which involves smooth functions only) and standard Sobolev trace and extension theorems (e.g.\ \cite{McLean}) that, for some positive constants $c_1$ and $c_2$ independent of $\tilde G$,
\begin{align}
\label{}
c_1  \mathrm{cap}_{1,\R^3}(G) \leq \mathrm{cap}_{1/2,\R^2}(\widetilde G) \leq c_2 \mathrm{cap}_{1,\R^3}(G).
\end{align}
Thus, where $\widetilde F = \{(x_1,x_2)\subset \R^2: (x_1,x_2,0)\in F\}$, it is clear that $\phi=0$ iff $\mathrm{cap}_{1,\R^3}(F) =0$ iff $\mathrm{cap}_{1/2,\R^2}(\widetilde F)=0$, i.e. iff $\widetilde F$ is $(-1/2)$-null as a subset of $\R^2$, where the latter equivalence follows from \cite[13.2.2]{Maz'ya} (restated in \cite[Theorem 2.5]{HewMoi:15}).

Turning now to the Neumann problem, consider a sequence of open screens $\Gamma_1, \Gamma_2,\ldots$, with $\Gamma_1 \subset \Gamma_2 \subset\ldots$, such that: (i) $\Gamma := \bigcup_{j=1}^\infty \Gamma_j$ is bounded; and (ii) each $\Gamma_j$ is either Lipschitz or is a finite union of Lipschitz open sets whose closures intersect in at most a finite number of points (the case discussed in Remark \ref{rem:finite_union}, which ensures, inter alia, that $\tH^{1/2}(\Gamma_j) = H^{1/2}_{\overline{\Gamma_j}}$). Suppose that $g_j\in H^{-1/2}(\Gamma_j)$ and let $\phi_j\in V_j:= \tH^{1/2}(\Gamma_j)= H^{1/2}_{\overline{\Gamma_j}}$ denote the solution $[u]$ to \eqref{eq:wfN} (equivalently to \eqref{eq:hyp}) when $\Gamma=\Gamma_j$ and $g=g_j$. Analogously to the Dirichlet case we assume that
\begin{equation} \label{eq:assum2}
\mbox{there exists }g_\infty\in H^{-1/2}(\Gamma_\infty) \;\mbox{ such that } g_j = g_\infty|_{\Gamma_j}, \; \mbox{for } j\in \N,
\end{equation}
and choose $R>0$ such that $\Gamma \subset \Gamma^R$. Then, as noted after \eqref{eq:rest2}, and where $g= g_\infty|_{\Gamma^R}$, we see that $\phi_j\in V_j\subset \tH^{1/2}(\Gamma^R)$ is the solution of
$$
a_\mathrm{N}^R(\phi_j,\psi) = \langle g,\psi\rangle, \quad \psi\in V_j.
$$
By Proposition \ref{prop:nestedopen},
$
V := \overline{\bigcup_{j\in\N} V_j} = \tH^{1/2}(\Gamma).
$
The first sentence of the following proposition is immediate from \eqref{eq:V-conver}, and the second sentence is clear.
\begin{prop} \label{thm:neu}
In the case that \eqref{eq:assum2} holds,  $\|\phi_j-\phi\|_{H^{1/2}(\Gamma_\infty)}=\|\phi_j-\phi\|_{\tH^{1/2}(\Gamma^R)}=\|\phi_j-\phi\|_{\tH^{1/2}(\Gamma)}\to 0$ as $j\to\infty$, where $\phi\in \tH^{1/2}(\Gamma)$ is the unique solution of
$$
a_\mathrm{N}^R(\phi,\psi) = \langle g,\psi\rangle, \quad \psi\in \tH^{1/2}(\Gamma).
$$
Further, if $\tH^{1/2}(\Gamma) \neq H^{1/2}_{\overline{\Gamma}}$, then there exists $g_\infty\in H^{-1/2}(\Gamma_\infty)$ such that $\phi\neq \phi^*$, where $\phi^*\in H^{1/2}_{\overline{\Gamma}}$ is the unique solution of
$$
a_\mathrm{N}^R(\phi^*,\psi) = \langle g,\psi\rangle, \quad \psi\in H^{1/2}_{\overline{\Gamma}}.
$$
\end{prop}
\begin{rem}
\label{rem:BIE}
The question: ``for which $s\in \R$ and open $\Omega\subset \R^n$ is $\tH^{s}(\Omega) \neq H^{s}_{\overline{\Omega}}$'' was addressed in \S\ref{subsec:3spaces}. From Lemma \ref{lem:equalityNullity} we have, in particular, that if $G:=\inter(\overline{\Gamma})\setminus \Gamma$ is not $-1/2$-null then $\tH^{1/2}(\Gamma) \subsetneqq H^{1/2}_{\overline{\Gamma}}$. Indeed, by Lemma \ref{lem:equalityNullity}(v), $\tH^{1/2}(\Gamma) = H^{1/2}_{\overline{\Gamma}}$ if and only if $G$ is $-1/2$-null, if it holds that $\tH^{1/2}(\inter(\overline{\Gamma})) = H^{1/2}_{\overline{\Gamma}}$, in particular if $\inter(\overline{\Gamma})$ is $C^0$. And, by Lemma \ref{lem:polarity}(xii) and (xiii), $G$ is $-1/2$-null if $\dim_H(G)<1$, while $G$ is not $-1/2$-null if $dim_H(G)>1$.

As a specific example, consider the sequence of closed sets $F_0\supset F_1\supset \ldots$ that are the prefractal approximations to the Sierpinski triangle  $F:= \bigcap_{j=0}^\infty F$ \cite[Example 9.4]{Fal}. $F_0$ is a (closed) triangle and $F_j$ is the union of $3^j$ closed triangles; the first four sets $F_0$, \ldots, $F_3$ in this sequence are shown in Figure \ref{fig:TSExamples}(a). For $j\in \N$ let $\Gamma_j := F_0\setminus F_j$, and let $\Gamma := \bigcup_{j\in\N} \Gamma_j$, so that $\overline{\Gamma} = F_0$ and $\partial \Gamma = \overline{\Gamma}\setminus \Gamma = F$. Then, using standard results on fractal dimension (e.g., \cite{Fal}), $\dim_H(\partial F_0) = 1$ while $\dim_H(F) = \log 3/\log 2$, so that also $\dim_H(\inter(\overline{\Gamma})\setminus \Gamma) = \dim_H(F\setminus \partial F_0) = \log 3/\log 2 >1$, which implies that $\tH^{1/2}(\Gamma) \subsetneqq H^{1/2}_{\overline{\Gamma}}$.
On the other hand, since $\Gamma^*:=\inter(\overline{\Gamma})$ is $C^0$, $H^{1/2}_{\overline{\Gamma}} = \tH^{1/2}(\Gamma^*)$, and $\phi^*\in \tH^{1/2}(\Gamma^*)$ (defined in Proposition~\ref{thm:neu}) is the solution $[u]$ to \eqref{eq:hyp} in the case when the screen is $\Gamma^*$ and $g$ in \eqref{eq:hyp} is the restriction of $g_\infty$ to $\Gamma^*$.

This specific example illustrates that the limit of the solutions $\phi_j\in \tH^{1/2}(\Gamma_j)$ to the BIE for the Neumann problem when the screen is $\Gamma_j$ can be different to the solution $\phi^* \in \tH^{1/2}(\Gamma^*)$ when the screen is $\Gamma^*$. It is surprising that this happens even though $\Gamma_j\to \Gamma^*$ in a number of senses. In particular, $\Gamma_j$ can be viewed as the screen $\Gamma^*$ with ``holes'' in it, but with the size of these holes, as measured by the 2D Lebesgue measure $m(\Gamma^*\setminus \Gamma_j)$, tending to $0$ as $j\to\infty$.
\end{rem}

\end{document}